\documentclass[a4paper]{article}

\usepackage{amsmath,amssymb,color}

 \usepackage{graphicx}
 \usepackage{subfigure}
\DeclareGraphicsExtensions{.pdf,.eps}

\newcommand{\bu}{\boldsymbol u}
\newcommand{\Om}{\Omega}

\newcommand{\bv}{\boldsymbol v}

\newcommand{\bw}{\boldsymbol w}

\newcommand{\btau}{\boldsymbol \tau}

\newcommand{\bpsi}{\boldsymbol \psi}

\newcommand{\be}{\boldsymbol e}

\newcommand{\bvar}{\boldsymbol \varphi}

\newcommand{\bs}{\boldsymbol s}

\newcommand{\bff}{\boldsymbol f}

\newtheorem{Theorem}{Theorem}[section]
\newtheorem{lema}[Theorem]{Lemma}

\newtheorem{remark}[Theorem]{Remark}

\newtheorem{Proof}{{\em Proof:}}

\newenvironment{proof}{\begin{Proof}\rm}{\hfill $\Box$ \end{Proof}}
\usepackage{hyperref}

\usepackage{macros}

\title{Error analysis of proper orthogonal decomposition data assimilation schemes for the Navier-Stokes equations}

\author{ Bosco
Garc\'{i}a-Archilla\thanks{Departamento de Matem\'atica Aplicada
II, Universidad de Sevilla, Spain. Research is supported by
Spanish MCINYU under grant PGC2018-096265-B-I00 (bosco@esi.us.es).}
  \and Julia Novo\thanks{Departamento de
Matem\'aticas, Universidad Aut\'onoma de Madrid, Spain.  Research is supported
by Spanish MINECO
under grant MTM2016-78995-P (AEI/FEDER, UE) (julia.novo@uam.es).}
\and Samuele Rubino\thanks{Department EDAN \&IMUS, Universidad de Sevilla, Spain. Research is supported by Spanish MCINYU under grant RTI2018-093521-B-C31 and Spanish State Research Agency through the national programme Juan de la Cierva-Incorporaci\'on 2017. (samuele@us.es).}}

\begin{document}

\maketitle

\begin{abstract}
The error analysis of a proper orthogonal decomposition (POD) data assimilation (DA) scheme for the Navier-Stokes equations is carried out. A grad-div stabilization term is added to the formulation of the POD method. Error bounds with constants independent on inverse powers of the viscosity parameter are derived for the POD algorithm.
No upper bounds in the nudging parameter of the data assimilation method are required.
Numerical experiments show that, for large values of the nudging parameter, the proposed method rapidly converges to the real solution, and greatly improves the overall accuracy of standard POD schemes up to low viscosities over predictive time intervals.
\end{abstract}

76B75. 

\section{Introduction}

Reduced order models (ROM) are a fairly extensive technique applied in many different fields to reduce the computational cost of direct
numerical simulations while keeping enough accurate numerical approximations. Proper Orthogonal Decomposition (POD) method provides the  elements (modes) of the reduced basis from a given database (snapshots) which are computed  by means of a direct or full order method.

Data assimilation refers to a class of techniques that combine experimental data and simulations in order to obtain better predictions in a physical system.
There is a vast literature on data assimilation methods
(see e.g., \cite{Asch_et_al_2016}, \cite{Daley_1993}, \cite{Kalnay_2003}, \cite{Law_Stuart_Zygalakis}, \cite{Reich_Cotter_2015}, and the references
therein). One of these techniques is  nudging in which a penalty term is added with the aim of driving the approximate solution towards coarse mesh observations of the data. In \cite{Az_Ol_Ti}, a new approach, known as continuous data assimilation,  is introduced for a large class of dissipative partial differential equations.

In this paper we study the numerical approximation of the Navier-Stokes equations with a continuous data assimilation method defined over a reduced order space. The basis functions in the ROM are based only on velocity approximations at different times computed with a mixed finite element Galerkin method using inf-sup stable elements. Both the snapshots and the basis of the ROM satisfy a discrete divergence-free condition.

We consider the Navier-Stokes equations (NSE)
\begin{align}
\label{NS} \partial_t\bu -\nu \Delta \bu + (\bu\cdot\nabla)\bu + \nabla p &= \bff &&\text{in }\ (0,T]\times\Omega,\nonumber\\
\nabla \cdot \bu &=0&&\text{in }\ (0,T]\times\Omega,
\end{align}
in a bounded domain $\Omega \subset {\mathbb R}^d$, $d \in \{2,3\}$ with initial condition $\bu(0)=\bu^0$. In~\eqref{NS},
$\bu$ is the velocity field, $p$ the kinematic pressure, $\nu>0$ the kinematic viscosity coefficient,
 and $\bff$ represents the accelerations due to external body forces acting
on the fluid. The Navier-Stokes equations \eqref{NS} must be complemented with boundary conditions. For simplicity,
we only consider homogeneous
Dirichlet boundary conditions $\bu = \boldsymbol 0$ on $\partial \Omega$.

As in \cite{Mondaini_Titi} we consider given coarse spatial mesh measurements, correspon\-ding to a solution $\bu$ of \eqref{NS}, observed at a coarse spatial mesh. We assume that the measurements are continuous in time and error-free and we denote by $I_H(\bu)$ the operator used for interpolating these
measurements, where $H$ denotes the resolution of the coarse spatial mesh.
Since no initial condition for $\bu$ is available one cannot  simulate equation \eqref{NS} directly.
To overcome this difficulty it was suggested in~\cite{Az_Ol_Ti} to consider instead a solution~$\bv$ of the following system
\begin{eqnarray}\label{eq:mod_NS}
 \partial_t\bv -\nu \Delta \bv + (\bv\cdot\nabla)\bv + \nabla \tilde p&=&\bff -\beta(I_H(\bv)-I_H(\bu)),\ \text{in }\ (0,T]\times\Omega,\nonumber\\
\nabla \cdot \bv&=&0, \ \text{in }\ (0,T]\times\Omega,
\end{eqnarray}
where $\beta$ is the nudging parameter.
In \cite{Mondaini_Titi} a semidiscrete postprocessed Galerkin spectral method in considered and analyzed. A fully discrete method for the
spatial discretization in \cite{Mondaini_Titi} is analyzed in \cite{Ibdah_et_al}.
In \cite{bosco_titi_yo} the continuous data assimilation algorithm is analyzed considering both a finite element Galerkin method and a Galerkin method with grad-div stabilization. The extension to the fully discrete case is carried out in \cite{bosco_y_yo}. For the Galerkin method with grad-div stabilization the constants in the error bounds in \cite{bosco_y_yo} and \cite{bosco_titi_yo} are independent on inverse powers of the viscosity parameter. In \cite{larios_et_al} the authors consider also fully discrete approximations to \eqref{eq:mod_NS} in which for the spatial discretization the Galerkin method with grad-div stabilization is considered. However, the constants in the error bounds in \cite{larios_et_al} are not independent on inverse powers of $\nu$. Moreover, in \cite{bosco_y_yo}, \cite{bosco_titi_yo} there is no need to impose an upper bound on the nudging parameter $\beta$ as required in \cite{Ibdah_et_al}, \cite{larios_et_al}, \cite{Mondaini_Titi}. This fact is important because, on the one hand, there is numerical evidence that no upper bound is required in the numerical experiments and, on the other hand, better results are obtained in some experiments for values of $\beta$ above the upper bound assumed in references \cite{Ibdah_et_al}, \cite{larios_et_al}, \cite{Mondaini_Titi}.

In \cite{zerfas_et_al} a continuous data assimilation reduced order model (DA-ROM) me\-thod is introduced and analyzed. The idea is to consider a Galerkin approximation to \eqref{eq:mod_NS} defined in a ROM space. The ROM space is based on a set of snapshots that are fully discrete Galerkin inf-sup stable mixed finite element approximations to \eqref{NS} at different time steps. The DA-ROM method in \cite{zerfas_et_al} is a Galerkin method without any kind of stabilization. The implicit Euler method is used as time integrator and error bounds are proved that converge exponentially fast in time to the true solution. The constants in the error bounds in \cite{zerfas_et_al} depend on inverse powers of the viscosity parameter.

In the present paper, we follow \cite{zerfas_et_al} and consider almost the same DA-ROM with the difference that we add grad-div stabilization. We will call the model grad-div-DA-ROM. We make some improvements compared with the error analy\-sis in \cite{zerfas_et_al}. First of all, we prove error bounds in which the constants do not depend on inverse powers of the viscosity. This fact is important in many applications with large Reynolds numbers. A second difference with respect to \cite{zerfas_et_al} is the following. In \cite{zerfas_et_al} the correlation matrix is based on the inner products of the snapshots without dividing by the number of snapshots as it is standard (see \cite{kunisch}). The reason for not dividing by the number of snapshots is that proceeding in that way one can bound the maximum in time of the $L^2$ error between the true solution and the projection onto the ROM space instead of having a bound for a discrete primitive in time of the $L^2$ error (let say the mean error, see  \cite{kunisch} again). Although an available bound for the maximum norm of the error in the projection simplifies the error analysis, one obtains for the correlation matrix not divided by the number of snapshots that the size of the eigenvalues scales exactly with the number of snapshots. This means that not dividing by the number of snapshots, say $M$ where $M$ is typically $(\Delta t)^{-1}$, $\Delta t$ being the
time step, we get eigenvalues $M$ times larger than u\-sing the standard correlation matrix, which in practice implies that the error bounds are multiplied by $M$ (say $(\Delta t)^{-1}$). As a consequence, there is no gain using the correlation matrix considered in \cite{zerfas_et_al}. In the present paper, we use the standard correlation matrix as defined in \cite{kunisch} and we get error bounds for the error between the  grad-div-DA-ROM and the orthogonal $L^2$ projection of the true solution onto the ROM space in which we apply the available bound for the mean error instead of requiring a bound for the maximum error. The last improvement respect to \cite{zerfas_et_al} is related to the nudging parameter. In the numerical experiments in  \cite{zerfas_et_al} there is evidence that using a large value for $\beta$ (say $\beta=100, 500$) makes a significant difference between the DA-ROM and the standard ROM, the first one being much more accurate. Although in
\cite[Remark 3.8]{zerfas_et_al} it is stated that with the analysis presented the usual upper bound on the nudging parameter can be relaxed
or even eliminated this is not true.
Actually, we found some mistakes in the statement of the main Theorem in \cite{zerfas_et_al}, Theorem 3.5. More precisely, constants $\alpha_1$,
$\alpha_2$ are defined in the following way
\begin{eqnarray}\label{alpha}
\alpha_1&:=&\nu-2\mu(\beta_2-1)C_I^2 H^2,\nonumber\\
\alpha_2&:=&2\mu-\frac{\mu C_I^2}{2\beta_1}-\frac{\mu}{2\beta_2}-6\nu^{-1}C_b^2\|{\cal S}_r\|_2\|\nabla \bu^{n+1}\|^2.
\end{eqnarray}
In  \eqref{alpha}, the value of $\mu$ is $\beta$, i.e. $\mu$ is the nudging paremeter  in \eqref{eq:mod_NS}, $H$ is the coarse mesh in \eqref{eq:mod_NS}, $n$ is the
time level, $C_I$ is a constant related to the interpolant operator $I_H$ and $C_b$ is a constant related to a standard bound of the nonlinear term. In \cite[Theorem 3.5]{zerfas_et_al} it is assumed that $\alpha_i>0$, $\beta_i>0$, $i=1,2$. Following the error analysis in \cite{zerfas_et_al} we found that the correct value for the constant $\alpha_2$ in \eqref{alpha} should be
$$
\alpha_2:=2\mu-\frac{\mu C_I^2}{\beta_1}-\frac{2\mu}{\beta_2}-6\nu^{-1}C_b^2\|{\cal S}_r\|_2\|\nabla \bu^{n+1}\|^2,
$$
while $\beta_2$ must be larger than 1. Then, in view of the assumption $\alpha_1>0$ we fall essentially into the upper bound $\nu-2\mu C_I^2 H^2>0$ assumed in references \cite{Ibdah_et_al}, \cite{larios_et_al}, \cite{Mondaini_Titi}, which means that the upper bound cannot be removed.
On the other hand, if we want to relax condition $\nu-2\mu C_I^2 H^2>0$ we can take $\beta_2=1+\epsilon$ with $\epsilon\rightarrow 0$ but in that case in view of the correct value of $\alpha_2$ we would need to take $\beta_1>(1+\epsilon)C_I^2/(2\epsilon)$, which increases as $\epsilon$ goes to zero. Since the factor $\beta_1 \mu$ multiplies the constant in the error bound of Theorem 3.5, relaxing the upper bound in the nudging parameter results in increasing the size of the constants in the error bounds.

In the present paper, as in \cite{bosco_y_yo}, \cite{bosco_titi_yo}, we do not need to assume an upper bound on the nudging parameter.
For the time integration we use the implicit Euler method although the error analysis for a second order time integrator as BDF2 can be carried out as in \cite{bosco_y_yo}.
We prove error bounds for the method with constants independent on inverse powers of the viscosity. As in  \cite{zerfas_et_al} and
previous references the error
in the initial condition goes to zero exponentially fast. The error in the grad-div-DA-ROM has three components, one coming from the time integrator used,
one due to the error in the snapshots (finite element error) and a third one due to the POD method, measured in terms on the eigenvalues of the correlation matrix. Numerical experiments confirm that, for large values of the nudging parameter, the proposed grad-div-DA-ROM rapidly converges to the real solution, and greatly improves the overall accuracy of standard POD schemes up to low viscosities over predictive time intervals, similarly to the DA-ROM in \cite{zerfas_et_al}.

The outline of the paper is as follows. In Section \ref{sec:FOM} we state some prelimina\-ries and notation. In Section \ref{sec:POD} we recall the POD method and get some a priori bounds for the orthogonal projection of the true solution onto the POD space. In Section \ref{sec:DA} we describe the proposed grad-div-DA-ROM and bound the error. Section \ref{sec:num} is devoted to show some numerical experiments. Finally, Section \ref{sec:Concl} presents the main conclusions of this work.

\section{Preliminaries and notation}\label{sec:FOM}

Let us denote by $Q=L_0^2(\Omega)=\left\{q\in L^2(\Omega)\mid (q,1)=0\right\}$.
Let $\mathcal{T}_{h}=(\tau_j^h,\phi_{j}^{h})_{j \in J_{h}}$, $h>0$ be a family of partitions of suitable domains $\Omega_h$, where $h$ denotes the maximum diameter of the elements $\tau_j^h\in \mathcal{T}_{h}$, and $\phi_j^h$ are the mappings from the reference simplex $\tau_0$ onto $\tau_j^h$.
We shall assume that the partitions are shape-regular and quasi-uniform. Let $r \geq 2$, we consider the finite-element spaces
\begin{eqnarray*}
S_{h,r}&=&\left\{ \chi_{h} \in \mathcal{C}\left(\overline{\Om}_{h}\right) \,  \big|
\, {\chi_{h}}_{|{\tau_{j}^{h}}}
\circ \phi^{h}_{j} \, \in \, P^{r-1}(\tau_{0})  \right\} \subset H^{1}(\Om_{h}),
\nonumber\\
{S}_{h,r}^0&=& S_{h,r}\cap H^{1}_{0}(\Om_{h}),
\end{eqnarray*}
where $P^{r-1}(\tau_{0})$ denotes the space of polynomials of degree at most $r-1$ on $\tau_{0}$.

We shall denote by $(X_{h,r}, Q_{h,r-1})$ the MFE pair known as Hood--Taylor elements \cite{BF,hood0} when $r\ge 3$, where
\begin{eqnarray*}
X_{h,r}=\left({S}_{h,r}^0\right)^{d},\quad
Q_{h,r-1}=S_{h,r-1}\cap L_0^2(\Omega_h),\quad r
\ge 3.
\end{eqnarray*}
To approximate the velocity
we consider the discrete divergence-free space
\begin{eqnarray*}
V_{h,r}=X_{h,r}\cap \left\{ \chi_{h} \in H^{1}_{0}(\Om_{h})^d \mid
(q_{h}, \nabla\cdot\chi_{h}) =0  \quad\forall q_{h} \in Q_{h,r-1}
\right\}.
\end{eqnarray*}
For $n\ge 1$  we define the fully discrete  Galerkin approximation with the BDF2 time discretization  $(\bu_h^{n},p_h^{n})\in X_{h,r}\times Q_{h,r-1}$ satisfying
for all $(\bvar_h,\psi_h)\in X_{h,r}\times Q_{h,r-1}$
\begin{align}\label{eq:gal}
\left(\frac{3\bu_h^{n}-4\bu_h^{n-1}+u_h^{n-2}}{2\Delta t},\bvar_h\right)
+\nu(\nabla \bu_h^{n},\nabla \bvar_h)+b_h(\bu_h^{n},\bu_h^{n},\bvar_h)
+(\nabla & p_h^{n},\bvar_h)\nonumber\\
&=(\bff^{n},\bvar_h),\nonumber\\
(\nabla \cdot \bu_h^{n},\psi_h)&=0.
\end{align}
In \eqref{eq:gal} $\bu_h^n$ is the Galerkin approximation at time $t_n$,  $\Delta t$
is the time step and $b_h(\cdot,\cdot,\cdot)$  is defined in the following way
$$
b_{h}(\bu_{h},\bv_{h},\bvar_{h}) =((\bu_{h}\cdot \nabla ) \bv_{h}, \bvar_{h})+ \frac{1}{2}( \nabla \cdot (\bu_{h})\bv_{h},\bvar_{h}),
\quad \, \forall \, \bu_{h}, \bv_{h}, \bvar_{h} \in X_{h,r}.
$$
It is straightforward to verify that $b_h$ enjoys the skew-symmetry property
\begin{equation}\label{skew}
b_h(\bu,\bv,\bw)=-b_h(\bu,\bw,\bv) \qquad \forall \, \bu, \bv, \bw\in H_0^1(\Omega)^d.
\end{equation}
Let us fix $T>0$ and define $M=T/\Delta t$. For the fully discrete Galerkin approxi\-mation the following bounds hold, see for example
\cite{Bosco_Julia_Javier_fully_post}:
\begin{eqnarray}\label{eq:cota_gal}
\|\bu^n-\bu_h^n\|_0&\le& C(\bu,p,\nu,r) (h^r+(\Delta t)^2 ),\quad 1\le n\le M\nonumber\\
\|\bu^n-\bu_h^n\|_1&\le& C(\bu,p,\nu,r) (h^{r-1}+(\Delta t)^2 ),\quad 1\le n\le M.
\end{eqnarray}
\begin{remark} \label{re:esta}
If we use a stabilized method instead of the Galerkin one we can get bounds with constants independent on inverse powers of $\nu$.
For the error analysis we carry out in this paper we need to have velocity approximations with discrete divergence zero.  Then, we could
start from a Galerkin method with grad-div stabilization as proposed in \cite{NS_grad_div}. A fully discrete version of the Galerkin method
with grad-div stabilization and the implicit Euler method is analyzed in \cite{NS_grad_div} resulting in the following bounds:
\begin{eqnarray}\label{eq:sta}
\|\bu^n-\bu_h^n\|_0+h\|\bu^n-\bu_h^n\|_1\le C(\bu,p,r)(h^{r-1}+\Delta t),\quad 1\le n\le M,
\end{eqnarray}
where the constant $C(\bu,p,r)$ depends on norms of the true solution but not directly on inverse powers of the viscosity parameter $\nu$.
Comparing the error bound \eqref{eq:sta} with \eqref{eq:cota_gal} we can observe that instead of rate $r$ in terms of $h$ a rate of convergence $r-1$ is proved. The numerical experiments in \cite{nos_review} show that this rate is sharp for small values of the viscosity parameter $\nu$.
\end{remark}

If the family of
meshes is quasi-uniform then  the following inverse
inequality holds for each $\bv_{h} \in S_{h,r}$, see e.g., \cite[Theorem 3.2.6]{Cia78},
\begin{equation}
\label{inv} \| \bv_{h} \|_{W^{m,p}(K)} \leq c_{\mathrm{inv}}
h_K^{n-m-d\left(\frac{1}{q}-\frac{1}{p}\right)}
\|\bv_{h}\|_{W^{n,q}(K)},
\end{equation}
where $0\leq n \leq m \leq 1$, $1\leq q \leq p \leq \infty$, and $h_K$
is the diameter of~$K \in \mathcal T_h$.

We consider a modified Stokes projection that was introduced in \cite{grad-div1} and that we denote by $\bs_h^m:V\rightarrow V_{h,r}$
satisfying
\begin{eqnarray}\label{stokespro_mod_def}
(\nabla \bs_h^m,\nabla \bvar_h)=(\nabla \bu,\nabla \bvar_h),\quad \forall \, \,
\bvar_{h} \in V_{h,r},
\end{eqnarray}
and the following error bound, see \cite{grad-div1}:
\begin{equation}
\|\bu-\bs_h^m\|_0+h\|\bu-\bs_h^m\|_1\le C\|\bu\|_j h^j,\qquad
1\le j\le r.
\label{stokespro_mod}
\end{equation}
From \cite{chenSiam}, we also have
\begin{align}
\|\nabla \bs_h^m\|_\infty\le C\|\nabla \bu\|_\infty \label{cotainfty1},
\end{align}
where $C$ does not depend on $\nu$ and~\cite[Lemma~3.8]{bosco_titi_yo}
\begin{align}
\label{cota_sh_inf_mu}
\|\bs_h^m\|_\infty  & \le C(\|\bu\|_{d-2}\|\bu\|_2)^{1/2},
\\
\label{la_cota_mu}
\|\nabla\bs_h^m\|_{L^{2d/(d-1)}} & \le
 C\bigl(\|\bu\|_1\|\bu\|_2\bigr)^{1/2},
\end{align}
where the constant~$C$ is independent of~$\nu$.

Let us denote by $P_Q$ the $L^2$ orthogonal projection onto $Q_{h,r-1}$. It holds
\begin{eqnarray}\label{eq:L2pre}
\|p-P_Q p\|_0\le C h^{r-1}\|p\|_{r-1},\quad p\in Q\cap H^{r-1}(\Omega).
\end{eqnarray}
We will also use the well-known
property, see \cite[Lemma 3.179]{Volker_libro}
\begin{equation}
\label{eq:div}
\left\| \nabla\cdot \bv\right\|_0\le \left\|\nabla \bv\right\|_0,\quad \bv \in
H^1_0(\Omega)^d.
\end{equation}
We will assume that the interpolation operator $I_H$ is stable in $L^2$, that is,
\begin{eqnarray}\label{eq:L^2inter}
\|I_H \bu\|_0\le c_0\|\bu\|_0,\quad \forall \bu\in L^2(\Omega)^d,
\end{eqnarray}
and that it satisfies the following approximation property,
\begin{eqnarray}\label{eq:cotainter}
\|\bu-I_H\bu\|_0\le c_I H\|\nabla \bu\|_0,\quad \forall \bu\in H_0^1(\Omega)^d.
\end{eqnarray}
The Bernardi--Girault~\cite{Ber_Gir}, Girault--Lions~\cite{Girault-Lions-2001}, or the Scott--Zhang~\cite{Scott-Z} interpolation operators
satisfy
\eqref{eq:L^2inter} and~\eqref{eq:cotainter}. Notice that the interpolation can be
on
piecewise constants.

\section{Proper Orthogonal decomposition}\label{sec:POD}

We will consider a proper orthogonal decomposition (POD) method.
Let us fix $T>0$ and $M>0$ and take $\Delta t=T/M$ and let us consider the following space
$$
{\cal V}=<\bu_h^1,\ldots,\bu_h^M>.
$$
Let $d_p$ be the dimension of the space $\cal V$.

Let $K$ be the correlation matrix corresponding to the snapshots $K=((k_{i,j}))\in {\mathbb R}^{M\times M}$ where
$$
k_{i,j}=\frac{1}{M}(\bu_h^i,\bu_h^j),
$$
and $(\cdot,\cdot)$ is the inner product in $L^2(\Omega)^d$. Following \cite{kunisch} we denote by
$ \lambda_1\ge  \lambda_2\ge\ldots\ge \lambda_{d_p}>0$ the positive eigenvalues of $K$ and by
$\bv_1,\ldots,\bv_{d_p}\in {\mathbb R}^{M}$ the associated eigenvectors. Then, the (orthonormal) POD basis is given by
\begin{eqnarray}\label{lachi}
\bpsi_k=\frac{1}{\sqrt{M}}\frac{1}{\sqrt{\lambda_k}}\sum_{j=1}^M v_k^j \bu_h(\cdot,t_j),
\end{eqnarray}
where $v_k^j$ is the $j$-th component of the eigenvector $\bv_k$ and the following error formula holds, see \cite[Proposition~1]{kunisch}
\begin{equation}\label{eq:cota_pod_0}
\frac{1}{M}\sum_{j=1}^M\left\|\bu_h^j-\sum_{k=1}^l(\bu_h^j,\bpsi_k)\bpsi_k\right\|_{0}^2=\sum_{k=l+1}^{d_p}\lambda_k,
\end{equation}
where we have used the notation $\bu_h^j=\bu_h(\cdot,t_j)$.

Denoting by $S$ the stiffness matrix for the POD basis $S=((s_{i,j}))\in {\mathbb R}^{d_p\times d_p}$ with
$s_{i,j}=(\nabla \bpsi_i,\nabla \bpsi_j)$ then for any $\bv \in {\cal V}$ the following inverse inequality holds, see \cite[Lemma 2]{kunisch}
\begin{equation}\label{eq:inv_S}
||\nabla \bv ||_0\le \sqrt{\|S\|_2}\|\bv\|_0,
\end{equation}
where $\|S\|_2$ denotes the spectral norm of $S$.

From this inverse inequality we get
\begin{eqnarray}\label{inv_1}
&&\frac{1}{M}\sum_{j=1}^M\left\|\nabla \bu_h^j-\sum_{k=1}^l(\bu_h^j,\bpsi_k)\nabla \bpsi_k\right\|_0^2\nonumber\\
&&\quad\le \frac{\|S\|_2}{M}\sum_{j=1}^M\left\|\bu_h^j-\sum_{k=1}^l(\bu_h^j,\bpsi_k)\bpsi_k\right\|_0^2\le \|S\|_2\sum_{k=l+1}^{d_p} \lambda_k.
\end{eqnarray}
Instead of \eqref{inv_1} we can also apply the following result that is taken from \cite[Lemma 3.2]{iliescu}
\begin{eqnarray}\label{inv_2}
\frac{1}{M}\sum_{j=1}^M\left\|\nabla\bu_h^j-\sum_{k=1}^l(\bu_h^j,\bpsi_k)\nabla\bpsi_k\right\|_0^2=
\sum_{k=l+1}^{d_p}\lambda_k \|\nabla\bpsi_k\|_0^2.
\end{eqnarray}
In the sequel we will denote by
$$
{\cal V}^l=<\bpsi_1,\bpsi_2,\ldots,\bpsi_l>,
$$
and by $P_l$ the $L^2$-orthogonal projection onto ${\cal V}^l$.

Although the proof of the following lemma can be found in \cite[Lemma 3.3]{iliescu} we include it here for convenience of the readers.
\begin{lema}\label{lema_iliescu}
Let $\bu$ be the solution of \eqref{NS} with initial condition $\bu^0$ and let us denote by $\bu^j=\bu(\cdot,t_j)$, then the following bounds hold
\begin{eqnarray}\label{eq:cota_pro_u0}
\frac{1}{M}\sum_{j=1}^M\|\bu^j-P_l \bu^j\|_0^2&\le& C_{0,P}:=2C(\bu,p,\nu,r)(h^{2r}+(\Delta t)^4)+2\sum_{k=l+1}^{d_p}\lambda_k,\nonumber\\
\label{eq:cota_pro_u}
\frac{1}{M}\sum_{j=1}^M\|\nabla(\bu^j-P_l \bu^j)\|_0^2&\le& C_{1,P}:= 3 C(\bu,p,\nu,r)(h^{2(r-1)}+(\Delta t)^4)\\
&&\ +3\sum_{k=l+1}^{d_p}\lambda_k\|\nabla\bpsi_k\|_0^2+3C(\bu,p,\nu,r)\|S\|_2(h^{2r}+(\Delta t)^4).\nonumber
\end{eqnarray}
\end{lema}
\begin{proof}
By definition of the $P_l$ projection
$$
\|\bu^j-P_l \bu^j\|_0\le \|\bu^j-P_l\bu_h^j\|_0.
$$
Then
$$
\|\bu^j-P_l \bu^j\|_0^2\le 2 \|\bu^j-\bu_h^j\|_0^2+2 \|\bu_h^j-P_l \bu_h^j\|_0^2.
$$
Applying now \eqref{eq:cota_gal} and \eqref{eq:cota_pod_0} we prove the first inequality in \eqref{eq:cota_pro_u}. To prove the second one
we write
\begin{eqnarray*}
\|\nabla(\bu^j-P_l \bu^j)\|_0^2&\le& 3\|\nabla(\bu^j-\bu_h^j)\|_0^2
+3\|\nabla(\bu_h^j-P_l \bu_h^j)\|_0^2
\nonumber\\
&&\quad+3 \|\nabla (P_l \bu_h^j-P_l\bu^j)\|_0^2.
\end{eqnarray*}
Taking into account that applying \eqref{eq:inv_S} we get
$$\|\nabla (P_l \bu_h^j-P_l\bu^j)\|_0^2\le \|S\|_2\|P_l(\bu_h^j-\bu^j)\|_0^2\le \|S\|_2\|\bu_h^j-\bu^j\|_0^2$$
we conclude by applying \eqref{eq:cota_gal} and \eqref{inv_2}.
\end{proof}
\subsection{A priori bounds for the orthogonal projection onto ${\cal V}^l$.}
In this section we will prove some a priori bounds for the orthogonal projection $P_l \bu^j$, $j=0,\cdots,M,$ that are
needed in the error analysis of the rest of the paper.
We start with the $L^\infty$ norm, using \eqref{inv} and \eqref{cota_sh_inf_mu}
we get
\begin{eqnarray*}
\|P_l \bu^j\|_\infty&\le& \|P_l\bu^j-\bs_h^m(\cdot,t_j)\|_\infty+\|\bs_h^m(\cdot,t_j)\|_\infty
\nonumber\\
&\le& C h^{-d/2}\|P_l\bu^j-\bs_h^m(\cdot,t_j)\|_0+C(\|\bu^j\|_{d-2}\|\bu^j\|_2)^{1/2}\nonumber\\
&\le&C h^{-d/2}\left(\|P_l\bu^j-\bu^j\|_0+\|\bu^j-\bs_h^m(\cdot,t_j)\|_0\right)+C(\|\bu^j\|_{d-2}\|\bu^j\|_2)^{1/2}.\nonumber
\end{eqnarray*}
Applying now  \eqref{eq:cota_pro_u} and \eqref{stokespro_mod} we obtain
\begin{eqnarray}\label{eq:cotaPlinf2}
\|P_l \bu^j\|_\infty\le C_{\rm inf}:=C h^{-d/2}(M^{1/2}C_{0,P}^{1/2}+C h^r\|\bu\|_r)+C(\|\bu^j\|_{d-2}\|\bu^j\|_2)^{1/2}.
\end{eqnarray}
Let us observe that in view of \eqref{eq:cotaPlinf2} and the definition of $C_{0,P}$ in \eqref{eq:cota_pro_u}
the following quantities have to be bounded:
\begin{eqnarray}\label{assumed}
\left[(\Delta t )^{-1/2}h^{r-d/2}\right],\left[ h^{-d/2}(\Delta t)^{3/2}\right], \left[(\Delta t)^{-1/2}h^{-d/2}\left(\sum_{k={l+1}}^{d_p}\lambda_k\right)^{1/2}\right].
\end{eqnarray}
\begin{remark}\label{re:apri}
The factor $M^{1/2}$ (essentially $(\Delta t)^{-1/2}$) appearing in  \eqref{eq:cotaPlinf2} comes from the rough estimate
$$
\|\bu^j-P_l \bu^j\|_0^2\le M\left(\frac{1}{M}\sum_{j=1}^M\|\bu^j-P_l\bu^j\|_0^2\right)\le M C_{0,P}.
$$
In practice one expects an equidistribution of the errors (no individual term much larger than others) in the $M$ factors in $\sum_{j=1}^M\|\bu^j-P_l\bu^j\|_0^2$
which would avoid the too pessimistic factor $(\Delta t)^{-1/2}$ in \eqref{assumed}. Actually, in some references this kind of assumption is included in the error analysis, see for example \cite[Assumption 3.2]{iliescu}. In other references, as in \cite{zerfas_et_al}, instead of taking the correlation matrix  $K=((k_{i,j}))\in {\mathbb R}^{M\times M}$ where
$$
k_{i,j}=\frac{1}{M}(\bu_h^i,\bu_h^j),
$$
they take
$$
k_{i,j}=(\bu_h^i,\bu_h^j),
$$
i.e., they drop the $1/M$ factor as suggested in \cite[Remark 3.2]{iliescu}. Then, instead of a bound for
$$
\frac{1}{M}\sum_{j=1}^M\|\bu^j-P_l\bu^j\|_0^2,
$$
as in \eqref{eq:cota_pro_u} one gets a bound for
$$
\sum_{j=1}^M\|\bu^j-P_l\bu^j\|_0^2,
$$
from which the bound for any of the terms $\|\bu^j-P_l\bu^j\|_0^2$ follows. The problem is that proceeding in this way the eigenvalues
of this approach  are the eigenvalues $\lambda_j$ in \eqref{eq:cota_pro_u} multiplied by $M$. For this reason we prefer to assume the quantities in \eqref{assumed}
are bounded since, in practice, this assumption is not hard to be satisfied while the method we propose has a smaller component of the error coming from the eigenvalues.
\end{remark}
We will get bounds for the orthogonal projection in two further norms.
Arguing as before, applying  \eqref{inv}, \eqref{la_cota_mu} we get
\begin{eqnarray*}
\|\nabla P_l\bu^j\|_{L^{2d/(d-1)}}&\le& \|\nabla(P_l \bu^j-\bs_h^m(\cdot,t_j))\|_{L^{2d/(d-1)}}+\|\nabla \bs_h^m(\cdot,t_j)\|_{L^{2d/(d-1)}}
\nonumber\\
&\le& C h^{-1/2}\|P_l \bu^j-\bs_h^m(\cdot,t_j)\|_1+C\bigl(\|\bu^j\|_1\|\bu^j\|_2\bigr)^{1/2}.
\end{eqnarray*}
Adding and subtracting $\bu^j$ and applying \eqref{eq:cota_pro_u}  and \eqref{stokespro_mod} we finally
obtain
\begin{eqnarray}\label{eq:cotanablaPld}
\|\nabla P_l\bu^j\|_{L^{2d/(d-1)}}&\le& C_{\rm ld}:=C h^{-1/2}\left(M^{1/2}C_{1,P}^{1/2}+C h^{r-1}\|\bu\|_r\right)\nonumber\\
&&\quad+C\bigl(\|\bu^j\|_1\|\bu^j\|_2\bigr)^{1/2}.
\end{eqnarray}
As before, we will assume $h$, $\Delta t$ and $l$ (number of modes) are chosen such that $h^{-1/2}M^{1/2}C_{1,P}$ is bounded.
Comments made on Remark 3.2 also apply here as well as for the following last bound.
Arguing as before, and applying  \eqref{inv} and \eqref{cotainfty1}
we get
\begin{eqnarray}\label{eq:grad_uh_infty}
\|\nabla P_l \bu^j\|_\infty&\le& \|\nabla P_l\bu^j-\nabla\bs_h^m(\cdot,t_j)\|_\infty+\|\nabla\bs_h^m(\cdot,t_j)\|_\infty
\nonumber\\
&\le& C h^{-d/2}\|\bu_h^j-\bs_h^m(\cdot,t_j)\|_1+C\|\nabla \bu^j\|_\infty.
\end{eqnarray}
Adding and subtracting $\bu^j$ and applying \eqref{eq:cota_pro_u}  and \eqref{stokespro_mod} we finally
obtain
\begin{eqnarray}\label{eq:cotanablaPlinf}
\|\nabla P_l\bu^j\|_{\infty}&\le&  C_{1,{\rm inf}}:=C h^{-d/2}\left(M^{1/2}C_{1,P}^{1/2}+C h^{r-1}\|\bu\|_r\right)+C\|\nabla \bu^j\|_\infty,\quad
\end{eqnarray}
so that in the sequel we assume $h^{-d/2}M^{1/2}C_{1,P}^{1/2}$ is bounded.

Slightly sharper a priori bounds can be obtained using a priori bounds for the Galerkin velocity approximation. 
We start with the $L^\infty$ norm, using \eqref{inv}, \eqref{cota_sh_inf_mu}, \eqref{eq:cota_gal} and  \eqref{stokespro_mod}
we get
\begin{eqnarray}\label{eq:uh_infty}
\|\bu_h^j\|_\infty&\le& \|\bu_h^j-\bs_h^m(\cdot,t_j)\|_\infty+\|\bs_h^m(\cdot,t_j)\|_\infty
\nonumber\\
&\le& C h^{-d/2}\|\bu_h^j-\bs_h^m(\cdot,t_j)\|_0+C(\|\bu^j\|_{d-2}\|\bu^j\|_2)^{1/2}\nonumber\\
&\le& C h^{-d/2}C(\bu,p,\nu,2)(h^2+\Delta t^2)+C(\|\bu^j\|_{d-2}\|\bu^j\|_2)^{1/2}\nonumber\\
&\le& C_{\bu,{\rm inf}}:=C\left( C(\bu,p,\nu,2)+(\|\bu^j\|_{d-2}\|\bu^j\|_2)^{1/2}\right),
\end{eqnarray}
whenever we assume the following condition holds for the time step
\begin{equation}\label{eq:Delta_t}
\Delta t \le C h^{d/4}.
\end{equation}

In the error bound \eqref{eq:uh_infty} we have included the factor $Ch^2\|\bu^j\|_2$ coming from the error $\|\bu^j-\bs_h^m(\cdot,t_j)\|_0$  into the factor $C(\bu,p,\nu,2)h^2$ coming from the error of the Galerkin method since $C(\bu,p,\nu,2)$ depends on $\|\bu\|_{L^\infty(H^2)}$.

Now we bound the $L^\infty$ norm of the gradient, using \eqref{inv}, \eqref{cotainfty1}, \eqref{eq:cota_gal} and  \eqref{stokespro_mod}
we get
\begin{eqnarray}\label{eq:grad_uh_infty}
\|\nabla\bu_h^j\|_\infty&\le& \|\nabla\bu_h^j-\nabla\bs_h^m(\cdot,t_j)\|_\infty+\|\nabla\bs_h^m(\cdot,t_j)\|_\infty
\nonumber\\
&\le& C h^{-d/2}\|\bu_h^j-\bs_h^m(\cdot,t_j)\|_1+C\|\nabla \bu^j\|_\infty\nonumber\\
&\le& C h^{-d/2}C(\bu,p,\nu,3)\left(h^2+(\Delta t)^2 \right)+C\|\nabla \bu^j\|_\infty\nonumber\\
&\le& C_{\bu,1,{\rm inf}}:= C\left(C(\bu,p,\nu,3)+\|\nabla  \bu\|_{L^\infty(L^\infty)}\right),
\end{eqnarray}
whenever condition \eqref{eq:Delta_t} holds.

Finally, we bound the $L^{2d/(d-1)}$ norm.  Using \eqref{inv}, \eqref{la_cota_mu}, \eqref{eq:cota_gal} and  \eqref{stokespro_mod}
and assuming again condition \eqref{eq:Delta_t} holds (indeed the weaker condition $\Delta t \le C h^{1/4}$ would be enough) we get
\begin{eqnarray}\label{eq:uh_2d_dmenosuno}
\|\nabla \bu_h^j\|_{L^{2d/(d-1)}}&\le& \|\nabla(\bu_h^j-\bs_h^m(\cdot,t_j))\|_{L^{2d/(d-1)}}+\|\nabla \bs_h^m(\cdot,t_j)\|_{L^{2d/(d-1)}}
\nonumber\\
&\le& C h^{-1/2}\|\bu_h^j-\bs_h^m(\cdot,t_j)\|_1+C\bigl(\|\bu\|_1\|\bu\|_2\bigr)^{1/2}\nonumber\\
&\le& C h^{-1/2}C(\bu,p,\nu,2)(h+(\Delta t)^2)+C\bigl(\|\bu\|_1\|\bu\|_2\bigr)^{1/2}\nonumber\\
&\le& C_{\bu,{\rm ld}}:=C\left( C(\bu,p,\nu,2)+\bigl(\|\bu\|_1\|\bu\|_2\bigr)^{1/2}\right).
\end{eqnarray}
Now, we prove a priori bounds in the same norms for $P_l \bu^j$. To this end, using inverse inequality \eqref{inv}, \eqref{eq:uh_infty},  the stability of
the $P_l$ projection, and \eqref{eq:cota_gal} we get
\begin{eqnarray}\label{eq:cotaP_i_inf2}
\|P_l \bu^j\|_\infty&\le& \|\bu_h^j\|_\infty+\|P_l \bu^j-\bu_h^j\|_\infty\le C_{\bu,{\rm inf}}+h^{-d/2}\|P_l \bu^j-\bu_h^j\|_0\nonumber\\
&\le&C_{\bu,{\rm inf}}+h^{-d/2}\|P_l (\bu^j-\bu_h^j)\|_0+h^{-d/2}\|P_l \bu_h^j-\bu_h^j\|_0\nonumber\\
&\le&C_{\bu,{\rm inf}}+h^{-d/2} \|\bu^j-\bu_h^j\|_0+h^{-d/2}\|P_l \bu_h^j-\bu_h^j\|_0\nonumber\\
&\le&C_{\bu,{\rm inf}}+h^{-d/2} C(\bu,p,\nu,2)(h^2+(\Delta t)^2)+h^{-d/2}\|P_l \bu_h^j-\bu_h^j\|_0\nonumber\\
&\le& C_{\bu,{\rm inf}}+C+h^{-d/2}\|P_l \bu_h^j-\bu_h^j\|_0,
\end{eqnarray}
where in the last inequality we assume, as before, condition \eqref{eq:Delta_t}. In view of \eqref{eq:cota_pod_0} we can write for the last term
$\|P_l \bu_h^j-\bu_h^j\|_0\le M^{1/2}\left(\sum_{k=l+1}^{d_p}\lambda_{k}\right)^{1/2}$, where, as before, the factor $M^{1/2}$ comes
from a rough estimate of any of the factors on the left-hand side in \eqref{eq:cota_pod_0}. Actually, this estimate can be slightly improved with 
the following argument.
It is easy to see that
$$
\bu_h^j-P_l \bu_h^j=\sum_{k=l+1}^{d_p}(\bu_h^j,\bpsi_k)\bpsi_k. 
$$
Using the definition of $\bpsi_k$ it is also easy to observe that
$$
\bu_h^j-P_l \bu_h^j=\sum_{k=l+1}^{d_p}(\bu_h^j,\bpsi_k)\bpsi_k=\sqrt{M}\sum_{k=l+1}^{d_p}\sqrt{\lambda_k}v_k^j\bpsi_k.
$$
And then
\begin{eqnarray}\label{cota_bos1}
\|\bu_h^j-P_l \bu_h^j\|_0&=&\sqrt{M}\left(\sum_{k=l+1}^{d_p}\lambda_k|v_k^j|^2\right)^{1/2}\le \sqrt{M}\sqrt{\lambda_{l+1}}
\left(\sum_{k=l+1}^{d_p}|v_k^j|^2\right)^{1/2}\nonumber\\
&\le& \sqrt{M}\sqrt{\lambda_{l+1}},
\end{eqnarray}
where in the last inequality we have used that $\left(\sum_{k=l+1}^{d_p}|v_k^j|^2\right)^{1/2}\le 1$ since the matrix with columns
the vectors $\bv_k$ can be enlarged to an $M\times M$ orthogonal matrix.

Inserting  \eqref{cota_bos1} into \eqref{eq:cotaP_i_inf2} we finally arrive
\begin{eqnarray}\label{inf_buena}
\|P_l \bu^j\|_\infty
&\le&C_{\rm inf}:= C_{\bu,{\rm inf}}+C+h^{-d/2}\sqrt{M}\sqrt{\lambda_{l+1}},
\end{eqnarray}
which is sharper than \eqref{eq:cotaPlinf2}.

Arguing similarly, using inverse inequality \eqref{inv}, \eqref{eq:grad_uh_infty}, \eqref{eq:inv_S}, the stability of
the $P_l$ projection, and \eqref{eq:cota_gal} we get
\begin{eqnarray}\label{eq:cotaP_i_inf3}
\|\nabla P_l \bu^j\|_\infty&\le& \|\nabla \bu_h^j\|_\infty+\|\nabla(P_l \bu^j-\bu_h^j)\|_\infty\nonumber\\
&\le& C_{\bu,1,{\rm inf}}+h^{-d/2}\|\nabla (P_l \bu^j-\bu_h^j)\|_0\nonumber\\
&\le&C_{\bu,1,{\rm inf}}+h^{-d/2}\| \nabla(P_l (\bu^j-\bu_h^j))\|_0+h^{-d/2}\|\nabla (P_l \bu_h^j-\bu_h^j)\|_0\nonumber\\
&\le&C_{\bu,{\rm inf}}+h^{-d/2}\|S\|_2^{1/2} \|\bu^j-\bu_h^j\|_0+h^{-d/2}\|\nabla (P_l \bu_h^j-\bu_h^j)\|_0\nonumber\\
&\le&C_{\bu,{\rm inf}}+h^{-d/2} \|S\|_2^{1/2}C(\bu,p,\nu,2)(h^2+(\Delta t)^2)\nonumber\\
&&\quad+h^{-d/2}\|\nabla (P_l \bu_h^j-\bu_h^j)\|_0\nonumber\\
&\le& C_{\bu,{\rm inf}}+C+h^{-d/2}\|\nabla(P_l \bu_h^j-\bu_h^j)\|_0.
\end{eqnarray}
Finally, from \eqref{inv_2} we get
$$
\|\nabla(P_l \bu_h^j-\bu_h^j)\|_0\le \sqrt{M}\|S\|_2^{1/2}\left(\sum_{k=l+1}^{d_p}\lambda_k\right)^{1/2},
$$
which inserted into \eqref{eq:cotaP_i_inf3} gives
\begin{eqnarray}\label{1_inf_buena}
\|\nabla P_l \bu^j\|_\infty
&\le&C_{1,\rm inf}:= C_{\bu,1,{\rm inf}}+C+h^{-d/2}\sqrt{M}\|S\|_2^{1/2}\left(\sum_{k=l+1}^{d_p}\lambda_k\right)^{1/2}.
\end{eqnarray}
Arguing exactly as before and applying \eqref{eq:uh_2d_dmenosuno} we also obtain
\begin{eqnarray}\label{ld_buena}
\|\nabla P_l \bu^j\|_{L^{2d/(d-1)}}
&\le&C_{\rm ld}:= C_{\bu,{\rm ld}}+C+h^{-1/2}\sqrt{M}\|S\|_2^{1/2}\left(\sum_{k=l+1}^{d_p}\lambda_k\right)^{1/2}.
\end{eqnarray}

\section{The POD Data assimilation algorithm}\label{sec:DA}

For any initial condition  the POD data assimilation approximation using the implicit Euler method and grad-div stabilization is obtained by solving for $n\ge 1$:
\begin{eqnarray}\label{eq:pod_method}
&&\left(\frac{\bu_l^{n}-\bu_l^{n-1}}{\Delta t},\bvar_l\right)+\nu(\nabla \bu_l^n,\nabla\bvar_l)+b_h(\bu_l^n,\bu_l^n,\bvar_l)
+\mu(\nabla \cdot\bu_l^n,\nabla \cdot\bvar_l)\nonumber\\
&&\quad=(\bff^n,\bvar_l)-\beta(I_H \bu_l^n-I_H\bu^n,I_H\bvar_l),\quad \forall \bvar_l\in {\cal V}^l,
\end{eqnarray}
where $\mu$ is the grad-div stabilization parameter, $\beta$ is the nudging parameter and $I_H$ is an interpolation operator over a coarse mesh. 
\begin{Theorem} \label{th_prin} Let $\bu_l^n$ be the grad-div-DA-ROM approximation defined in \eqref{eq:pod_method},  let $\bu^n$ be the velocity approximation
of the Navier-Stokes equations \eqref{NS} at time $t_n$ and let $P_l \bu^n$ be its orthogonal projection over the POD space ${\cal V}^l$. Assuming the solution $(\bu,p)$ of \eqref{NS} is smooth enough the following bound holds
\begin{eqnarray}\label{cota_final}
\|\bu_l^n-P_l \bu^n\|_0^2\le \frac{1}{\left(1+\frac{\gamma}{2}\Delta t\right)^n}\|\be_l^0\|_0^2
+TC_{1,P}\left(\nu+2\mu+\frac{2}{L}\left(\|\bu\|_2+C_{\rm ld}+C_{\rm inf}\right)\right)\nonumber\\
+T\beta c_0^2 C_{0,P}+ \frac{C}{\mu} h^{2(r-1)}\Delta t  \sum_{j=1}^n \|p^j\|_{r-1}^2+ \frac{C(\Delta t)^2}{L}\int_0^{t_n}\|\bu_{tt}(s)\|_0^2~ds,
\end{eqnarray}
where $C_{0,P}$, $C_{1,P}$ are the constants in \eqref{eq:cota_pro_u0}, and $C_{\rm ld}$, $C_{\rm inf}$ are the constants in \eqref{inf_buena}, \eqref{ld_buena}.
\end{Theorem}
\begin{proof}
Following \cite{zerfas_et_al} we will compare $\bu_l^n$ with $P_l \bu^n$. It is easy to obtain
\begin{eqnarray}\label{eq:pro_u_vl}
\left( \frac{P_l \bu^n-P_l\bu^{n-1}}{\Delta t},\bvar_l\right)
+\nu(\nabla P_l \bu^{n},\nabla \bvar_l)+b_h(P_l\bu^n,P_l \bu^n,\bvar_l)
\nonumber\\
+\mu(\nabla \cdot P_l\bu^n,\nabla \cdot \bvar_l)
=(\bff^n,\bvar_l)+
\nu(\nabla \btau_1^n,\nabla \bvar_l)+(\btau_2^n,\nabla \cdot \bvar_l) \nonumber\\+
(\btau_3^n,\bvar_l)+(\btau_4^n,\bvar_l),\quad \forall \bvar_l\in {\cal V}^l,
\end{eqnarray}
where $\btau_1^n$, $\btau_2^n$, $\btau_3^n$ and $\btau_4^n$ are defined by:
\begin{eqnarray}\label{eq:trun}
\btau_1^n&=&(P_l\bu^n-\bu^n),\nonumber\\
\btau_2^n&=&\left(p^n-P_{Q}(p^n)\right)+\mu\left(\nabla \cdot (P_l \bu^n-\bu^n)\right),\nonumber\\
\btau_3^n&=&\frac{1}{\Delta t}(\bu^n-\bu^{n-1})-\bu_t^n,\\
(\btau_4^n,\bvar_l)&=&b_h(P_l\bu^n,P_l \bu^n,\bvar_l)-b_h(\bu^n,\bu^n,\bvar_l),\nonumber
\end{eqnarray}
and we denote by $P_Q$ the $L^2$ orthogonal projection onto $Q_{h,r-1}$.

Let us denote by
$$
\be_l^n=\bu_l^n-P_l \bu^n.
$$
Subtracting \eqref{eq:pro_u_vl} from \eqref{eq:pod_method} and taking $\bvar_l=\be_l^n$ we get
\begin{eqnarray}\label{eq:error1}
&&\frac{1}{2\Delta t}\left(\|\be_l^n\|_0^2-\|\be_l^{n-1}\|_0^2\right)+\nu\|\nabla \be_l^n\|_0^2+\mu\|\nabla \cdot \be_l^n\|_0^2
+\beta\|I_H\be_l^n\|_0^2\\
&&\quad \le -b_h(\bu_l^n,\bu_l^n,\be_l^n)+b_h(P_l \bu^n,P_l\bu^n,\be_l^n)+\beta(I_H(\bu^n-P_l \bu^n),I_H\be_l^n)\nonumber\\
&&\quad \ -\nu(\nabla \btau_1^n,\nabla \be_l^n)-(\btau_2^n,\nabla \cdot \be_l^n)-
(\btau_3^n,\be_l^n)-(\btau_4^n,\be_l^n).\nonumber
\end{eqnarray}
We will argue as in \cite{bosco_y_yo}.

For the first term on the right-hand side of \eqref{eq:error1} using the skew-symmetric property \eqref{skew} we get
\begin{eqnarray}\label{eq:nonli}
&&\left|b_h(\bu_l^n,\bu_l^n,\be_l^n)-b_h(P_l \bu^n,P_l\bu^n,\be_l^n)\right|=\left|b_h(\be_l^n,P_l\bu^n,\be_l^n)\right|\nonumber\\
&&\quad\le\|\nabla P_l\bu^n\|_\infty\|\be_l^n\|_0^2+\frac{1}{2}\|\nabla \cdot \be_l^n\|_0\|P_l\bu^n\|_\infty\|\be_l^n\|_0\nonumber\\
&&\quad\le\frac{L}{2}\|\be_l^n\|_0^2+\frac{\mu}{4}\|\nabla \cdot \be_l^n\|_0^2,
\end{eqnarray}
where
\begin{eqnarray}\label{eq:laL}
L=2\max_{n\ge 0}\left(\|\nabla P_l \bu^n\|_\infty+\frac{1}{4\mu}\|P_l\bu^m\|_\infty^2\right)
\le 2 \left(C_{1,\rm inf}+\frac{C_{{\rm inf}}^2}{4\mu}\right),
\end{eqnarray}
and we have applied \eqref{inf_buena} and \eqref{1_inf_buena} in the last inequality.

For the second term on the right-hand side of \eqref{eq:error1}, applying the $L^2$-stability of the interpolation operator
\eqref{eq:L^2inter} we get
\begin{eqnarray}\label{eq:conbeta}
\beta(I_H(\bu^n-P_l \bu^n),I_H\be_l^n)&\le& \beta c_0 \|\bu^n-P_l\bu^n\|_0\|I_H\be_l^n\|_0\nonumber\\
&\le& \frac{\beta}{2}c_0^2\|\bu^n-P_l\bu^n\|_0^2+\frac{\beta}{2}\|I_H\be_l^n\|_0^2.
\end{eqnarray}
For the truncation errors we write
\begin{eqnarray}\label{eq:trun2}
|\nu(\nabla \btau_1^n,\nabla \be_l^n)|&\le& \frac{\nu}{2}\|\nabla \btau_1^n\|_0^2+\frac{\nu}{2}\|\nabla \be_l^n\|_0^2,\nonumber\\
|(\btau_2^n,\nabla \cdot\be_l^n)|&\le& \frac{\|\btau_2^n\|_0^2}{\mu}+\frac{\mu}{4}\|\nabla \cdot \be_l^n\|_0^2,\\
|(\btau_3^n+\btau_4^n,\be_l^n)|&\le&\frac{1}{2L}\|\btau_3^n+\btau_4^n\|_0^2+\frac{L}{2}\|\be_l^n\|_0^2.\nonumber
\end{eqnarray}
Inserting \eqref{eq:nonli}, \eqref{eq:conbeta} and \eqref{eq:trun2} into \eqref{eq:error1} we get
\begin{eqnarray}\label{eq:error2}
&&\frac{1}{2}\frac{1}{\Delta t}\left(\|\be_l^n\|_0^2-\|\be_l^{n-1}\|_0^2\right)+\frac{\nu}{2}\|\nabla \be_l^n\|_0^2
+\frac{\beta}{2}\|I_H \be_l^n\|_0^2+\frac{\mu}{2}\|\nabla \cdot \be_l^n\|_0^2\\
&&\quad\le L\|\be_l^n\|_0^2+\frac{\nu}{2}\|\nabla \btau_1^n\|_0^2+\frac{\|\btau_2^n\|_0^2}{\mu}+\frac{1}{2L}\|\btau_3^n+\btau_4^n\|_0^2+\frac{\beta}{2}c_0^2\|\bu^n-P_l\bu^n\|_0^2.\nonumber
\end{eqnarray}
The following argument is taken from \cite{bosco_y_yo} and \cite{bosco_titi_yo}. We first observe that
$$
L\|\be_l^n\|_0^2\le 2L \|I_H \be_l^n\|_0^2+2L\|(I-I_H)\be_l^n\|_0^2
$$
so that assuming
$$
\beta\ge 8L
$$
and multiplying \eqref{eq:error2} by 2 we obtain
\begin{eqnarray*}
\frac{1}{\Delta t}\left(\|\be_l^n\|_0^2-\|\be_l^{n-1}\|_0^2\right)+{\nu}\|\nabla \be_l^n\|_0^2
+\frac{\beta}{2}\|I_H \be_l^n\|_0^2+{\mu}\|\nabla \cdot \be_l^n\|_0^2-4L\|(I-I_H)\be_l^n\|_0^2\nonumber\\
\le {\nu}\|\nabla \btau_1^n\|_0^2+\frac{2\|\btau_2^n\|_0^2}{\mu}+\frac{1}{L}\|\btau_3^n+\btau_4^n\|_0^2+{\beta}c_0^2\|\bu^n-P_l\bu^n\|_0^2.
\end{eqnarray*}
Applying \eqref{eq:cotainter} we have
$$
{\nu}\|\nabla  \be_l^n\|_0^2-4L\|(I-I_H)\be_l^n\|_0^2
\ge {\nu}\|\nabla  \be_l^n\|_0^2-4L c_I^2H^2\|\nabla \be_l^n\|_0^2\ge \frac{\nu}{2}\|\nabla \be_l^n\|_0^2,
$$
whenever
\begin{equation}\label{eq:laH}
H\le \frac{\nu^{1/2}}{(8L)^{1/2}c_I},
\end{equation}
and then
\begin{eqnarray}\label{eq:error3}
\frac{1}{\Delta t}\left(\|\be_l^n\|_0^2-\|\be_l^{n-1}\|_0^2\right)+\frac{\nu}{2}\|\nabla \be_l^n\|_0^2
+\frac{\beta}{2}\|I_H \be_l^n\|_0^2+{\mu}\|\nabla \cdot \be_l^n\|_0^2\nonumber\\
\le {\nu}\|\nabla \btau_1^n\|_0^2+\frac{2\|\btau_2^n\|_0^2}{\mu}+\frac{1}{L}\|\btau_3^n+\btau_4^n\|_0^2+{\beta}c_0^2\|\bu^n-P_l\bu^n\|_0^2.
\end{eqnarray}
Applying \eqref{eq:cotainter} again we get
\begin{eqnarray*}
\frac{\nu}{2}\|\nabla \be_l^n\|_0^2
+\frac{\beta}{2}\|I_H \be_l^n\|_0^2
&\ge& \frac{\nu}{2}c_I^{-2}H^{-2}\|(I-I_H)\be_l^n\|_0^2+\frac{\beta}{2}\|I_H \be_l^n\|_0^2\nonumber\\
&\ge& \gamma\left(\|I_H \be_l^n\|_0^2+\|(I-I_H)\be_l^n\|_0^2\right)\ge\frac{\gamma}{2}\|\be_l^n\|_0^2,
\end{eqnarray*}
where
\begin{equation}\label{eq:gamma}
\gamma=\min\left\{\frac{\nu}{2}c_I^{-2}H^{-2},\frac{\beta}{2}\right\},
\end{equation}
and then, going back to \eqref{eq:error3} we reach
\begin{eqnarray}\label{eq:error4}
&&\frac{1}{\Delta t}\left(\|\be_l^n\|_0^2-\|\be_l^{n-1}\|_0^2\right)+\frac{\gamma}{2}\|\be_l^n\|_0^2+{\mu}\|\nabla \cdot \be_l^n\|_0^2\nonumber\\
&&\quad\le {\nu}\|\nabla \btau_1^n\|_0^2+\frac{2\|\btau_2^n\|_0^2}{\mu}+\frac{1}{L}\|\btau_3^n+\btau_4^n\|_0^2+{\beta}c_0^2\|\bu^n-P_l\bu^n\|_0^2.
\end{eqnarray}
Let us denote by
$$
\|\btau^n\|_0^2:={\nu}\|\nabla \btau_1^n\|_0^2+\frac{2\|\btau_2^n\|_0^2}{\mu}+\frac{1}{L}\|\btau_3^n+\btau_4^n\|_0^2+{\beta}c_0^2\|\bu^n-P_l\bu^n\|_0^2.
$$
From \eqref{eq:error4} we have
$$
\left(1+\frac{\gamma}{2}\Delta t\right)\|\be_l^n\|_0^2\le  \|\be_l^{n-1}\|_0^2+\Delta t\|\btau^n\|_0^2,
$$
and then for $1\le n\le M$ we get
\begin{eqnarray}\label{eq:error5}
\|\be_l^n\|_0^2&\le& \frac{1}{\left(1+\frac{\gamma}{2}\Delta t\right)^n}\|\be_l^0\|_0^2
+\Delta t \sum_{j=1}^n\frac{1}{\left(1+\frac{\gamma}{2}\Delta t\right)^{n-j+1}}\|\btau^j\|_0^2\nonumber\\
&\le&\frac{1}{\left(1+\frac{\gamma}{2}\Delta t\right)^n}\|\be_l^0\|_0^2
+\Delta t \sum_{j=1}^n\|\btau^j\|_0^2.
\end{eqnarray}
To conclude we need to bound the truncation error on the right-hand side of \eqref{eq:error5}. We first observe that
applying \eqref{eq:cota_pro_u} we get
\begin{eqnarray}\label{eq:tau1}
\nu\Delta t \sum_{j=1}^n\|\nabla\btau_1^j\|_0^2&=&  \frac{\nu T}{M}\sum_{j=1}^n\|\nabla\btau_1^j\|_0^2
\le \frac{\nu T}{M}\sum_{j=1}^M\|\nabla(P_l \bu^j-\bu^j)\|_0^2\nonumber\\
&\le& \nu TC_{1,P}.
\end{eqnarray}
For the second term in the truncation error applying \eqref{eq:L2pre}, \eqref{eq:div} and \eqref{eq:cota_pro_u} again we get
\begin{eqnarray}\label{eq:tau2}
\Delta t \sum_{j=1}^n\frac{\|\btau_2^j\|_0^2}{\mu}&\le& \frac{2}{\mu}\Delta t \sum_{j=1}^n\|p^j-P_Q (p^j)\|_0^2+\frac{2\mu T}{M}\sum_{j=1}^M\|\nabla(P_l \bu^j-\bu^j)\|_0^2\nonumber\\
&\le& \frac{C}{\mu} h^{2(r-1)}\Delta t  \sum_{j=1}^n \|p^j\|_{r-1}^2+ 2\mu T C_{1,P}.
\end{eqnarray}
For the first term in the third term of the truncation error we obtain
\begin{eqnarray}\label{eq:tau3}
\Delta t \sum_{j=1}^n\|\btau_3^j\|_0^2=\Delta t \sum_{j=1}^n \left\|\bu_t^j-\frac{\bu^j-\bu^{j-1}}{\Delta t}\right\|_0^2\le C
(\Delta t)^2\int_0^{t_n}\|\bu_{tt}(s)\|_0^2~ds.
\end{eqnarray}
For the second term in the third term of the truncation error we apply \cite[Lemma 2]{NS_grad_div} and \eqref{eq:cotaPlinf2} and
\eqref{eq:cotanablaPlinf} to get
\begin{eqnarray*}
&&\left\|b_h(P_l\bu^j,P_l \bu^j,\bvar_l)-b_h(\bu^j,\bu^j,\bvar_l)\right\|_0\nonumber\\
&&\quad \le \left(\|\bu\|_2+\|\nabla P_l \bu^j\|_{L^{2d/(d-1)}}
+\|P_l \bu^j\|_\infty\right)\|\nabla (P_l \bu^j-\bu^j)\|_0\nonumber\\
&&\quad \le \left(\|\bu\|_2+C_{\rm ld}+C_{\rm inf}\right)\|\nabla (P_l \bu^j-\bu^j)\|_0.
\end{eqnarray*}
And then applying \eqref{eq:cota_pro_u} we get
\begin{eqnarray}\label{eq:tau4}
\Delta t \sum_{j=1}^n\|\btau_4^j\|_0^2&\le& \frac{T}{M}\left(\|\bu\|_2+C_{\rm ld}+C_{\rm inf}\right)\sum_{j=1}^M\|\nabla(P_l \bu^j-\bu^j)\|_0^2
\nonumber\\
&\le& T\left(\|\bu\|_2+C_{\rm ld}+C_{\rm inf}\right)C_{1,P}.
\end{eqnarray}
Finally, for the last term in the truncation error applying \eqref{eq:cota_pro_u} again we obtain
\begin{eqnarray}\label{eq:tauulti}
\beta c_0^2\Delta t \sum_{j=1}^n\|P_l \bu^j-\bu^j\|_0^2\le \frac{T\beta c_0^2}{M}\sum_{j=1}^n\|P_l \bu^j-\bu^j\|_0^2
\le T\beta c_0^2 C_{0,P}.
\end{eqnarray}
Inserting \eqref{eq:tau1}, \eqref{eq:tau2}, \eqref{eq:tau3}, \eqref{eq:tau4} and \eqref{eq:tauulti} into
\eqref{eq:error5} we conclude \eqref{cota_final}.
\end{proof}.
\begin{remark}
Let us observe that for the POD data assimilation method we can start from $\bu_l^0=0$ since the initial error decays exponentially
to zero. For the set of snapshots we do not need to include $\bu_h^0$ since we apply \eqref{eq:cota_pro_u} with $j$ starting at $1$.
This is different from references \cite{iliescu}, \cite{kunisch} where the initial condition $\bu_h^0$ is included into
the set of snapshots and agrees with \cite{zerfas_et_al}.
\end{remark}
\begin{Theorem}Assuming conditions of Theorem \ref{th_prin} hold the following bounds can be obtained
\begin{eqnarray}\label{cota_final3}
\frac{1}{M}\sum_{j=1}^M\|\bu_l^j-\bu^j\|_0^2&\le&{\left(1+\frac{\gamma}{2}\Delta t\right)^{-1}}\|\be_l^0\|_0^2
+(T\beta c_0^2+1) C_{0,P}\nonumber\\
&&\ +TC_{1,P}\left(\nu+2\mu+\frac{2}{L}\left(\|\bu\|_2+C_{\rm ld}+C_{\rm inf}\right)\right)\\
&&\ + \frac{C}{\mu} h^{2(r-1)}\Delta t  \sum_{j=1}^M \|p^j\|_{r-1}^2+ \frac{C(\Delta t)^2}{L}\int_0^{T}\|\bu_{tt}(s)\|_0^2~ds.\nonumber
\end{eqnarray}
\end{Theorem}
\begin{proof}
Arguing as in \cite{samuele_pod}, we observe that from \eqref{cota_final} we  get
\begin{eqnarray}\label{cota_final2}
\frac{1}{M}\sum_{j=1}^M\|\be_l^j\|_0^2&\le&{\left(1+\frac{\gamma}{2}\Delta t\right)^{-1}}\|\be_l^0\|_0^2+T\beta c_0^2 C_{0,P}
\nonumber\\&&\ +TC_{1,P}\left(\nu+2\mu+\frac{2}{L}\left(\|\bu\|_2+C_{\rm ld}+C_{\rm inf}\right)\right)\\
&&\ + \frac{C}{\mu} h^{2(r-1)}\Delta t  \sum_{j=1}^M \|p^j\|_{r-1}^2+ \frac{C(\Delta t)^2}{L}\int_0^{T}\|\bu_{tt}(s)\|_0^2~ds,\nonumber
\end{eqnarray}
so that applying triangle inequality together with \eqref{eq:cota_pro_u} we finally reach \eqref{cota_final3}
\end{proof}
\begin{remark}
Let us observe that in the error bound \eqref{cota_final3} we have lost the exponential decay of the initial error since we have taken
the maximum error on the right-hand side of \eqref{cota_final} to reach \eqref{cota_final2} and consequently \eqref{cota_final3}. To avoid
this problem one can apply triangle inequality to \eqref{cota_final} to bound the error $\|\bu_l^n-\bu^n\|_0$. Then, one would have on the
right-hand side of the error bound the term $\|\bu^n-P_j\bu^n\|_0$ for which the rough estimate $\|\bu^n-P_j\bu^n\|_0\le M^{1/2}C_{0,P}^{1/2}$ follows
from \eqref{eq:cota_pro_u}. Assuming an equidistribution of the errors in \eqref{eq:cota_pro_u} (as observed in Remark~\ref{re:apri}) one would
avoid the factor $M^{1/2}$. This is the behavior we observe in practice in the numerical experiments (see Section~\ref{sec:num}) where both the exponential decay of the initial errors is observed together with the absence of the factor $M^{1/2}$ in the error behavior.
\end{remark}
 \begin{remark}
 Accordingly to Remark~\ref{re:apri} we observe that to get the error bounds \eqref{cota_final}  we have applied \eqref{eq:cota_pro_u}. In reference
 \cite{zerfas_et_al} the authors instead of the left-hand side of \eqref{eq:cota_pro_u} they bound $M$ times the left-hand side
 of \eqref{eq:cota_pro_u}.  To this end, instead
 of the correlation matrix  $K=((k_{i,j}))\in {\mathbb R}^{M\times M}$ whith
$
k_{i,j}=(1/M)(\bu_h^i,\bu_h^j),
$
they take
$
k_{i,j}=(\bu_h^i,\bu_h^j),
$
dropping the $1/M$ factor.
Then, the eigenvalues $\lambda_j$ in the error bounds of \cite{zerfas_et_al} are multiplied by $M$ respect to the eigenvalues $\lambda_j$ of the present paper.
\end{remark}

\section{Numerical Experiments}\label{sec:num}

In this section, we present numerical results for the grad-div-DA-ROM \eqref{eq:pod_method} introduced and analyzed in the previous section. The numerical experiments are performed on the benchmark problem of the 2D unsteady flow around a cylinder with circular cross-section \cite{SchaferTurek96} at Reynolds numbers $Re=100,1000$. The open-source FE software FreeFEM \cite{Hecht12} has been used to run the numerical experiments.

\medskip 

{\em Setup for numerical simulations.}
Following \cite{SchaferTurek96}, the computational domain is given by a rectangular channel with a circular hole (see Figure \ref{fig:Mesh} on top for the computational grid used for $Re=100$ and Figure \ref{fig:MeshRef} on top for the computational grid used for $Re=1000$):
$$
\Om=\{(0,2.2)\times(0,0.41)\}\backslash \{\xv : (\xv-(0.2,0.2))^2 \leq 0.05^2\}.
$$
\begin{figure}[htb]
\begin{center}
\centerline{\includegraphics[width=4.75in]{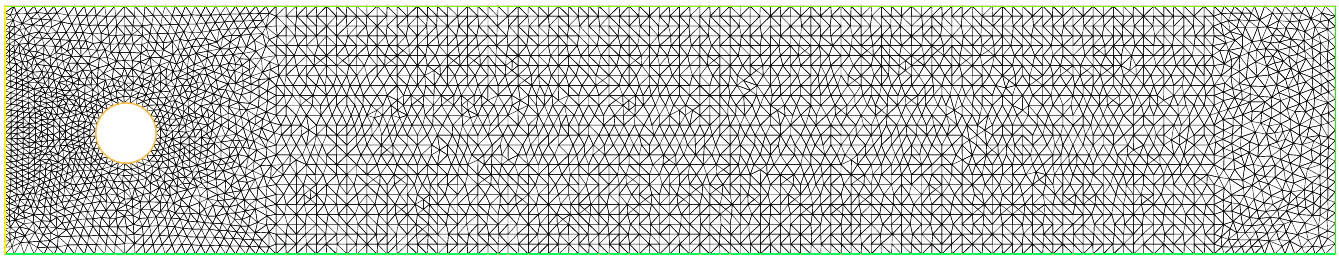}}
\centerline{\includegraphics[width=4.75in]{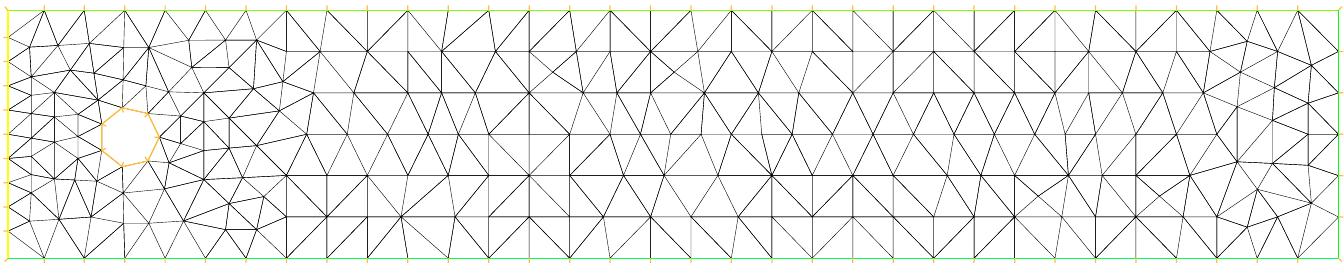}}
\caption{Fine mesh (top) and coarse mesh (bottom), $H=4\,h$, for example \ref{sec:Re=100} (Case $Re=100$).}\label{fig:Mesh}
\end{center}
\end{figure}

\begin{figure}[htb]
\begin{center}
\centerline{\includegraphics[width=4.75in]{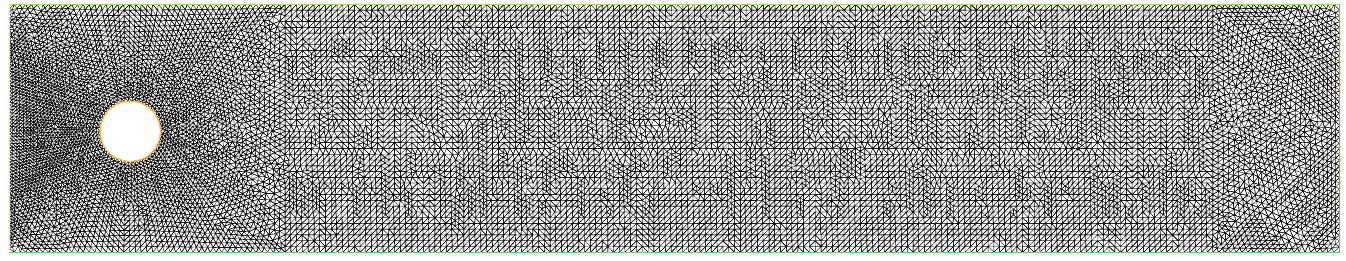}}
\centerline{\includegraphics[width=4.75in]{CoarseMesh}}
\caption{Fine mesh (top) and coarse mesh (bottom), $H=7\,h$, for example \ref{sec:Re=1000} (Case $Re=1000$).}\label{fig:MeshRef}
\end{center}
\end{figure}
No slip boundary conditions are prescribed on the horizontal walls and on the cylinder, and a parabolic inflow profile is provided at the inlet:
$$
\uv(0,y,t)=(4U_{m}y(A-y)/A^2, 0)^{T},
$$
with $U_{m}=\uv(0,H/2,t)=1.5\,\rm{m/s}$, and $A=0.41\,\rm{m}$ the channel height. At the outlet, we impose outflow (do nothing) boundary conditions $(\nu\nabla\uv - p\,Id)\nv={\bf 0}$, with $\nv$ the outward normal to the domain.

We consider two different values of the kinematic viscosity of the fluid: $\nu=10^{-3}, 10^{-4}\,\rm{m^{2}/s}$, and there is no external (gravity) forcing, i.e. $\fv={\bf 0}\,\rm{m/s^2}$. Based on the mean inflow velocity $\overline{U}=2U_{m}/3=1\,\rm{m/s}$, the cylinder diameter $D=0.1\,\rm{m}$ and the different values of the kinematic viscosity of the fluid $\nu=10^{-3}, 10^{-4}\,\rm{m^{2}/s}$, the Reynolds numbers considered are $Re=\overline{U}D/\nu=100, 1000$. In the fully developed periodic regime for the two Reynolds numbers, a vortex shedding can be observed behind the obstacle, resulting in the well-known von K\'arm\'an vortex street (see Figure \ref{fig:FinFEMVel}).

\begin{figure}[htb]
\begin{center}
\centerline{\includegraphics[width=4.75in]{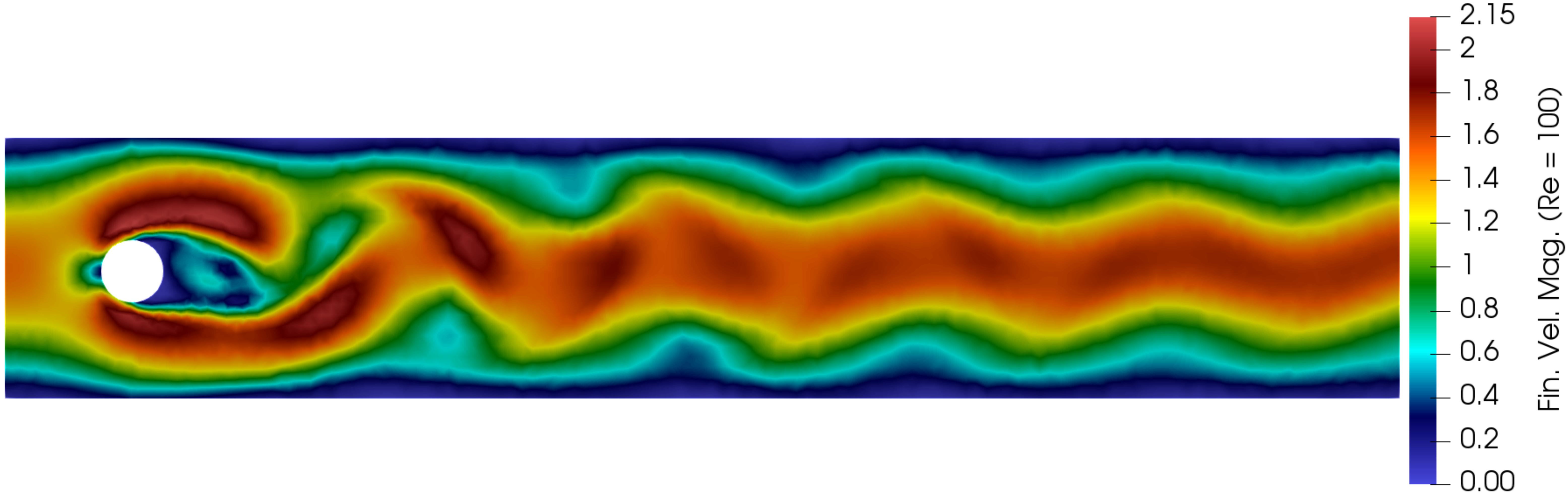}}\vspace{0.5cm}
\centerline{\includegraphics[width=4.75in]{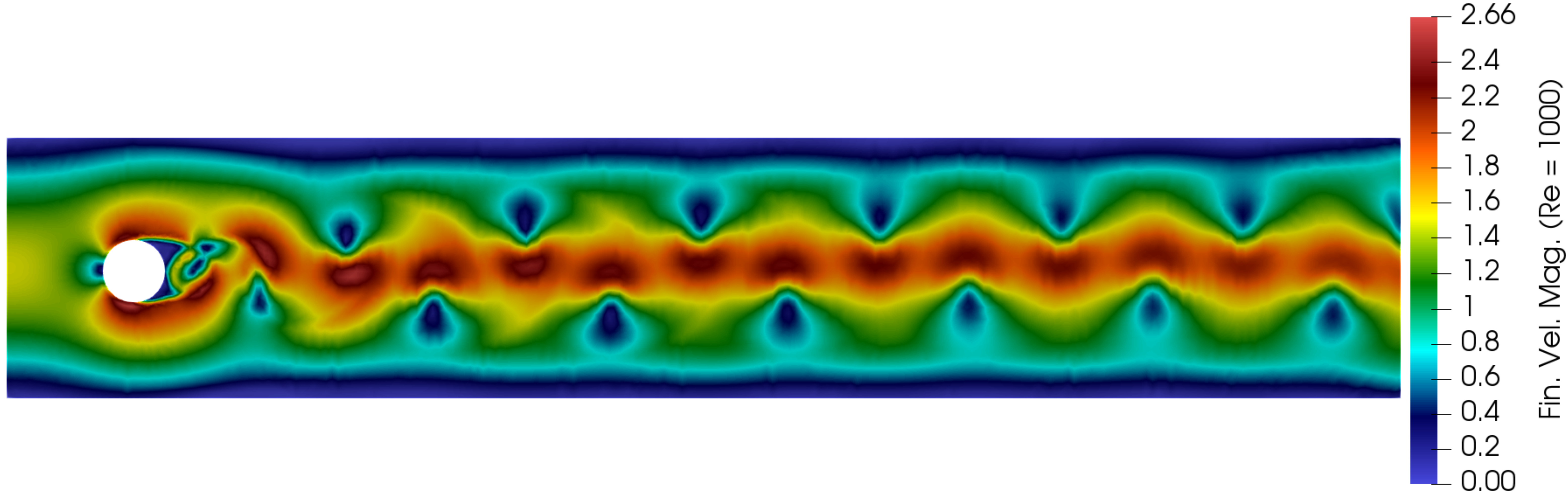}}
\caption{Final finite element DNS velocity magnitude for examples \ref{sec:Re=100} (Case $Re=100$) and \ref{sec:Re=1000} (Case $Re=1000$), from top to bottom.}\label{fig:FinFEMVel}
\end{center}
\end{figure}

For the evaluation of computational results, we are interested in studying the temporal evolution of the following
quantities of interest. The kinetic energy of the flow is the most frequently monitored quantity, given by:
$$
E_{Kin}=\frac{1}{2}\nor{\uv}{{\bf L}^2}^2.
$$
Other relevant quantities of interest are the drag and lift coefficients. In order to reduce the boundary approximation influences, in the present work these quantities are computed as volume integrals \cite{John04paper}:
$$
c_{D}=-\frac{2}{D\overline{U}^2}\left[(\partial_t \uv,\vv_D)+b(\uv,\uv,\vv_D)+\nu(\nabla\uv,\nabla\vv_D)-(p,\div\vv_D)\right],
$$
$$
c_{L}=-\frac{2}{D\overline{U}^2}\left[(\partial_t \uv,\vv_L)+b(\uv,\uv,\vv_L)+\nu(\nabla\uv,\nabla\vv_L)-(p,\div\vv_L)\right],
$$
for arbitrary test functions $\vv_{D},\vv_{L}\in {\bf H}^{1}$ such that $\vv_{D}=(1,0)^{T}$ on the boundary of the cylinder and vanishes on the other boundaries, $\vv_{L}=(0,1)^{T}$ on the boundary of the cylinder and vanishes on the other boundaries. In the actual computations, we have used the approach in \cite{zerfas_et_al}, where the pressure term is not necessary to compute $c_{D}, c_{L}$, since the test functions $\vv_D,\vv_L$ are computed by Stokes projection, so that they are taken properly in the discrete divergence-free space $V_{h,r}$. For the lower Reynolds number case ($Re=100$), reference intervals for these coefficients were given in \cite{SchaferTurek96} (see second row of Table \ref{tab:DragLiftCoef}), together with the Strouhal number $St=Df/\overline{U}$, where $f$ is the frequency of the vortex shedding. For the higher Reynolds number case ($Re=1000$), we will take the computed finite element DNS drag and lift coefficients as reference values.

\begin{table}[htb]
$$\hspace{-0.1cm}
\begin{tabular}{|c|c|c|c|}
\hline
 & $c_{D}^{max}$ & $c_{L}^{max}$ & $St$\\
\hline
DNS & $3.22$ & $0.96$ & $0.303$\\
\hline
Reference results from \cite{SchaferTurek96} & $[3.22,3.24]$ & $[0.99,1.01]$ & $[0.295,0.305]$\\
\hline
\end{tabular}$$\caption{Maximum drag coefficient $c_{D}^{max}$, maximum lift coefficient $c_{L}^{max}$, and Strouhal number for the finite element DNS solution (first row), compared with reference intervals from \cite{SchaferTurek96} (second row) for example \ref{sec:Re=100} (Case $Re=100$).}\label{tab:DragLiftCoef}
\end{table}

\medskip 

{\em DNS-FEM and POD modes.}
The numerical method used to compute the snapshots is the DNS-FEM \eqref{eq:gal} described in Section \ref{sec:FOM}, with a spatial discretization using the mixed inf-sup stable ${\bf P}^{2}-\mathbb{P}^1$ Taylor-Hood FE for the pair velocity-pressure. For the time discretization, a semi-implicit Backward Differentiation Formula of order 2 (BDF2) has been applied, which guarantees a good balance between numerical accuracy and computational complexity (cf. \cite{AhmedRubino19}). In particular, we have considered an extrapolation for the convection velocity by means of Newton--Gregory backward polynomials \cite{Cellier91}. Without entering into the details of the derivation, for which we refer the reader to e.g. \cite{Cellier91}, we consider the following extrapolation of order two for the discrete velocity: $\widehat{\uv}_{h}^{n}=2\uhv^{n}-\uhv^{n-1}$, $n\geq 1$, in order to achieve a second-order accuracy in time. For the initialization $(n=0)$, we have considered $\uhv^{-1}=\uhv^{0}=\uv_{0h}$, being $\uv_{0h}$ the initial condition, so that the time scheme reduces to the semi-implicit Euler method for the first time step $(\Delta t)^{0}=(2/3)\Delta t$. In the DNS, an impulsive start is performed, i.e. the initial condition is a zero velocity field, and the time step is $\Delta t = 2\times 10^{-3}\,\rm{s}$. Time integration is performed till a final time $T=7\,\rm{s}$. In the time period $[0,5]\,\rm s$, after an initial spin-up, the flow is expected to develop to full extent, including a subsequent relaxation time. Afterwards, it reaches a periodic-in-time (statistically- or quasi-steady) state, see Figure \ref{fig:Energy}, where we plot kinetic energy temporal evolution for the DNS solutions at Reynolds numbers $Re=100, 1000$. From Table \ref{tab:DragLiftCoef}, we observe that DNS results at $Re=100$ agree quite well with reference results from \cite{SchaferTurek96}.

\begin{figure}[htb]
\begin{center}
\centerline{\includegraphics[width=5in]{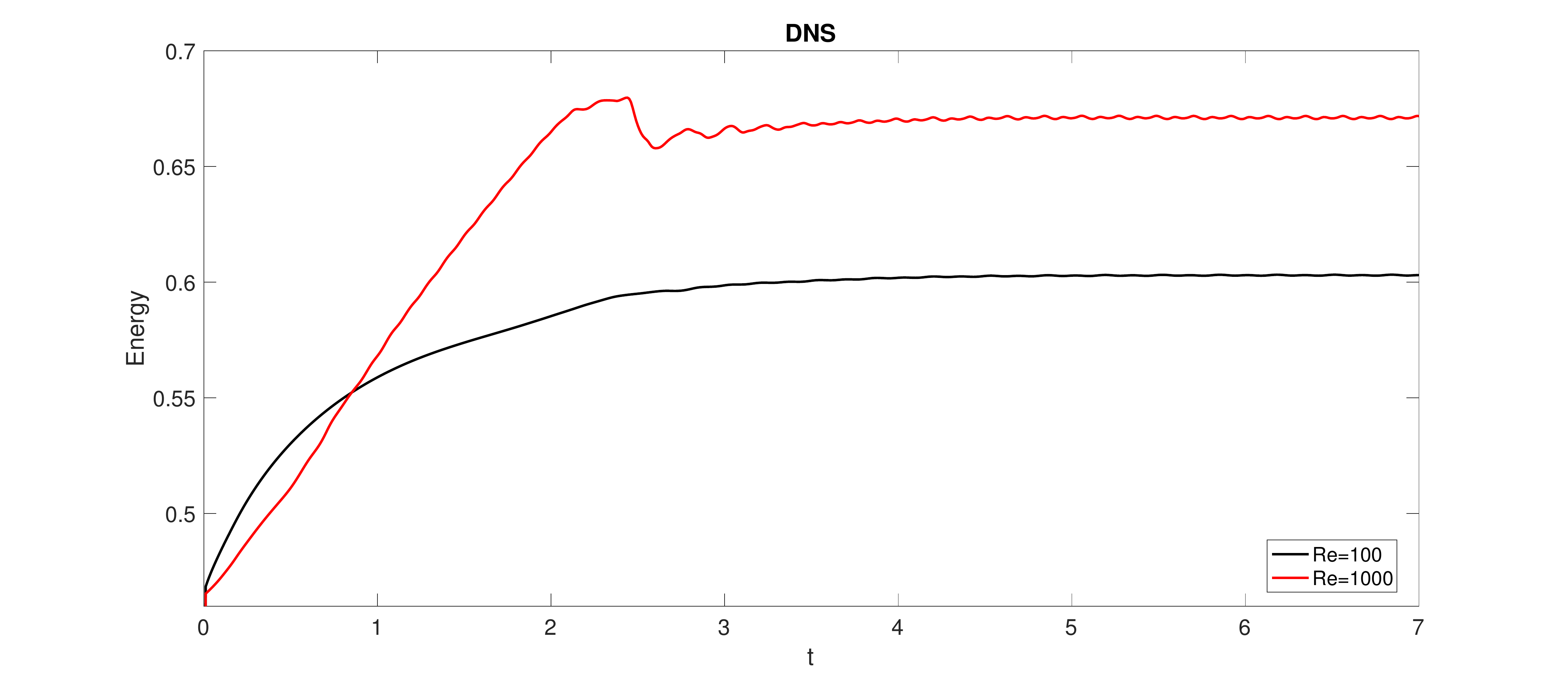}}
\caption{Temporal evolution of kinetic energy for the DNS solution computed for examples \ref{sec:Re=100} (Case $Re=100$) and \ref{sec:Re=1000} (Case $Re=1000$).}\label{fig:Energy}
\end{center}
\end{figure}

The POD velocity modes are generated in $L^2$ by the method of snapshots with velocity centered-trajectories \cite{IliescuJohn15} by storing every DNS velocity solution from $t=5$, when the solution had reached a periodic-in-time state, and using one period of snapshot data for the two Reynolds numbers $Re=100, 1000$. The full period length of the statistically steady state is, respectively, $0.332\,\rm{s}$ for $Re=100$ and $0.22\,\rm{s}$ for $Re=1000$, thus we collect $166$ snapshots for $Re=100$ and $110$ snapshots for $Re=1000$. The rank of the velocity data set at $Re=100, 1000$ is, respectively, $d_{p}=27, 51$, for which $\lambda_k<10^{-10}, k>d_{p}$, see Figure \ref{fig:PODev} where we show the decay of POD velocity eigenvalues $\lambda_k$, $k=1,\ldots, d_p$, for the two Reynolds numbers $Re=100, 1000$.

\begin{figure}[htb]
\begin{center}
\centerline{\includegraphics[width=5in]{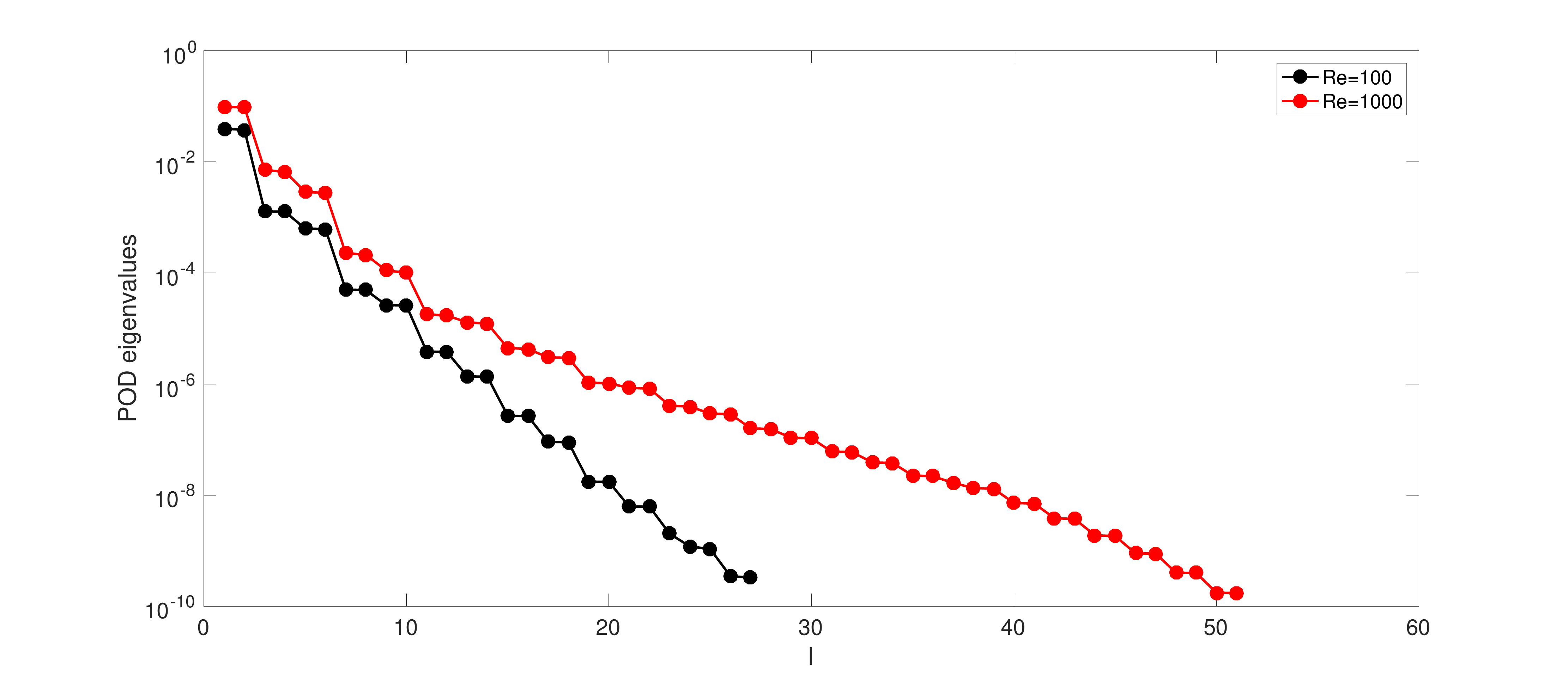}}
\caption{POD velocity eigenvalues for examples \ref{sec:Re=100} (Case $Re=100$) and \ref{sec:Re=1000} (Case $Re=1000$).}\label{fig:PODev}
\end{center}
\end{figure}

\medskip 

{\em Numerical results for grad-div-DA-ROM.}
With POD velocity modes genera\-ted, the fully discrete grad-div-DA-ROM \eqref{eq:pod_method} is constructed as discussed in the previous section, using the semi-implicit BDF2 time scheme as for the DNS-FEM, and run with varying values of the nudging parameter ($\beta=10, 100, 500$) in the stable response time interval $[5,7]\,\rm{s}$ with $\Delta t=2\times 10^{-3}\,\rm{s}$ and a small number ($l=8$) of POD velocity modes, which already give a reasonable accuracy for the proposed method at Reynolds numbers $Re=100, 1000$, especially for large values of the nudging parameter ($\beta=100, 500$). The coarse mesh for grad-div-DA-ROM is given by the same computational grid for the two Reynolds numbers, represented at the bottom of Figures \ref{fig:Mesh}, \ref{fig:MeshRef}. For $Re=100$ this coarse mesh corresponds to $H=4\,h$, while for $Re=1000$ it corresponds to $H=7\,h$, being $H$ the resolution of the coarse spatial mesh, and $h$ the one of the used fine spatial computational grid. In the current implementation, since $H/h$ is bounded, $I_H$ has been chosen as the nodal Lagrange interpolation operator onto the coarse mesh of size $H$, for which error bounds have been proven in \cite{bosco_y_yo, bosco_titi_yo}. A numerical comparison with respect to an interpolation operator on piecewise constants \cite{bosco_y_yo, bosco_titi_yo} gave almost similar results (not shown for brevity). For the grad-div-DA-ROM computations, we start from zero initial velocity conditions at $t=5\,\rm{s}$ and begin assimilation with the DNS solution at $t=5.002\,\rm{s}$, whereas $I_{H}\uv^{n}$ is computed only in one period and then repeated in the rest of periods, thus being the DNS data to construct the reduced basis sufficient to implement the DA term, and no further information is needed. In the following numerical experiments, we observe that the grad-div-DA-ROM solution exponentially converges to the DNS solution in time and the speed of convergence grows as we increase the nudging parameter $\beta$.

To assess the numerical accuracy of the new grad-div-DA-ROM, the temporal evolution of the drag and lift coefficients, and kinetic energy are monitored and compared to the DNS solutions in the stable response time interval $[5,7]\,\rm{s}$. Following \cite{zerfas_et_al}, we also investigate the new grad-div-DA-ROM in predicting the cited quantities of interest when inaccurate snapshots ($64\%$ of one full period) are used in its construction. The interest of this numerical investigation relies on the fact that, in practice, complete sets of data are usually not available, or the quantity of data needed to reasonably catch up the behavior of the real solution is usually unknown. This also allows to reduce the offline computational cost of the method, since a reduced number of snapshots is used to build the correlation matrix, while almost maintaining the numerical accuracy of comp\-lete data sets simulations. At the same time, we compare the performance of the grad-div-DA-ROM to that of the standard Galerkin-ROM (G-ROM), for which $\mu=0$ and $\beta=0$, the grad-div-ROM, for which $\beta=0$, and the DA-ROM, for which $\mu=0$. The DA-ROM has been introduced and analyzed in \cite{zerfas_et_al}. To perform the comparison, here we run it with the same numerical setup as for the grad-div-DA-ROM. From the following numerical experiments, we observe that under the same setup conditions, both DA reduced order methods tested gave almost similar reliable results. In terms of computational cost, note that the CPU time of all the ROM tested is at least three orders of magnitude lower than the CPU time of the DNS-FEM, thus proving their computational efficiency.

Of particular interest is also the comparison of the G-ROM and the grad-div-ROM. For these methods, the initial velocity condition at $t=5\,\rm{s}$ is taken as the $L^2$-orthogonal projection of the DNS solution onto ${\cal V}^l$. The rest of the numerical setup is the same as for the DA reduced order methods tested. In the following numerical experiments, we notice that, whereas the G-ROM solution is totally inaccurate, the application of the grad-div stabilization term already helps to improve the G-ROM solution, allowing to compute a solution with reasonable accuracy, especially at Reynolds number $Re=100$. However, for the higher Reynolds number $Re=1000$, both DA reduced order methods tested outperform both G-ROM and grad-div-ROM, especially for large values of the nudging parameter, thus supporting the performed numerical analysis, in which we do not need to assume at all an upper bound on the nudging parameter. In these case, the grad-div-ROM should be combined with convection stabilization (e.g., SUPG \cite{Brooks82} or LPS \cite{IMAJNA, BB01, Chacon13, lps_nos}) in order to obtain more accurate results, but this falls outside the scope of the present work. Nevertheless, up to our knowledge, this is the first time that the grad-div-ROM is numerically investigated as itself. Indeed, although the grad-div stabilization term has been already considered e.g. in \cite{BergmannIollo09, IliescuJohn14} within a ROM framework, actually in \cite{BergmannIollo09} it has been embedded within a residual-based VMS \cite{ARCME, Bazilevs07a} method, thus making difficult to understand its real contribution, while in \cite{IliescuJohn14} it has been neglected in the numerical studies. However, we found convenient to add it to the G-ROM in the present numerical experiments. Indeed, this term generally provides improvement of local discrete mass conservation \cite{Linke09, Lube09}, and thus it is particularly important in the present framework, in which mixed interpolations that sa\-tisfy the inf-sup condition but are not exactly divergence-free have been used to compute the snapshots. This allows to work with only velocity ROM, as in this case, since the POD velocity modes are solenoidal and the pressure term drops out, but could lead to a poor resolution, as the G-ROM results confirm. We emphasize again that when considering DA into the ROM, thus adding or not the grad-div stabilization term makes no significant difference, as showed in the following numerical experiments.


\subsection{Case $Re=100$}\label{sec:Re=100}

In this section, we discuss results for $Re=100$. In this case, we have used the computational grid represented in Figure \ref{fig:Mesh} on top to compute the snapshots, for which $h=2.76\times 10^{-2}\,\rm{m}$, resulting in $32\,488$ d.o.f. for velocities and $4\,151$ d.o.f. for pressure. Also, $166$ snapshots were collected, which comprise one full period from $t=5\,\rm{s}$ to $t=5.332\,\rm{s}$. All tested ROM have been run in the stable response time interval $[5,7]\,\rm{s}$, corresponding to six periods for the lift coefficient. Thus, we are actually testing the ability of the considered ROM to predict/extrapolate in time, monitoring their performance over a six times larger time interval with respect to the one used to compute the snapshots and generate the POD modes. This will show how the strategy to incorporate DA into the ROM can provide long time stability and accuracy, thus proving its robustness.

Numerical results for energy, drag and lift predictions using $l=8$ modes are shown in Figures \ref{fig:QOI}, \ref{fig:QOIda}, \ref{fig:QOIgdDA}. In particular, Figure \ref{fig:QOI} shows a comparison within DNS, G-ROM, grad-div-ROM with $\mu=0.15$, DA-ROM with $\beta=10$, and grad-div-DA-ROM with $\mu=0.15$ and $\beta=10$. The value $\mu=0.15$ for the grad-div stabilization term has been fixed minimizing the error with respect to the DNS energy. From this figure, we observe that, whereas the G-ROM solution is totally inaccurate, the application of the grad-div stabilization term greatly improves the G-ROM solution, allowing to compute rather accurate quantities of interest. Indeed, the temporal evolution of the kinetic energy and lift coefficient is very close to that of the DNS, being the drag coefficient temporal evolution the most sensitive quantity presenting larger differences. A slight improvement is observed for using DA with $\beta=10$, being results for DA-ROM and grad-div-DA-ROM almost identical. Note that using DA, since we started from zero initial velocity conditions, the DNS results are approached around $t=5.4\,\rm{s}$ with $\beta=10$. 

A significant improvement is observed by increasing the nudging parameter $\beta$ for DA reduced order methods. This is clearly displayed in Figures \ref{fig:QOIda}, \ref{fig:QOIgdDA}, which respectively show the behavior of the DA-ROM and the grad-div-DA-ROM, varying the nudging parameter $\beta$ from $10$ to $500$. Again, almost identical results are obtained with both DA reduced order methods, for which the best predictions are given by the largest values $\beta=500$ of the nudging parameter, although we observe a similar accuracy already for $\beta=100$. Note also that for large values of the nudging parameter ($\beta=100, 500$), although we started from zero initial velocity conditions, the DNS results are approached with a rather accurate resolution just after very few iterations (around $20$, i.e. $0.04\,\rm{s}$, for $\beta=100$ and $5$, i.e. $0.01\,\rm{s}$, for $\beta=500$). All these results are also confirmed by Table \ref{tab:ErrLevComp}, where we display the error levels with respect to DNS of maximum kinetic energy $|E_{kin,l}^{max}-E_{kin,DNS}^{max}|$, maximum drag coefficient $|c_{D,l}^{max}-c_{D,DNS}^{max}|$, maximum lift coefficient $|c_{L,l}^{max}-c_{L,DNS}^{max}|$, and velocity norm $\nor{\uv_{l}-\uv_{DNS}}{\ell^2({\bf L}^2)}$ using $l=8$ modes for G-ROM, grad-div-ROM ($\mu=0.15$), DA-ROM ($\beta=500$), and grad-div-DA-ROM ($\mu=0.15, \beta=500$) in the time interval $[5.01,7]\,\rm{s}$. Note how grad-div-ROM already reduces the error level in $E_{kin}^{max}$ of three orders of magnitude with respect to G-ROM, similarly to both DA reduced order methods, and in $c_{L}^{max}$ of one order of magnitude, while both DA reduced order methods of two orders of magnitude. However, for $c_{D}^{max}$, while grad-div-ROM slightly reduces the error level with respect to G-ROM (five times), both DA reduced order methods guarantee again a reduction of two orders of magnitude. In terms of $\ell^2({\bf L}^2)$ velocity norm, both DA reduced order methods reduces the G-ROM error level eight times, while the grad-div-ROM is just slightly better accurate than G-ROM.

\begin{figure}[htb]
\begin{center}
\centerline{\includegraphics[width=5.in]{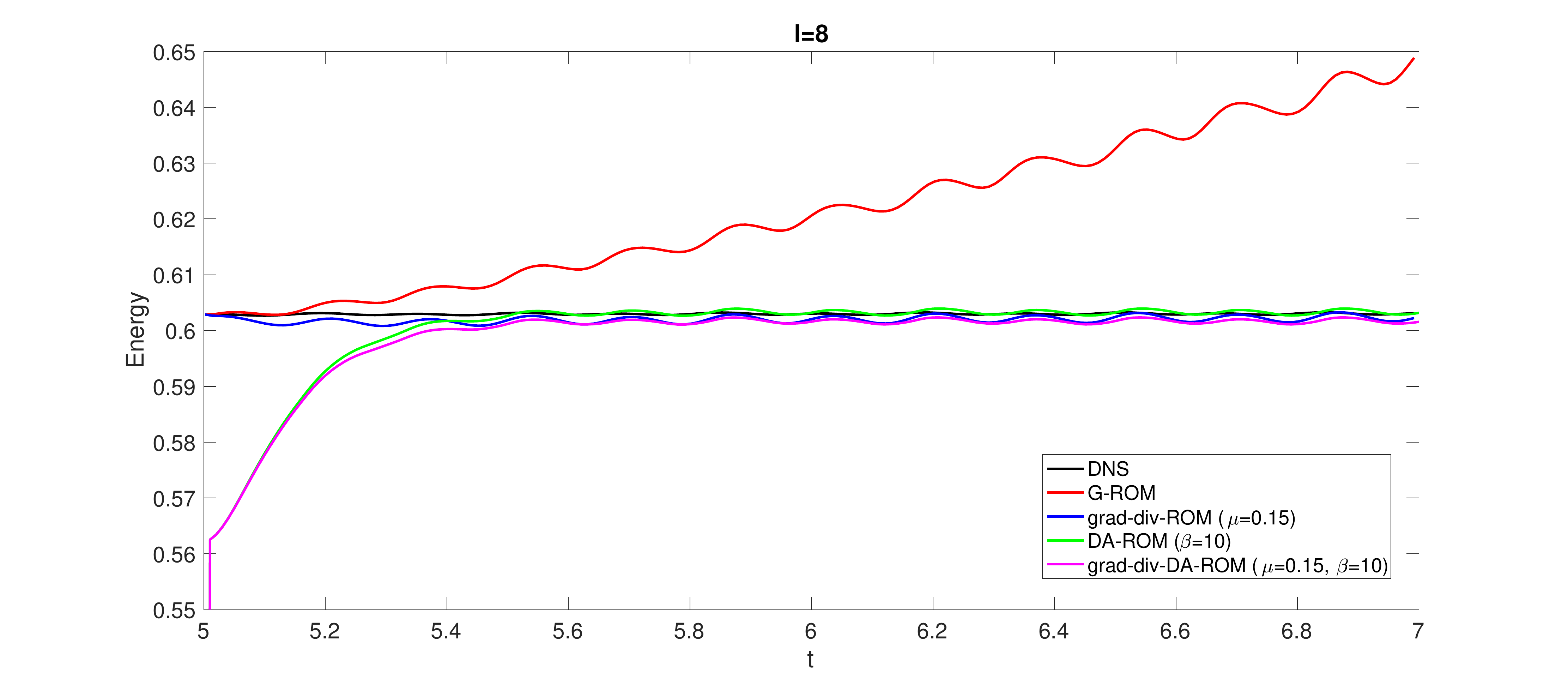}}
\centerline{\includegraphics[width=5.in]{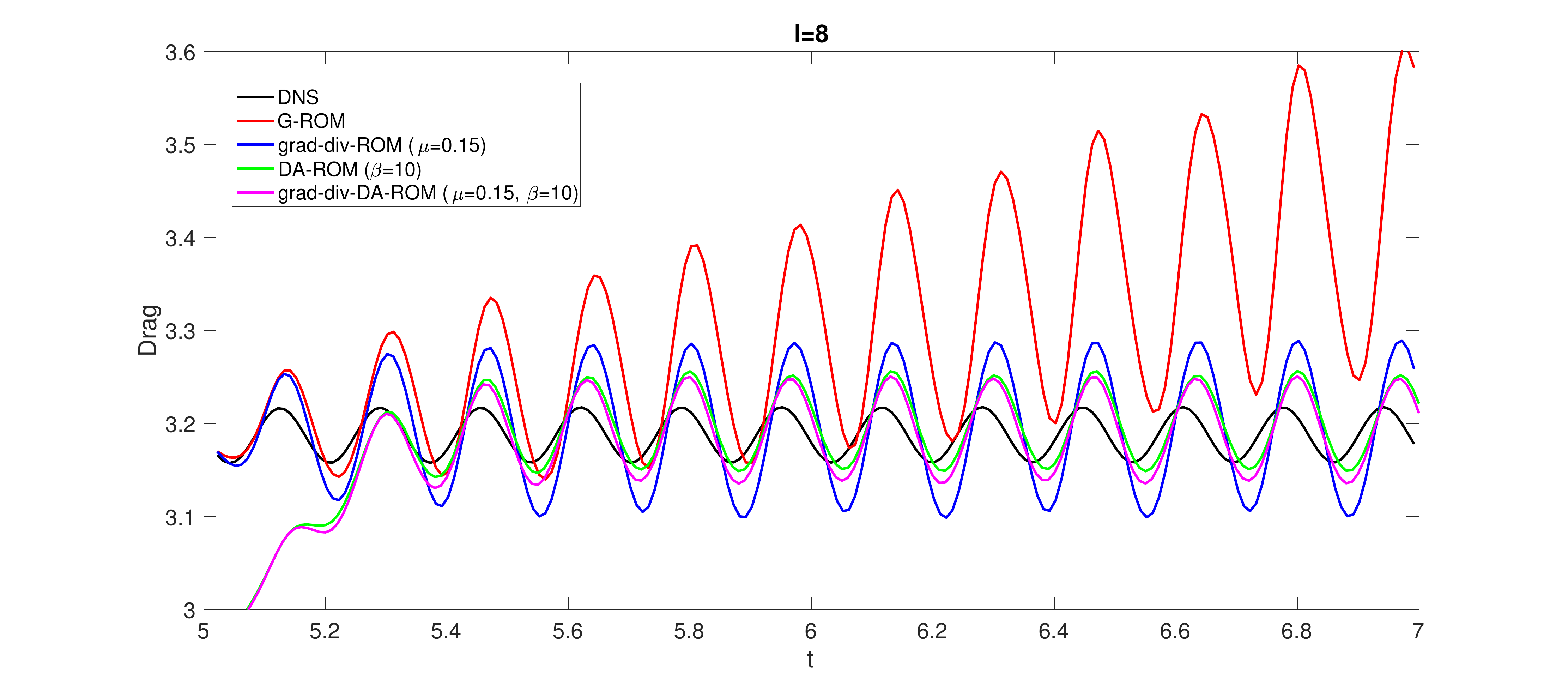}}
\centerline{\includegraphics[width=5.in]{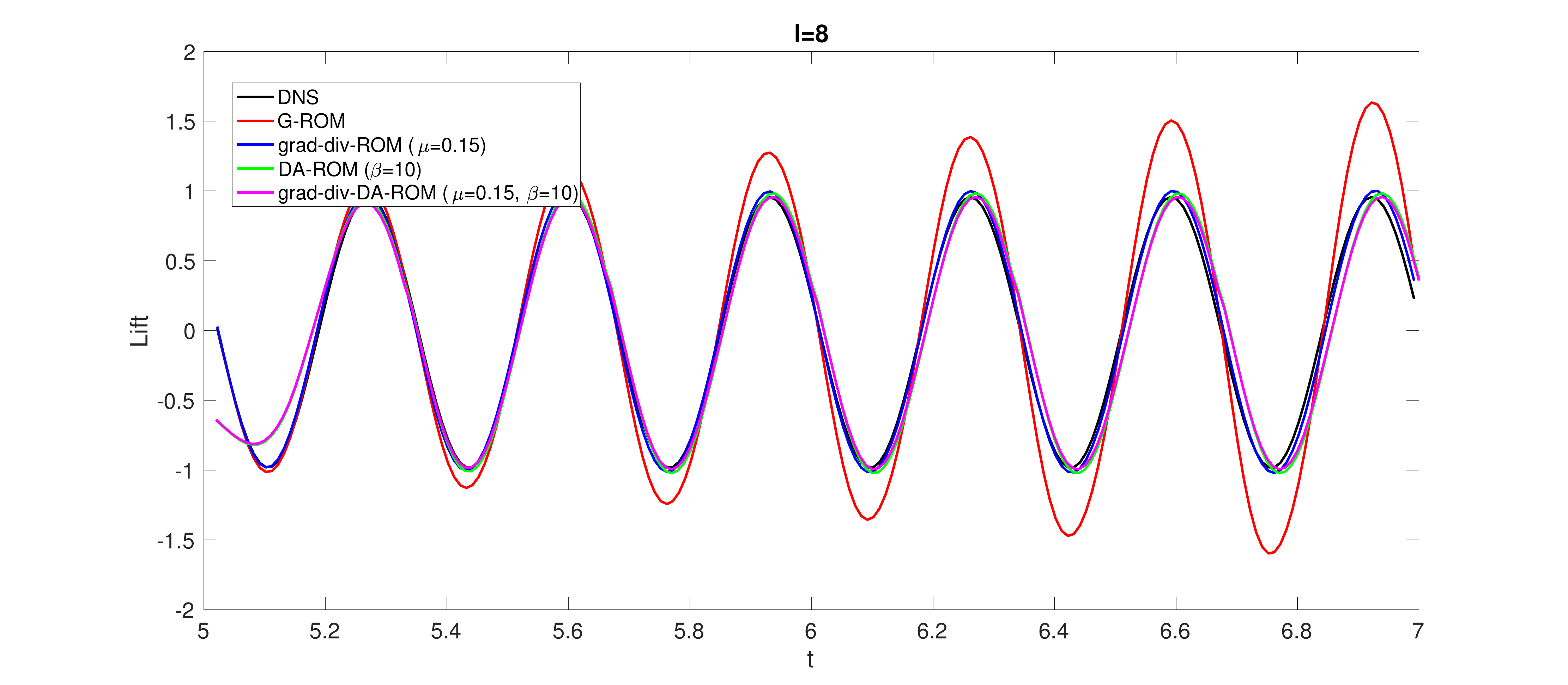}}
\caption{Example \ref{sec:Re=100} (Case $Re=100$): Temporal evolution of kinetic energy, drag coefficient and lift coefficient using $l=8$ modes ($166$ snapshots used, which comprise one full period from $t=5\,\rm{s}$ to $t=5.332\,\rm{s}$).}\label{fig:QOI}
\end{center}
\end{figure}

\begin{figure}[htb]
\begin{center}
\centerline{\includegraphics[width=5.in]{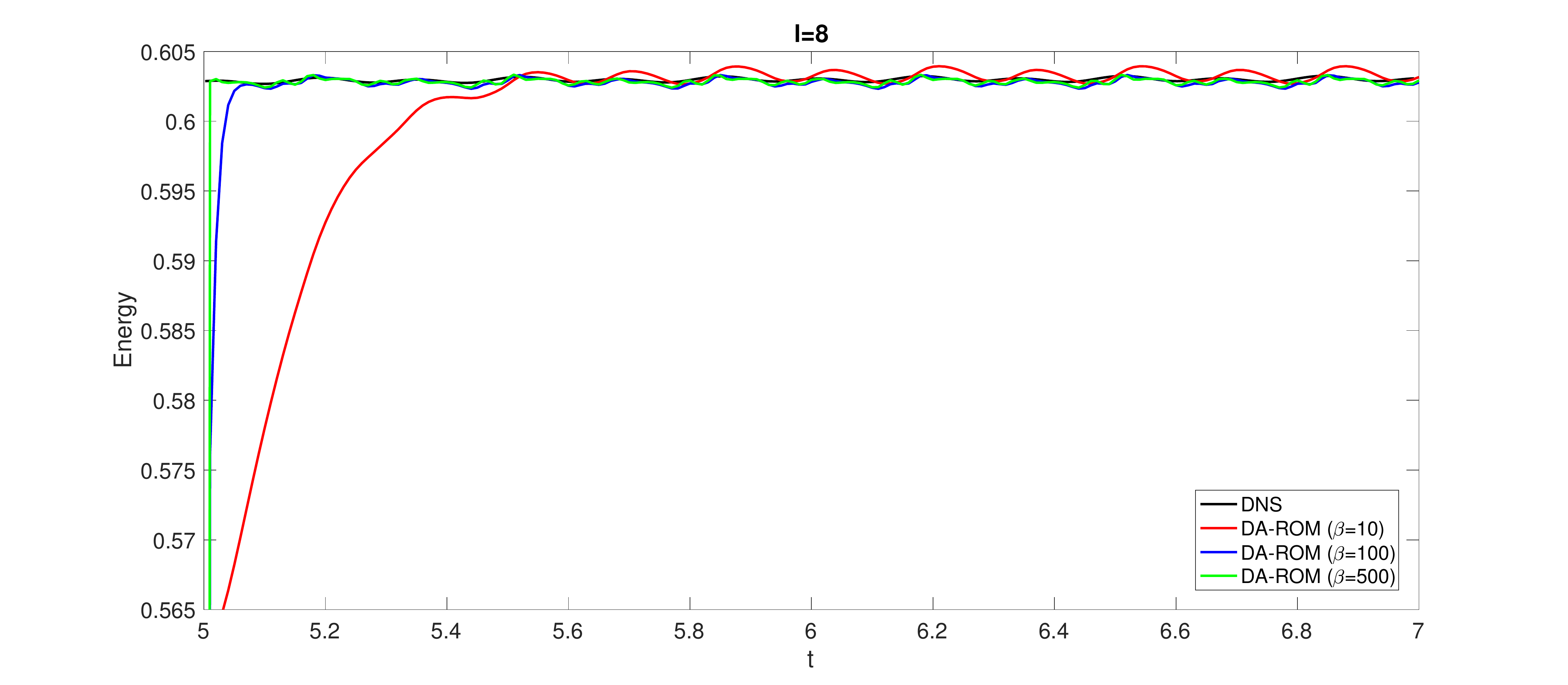}}
\centerline{\includegraphics[width=5.in]{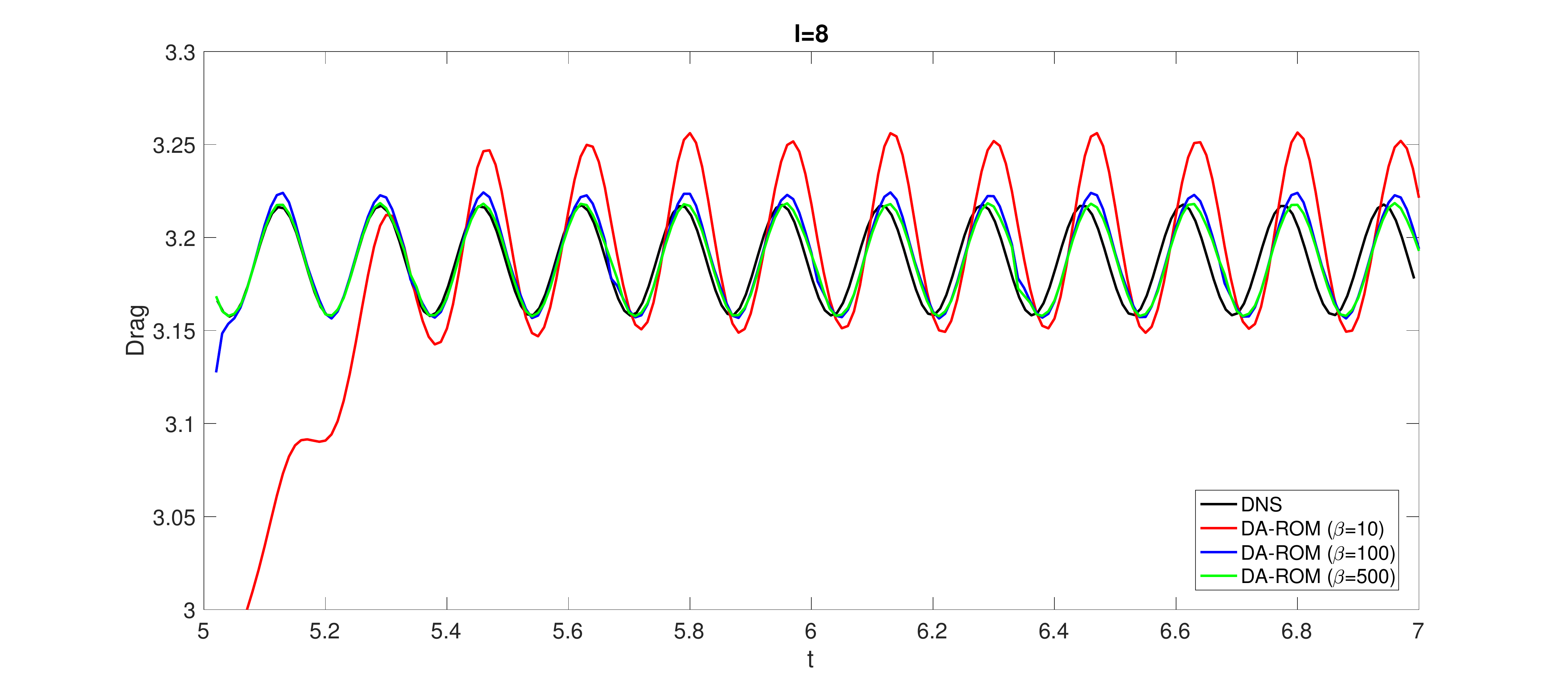}}
\centerline{\includegraphics[width=5.in]{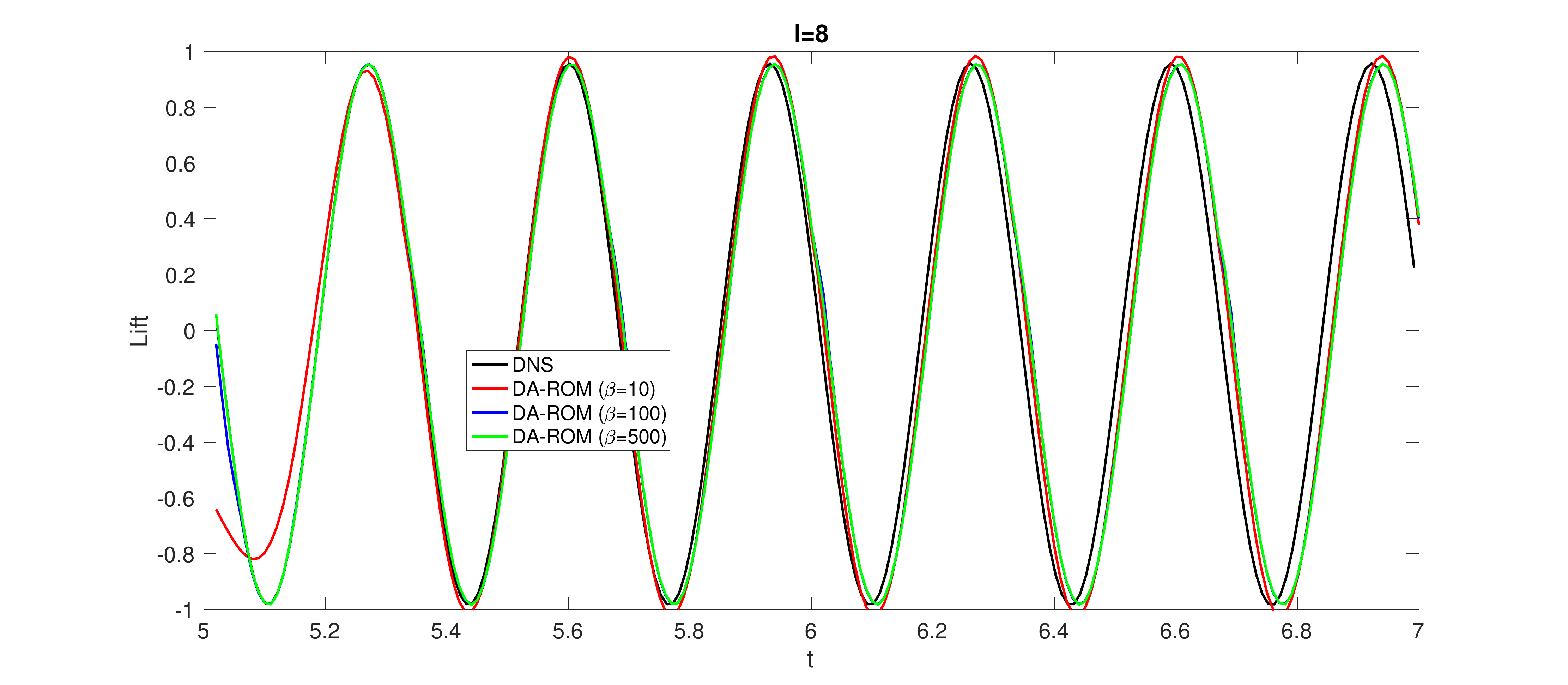}}
\caption{Example \ref{sec:Re=100} (Case $Re=100$): Temporal evolution of kinetic energy, drag coefficient and lift coefficient using $l=8$ modes for DA-ROM with $\beta=10,\,100,\,500$ ($166$ snapshots used, which comprise one full period from $t=5\,\rm{s}$ to $t=5.332\,\rm{s}$).}\label{fig:QOIda}
\end{center}
\end{figure}

\begin{figure}[htb]
\begin{center}
\centerline{\includegraphics[width=5.in]{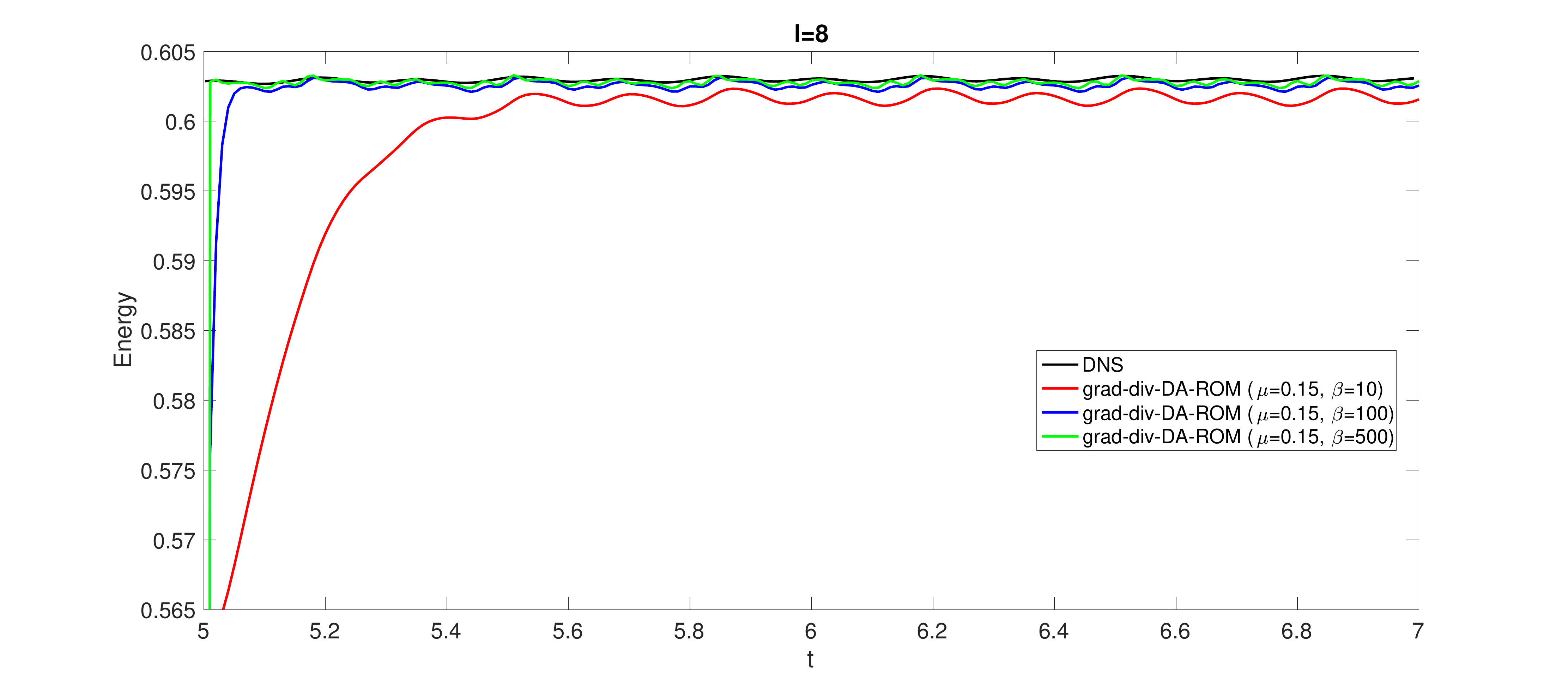}}
\centerline{\includegraphics[width=5.in]{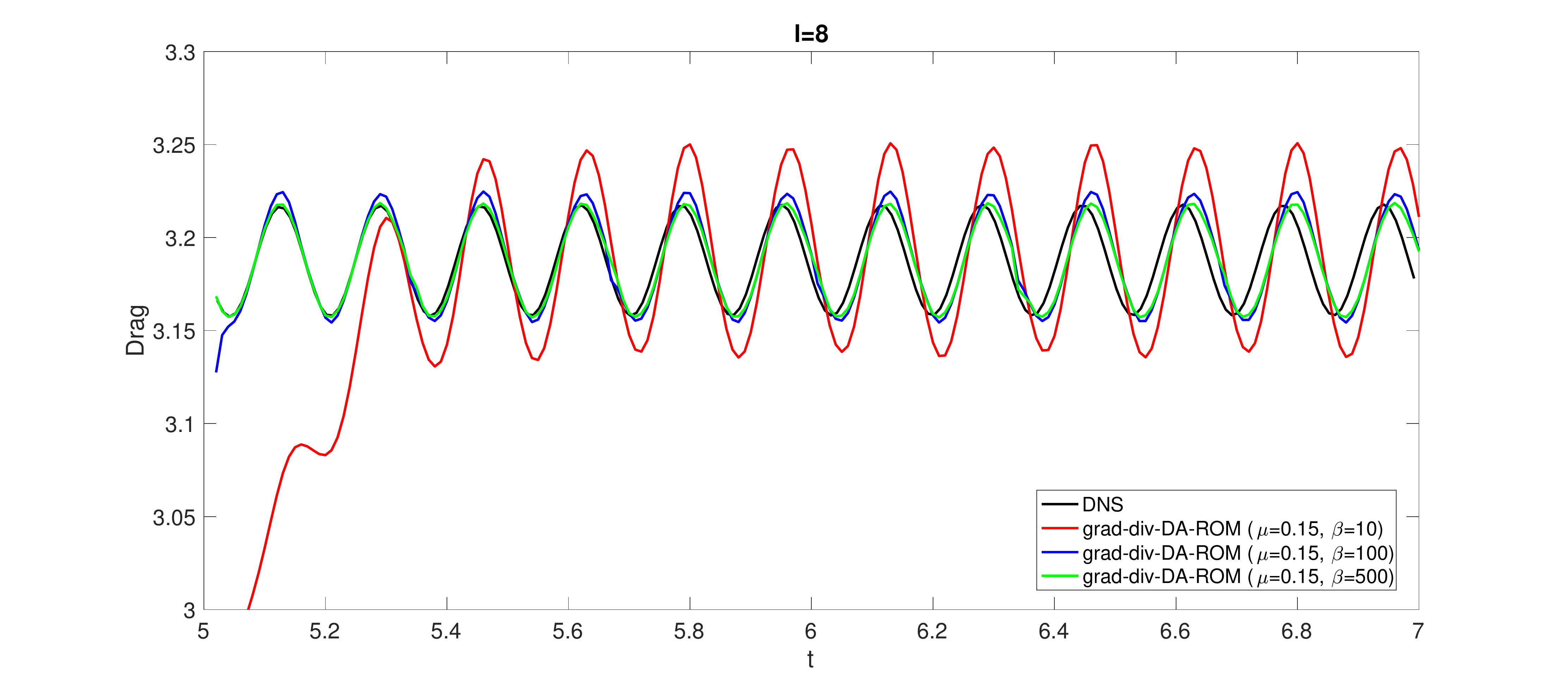}}
\centerline{\includegraphics[width=5.in]{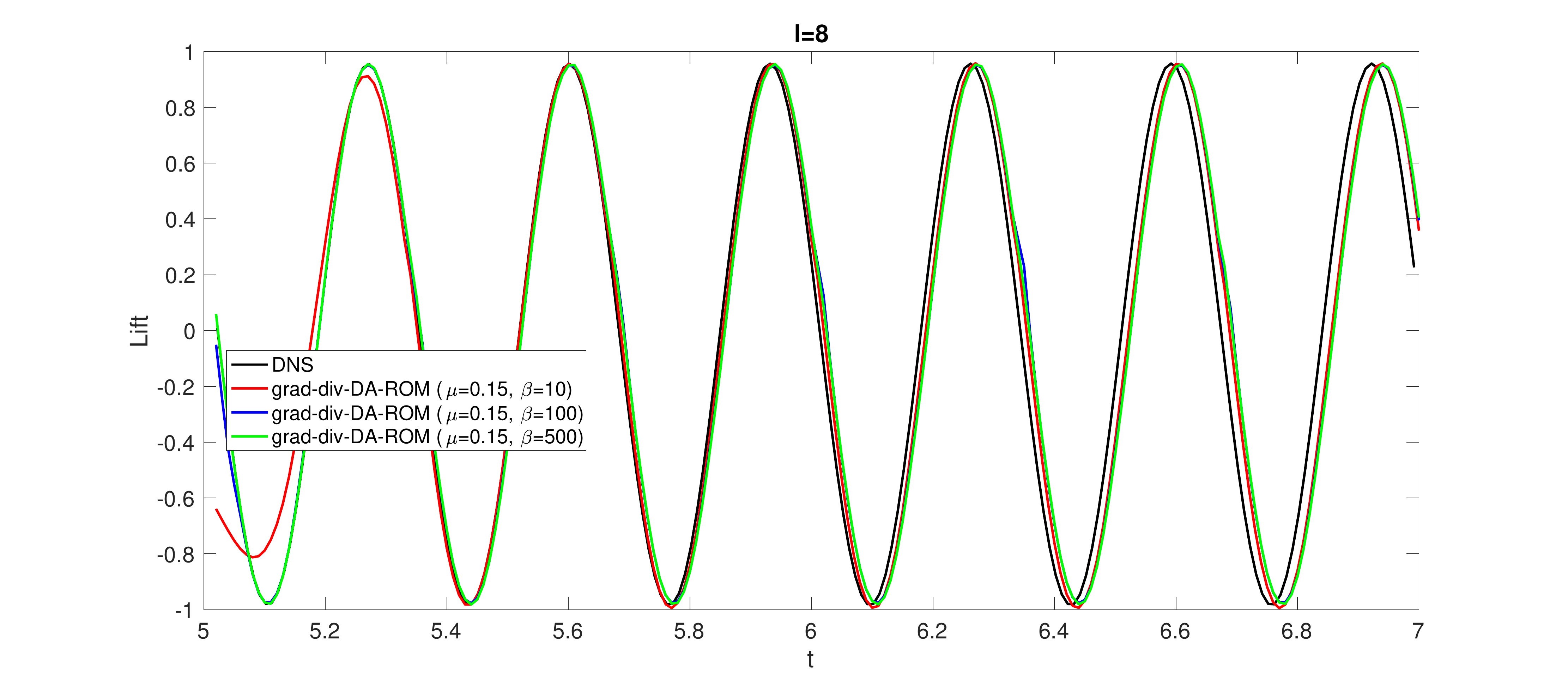}}
\caption{Example \ref{sec:Re=100} (Case $Re=100$): Temporal evolution of kinetic energy, drag coefficient and lift coefficient using $l=8$ modes for grad-div-DA-ROM with $\mu=0.15$ and $\beta=10,\,100,\,500$ ($166$ snapshots used, which comprise one full period from $t=5\,\rm{s}$ to $t=5.332\,\rm{s}$).}\label{fig:QOIgdDA}
\end{center}
\end{figure}

\begin{table}[htb]
$$\hspace{-0.1cm}
\begin{tabular}{|c|c|c|c|c|}
\hline
 & \multicolumn{4}{|c|}{$Re=100$}\\
\hline
Errors & G-ROM & grad-div-ROM & DA-ROM & grad-div-DA-ROM\\
\hline
$E_{kin}^{max}$ & 4.56e-02 & 1.60e-05 & 8.20e-05 & 4.30e-05\\
\hline
$c_{D}^{max}$ & 3.84e-01 & 7.15e-02 & 2.72e-03 & 2.87e-03\\
\hline
$c_{L}^{max}$ & 6.78e-01 & 4.33e-02 & 4.18e-03 & 4.76e-03\\
\hline
$\ell^2({\bf L}^2)\, \uv$ norm & 1.68e-01 & 9.73e-02 & 2.04e-02 & 2.04e-02\\
\hline
\end{tabular}$$\caption{Example \ref{sec:Re=100} (Case $Re=100$): Errors levels with respect to DNS for G-ROM, grad-div-ROM ($\mu=0.15$), DA-ROM ($\beta=500$), and grad-div-DA-ROM ($\mu=0.15, \beta=500$) ($166$ snapshots used, which comprise one full period from $t=5\,\rm{s}$ to $t=5.332\,\rm{s}$).}\label{tab:ErrLevComp}
\end{table}

We also investigate the considered ROM performances in predicting quantities of interest when inaccurate snapshots ($64\%$ of one full period) are used in their construction. Thus, we generate inaccurate snapshots using $64\%$ of one full period of DNS data, which corresponds in this case to the first $106$ DNS time step solutions from $t=5\,\rm{s}$ to $t=5.212\,\rm{s}$. Figure \ref{fig:PODvelmodes} displays the Euclidean norm of the first POD velocity modes obtained with the full set of snapshots (left) and the inaccurate set of snapshots (right). Results for the considered ROM using $l=8$ modes in this case are shown in Figures \ref{fig:QOIInac1}, \ref{fig:QOIdaInac1}, \ref{fig:QOIgdDAInac1}. Similar to the previous results, DA significantly improves the accuracy of the G-ROM, especially for large values of the nudging parameter, without the need to increase the number of reduced basis functions. While results for G-ROM becomes more and more inaccurate as time goes on, results for grad-div-ROM remain still acceptable if compared with DA reduced order methods for a small value of the nudging parameter. Again, results for both DA-ROM (with and without grad-div term) are very close and almost approaches DNS results for large values of the nudging parameter. Actually, they are almost comparable to previous results for one full period of DNS data. All these considerations are also reflected by the error levels displayed in Table \ref{tab:ErrLevCompInac1}. These results suggest that, despite its simple implementation, DA can greatly improve the overall accuracy of the standard G-ROM in the computation of quantities of interest even when low-resolution data are available to construct the reduced basis, which is common in practice, whereas grad-div stabilization (without DA) continues providing reliable results. We notice, however, that as the Reynolds number is increased (see next section), results for grad-div-ROM (without DA) are less accurate, and maybe it should be combined with convection stabilization if one does not use DA in order to obtain more accurate results.

\begin{figure}[htb]
\begin{center}
\includegraphics[width=2.385in]{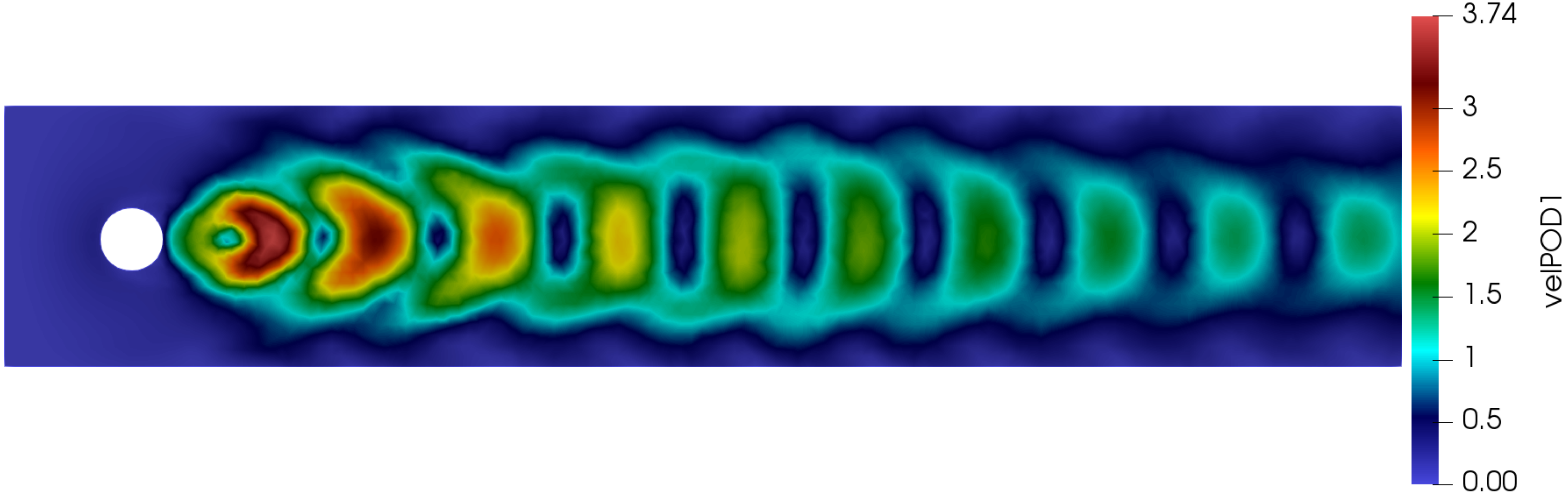}\hspace{-0.1cm}
\includegraphics[width=2.385in]{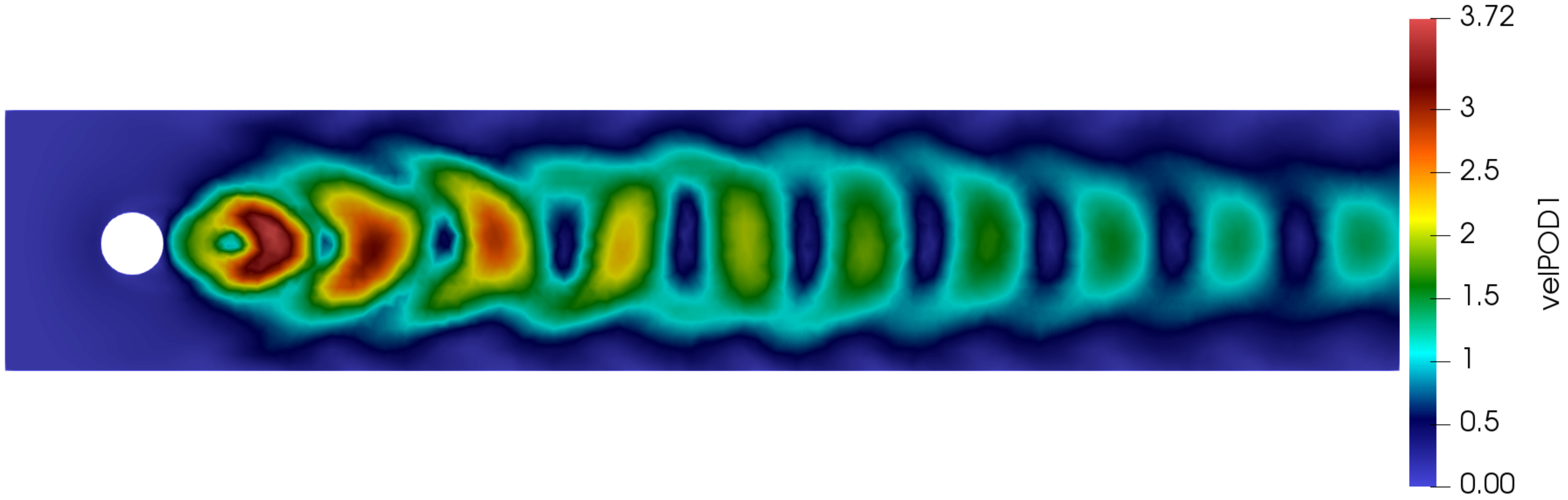}\\
\includegraphics[width=2.385in]{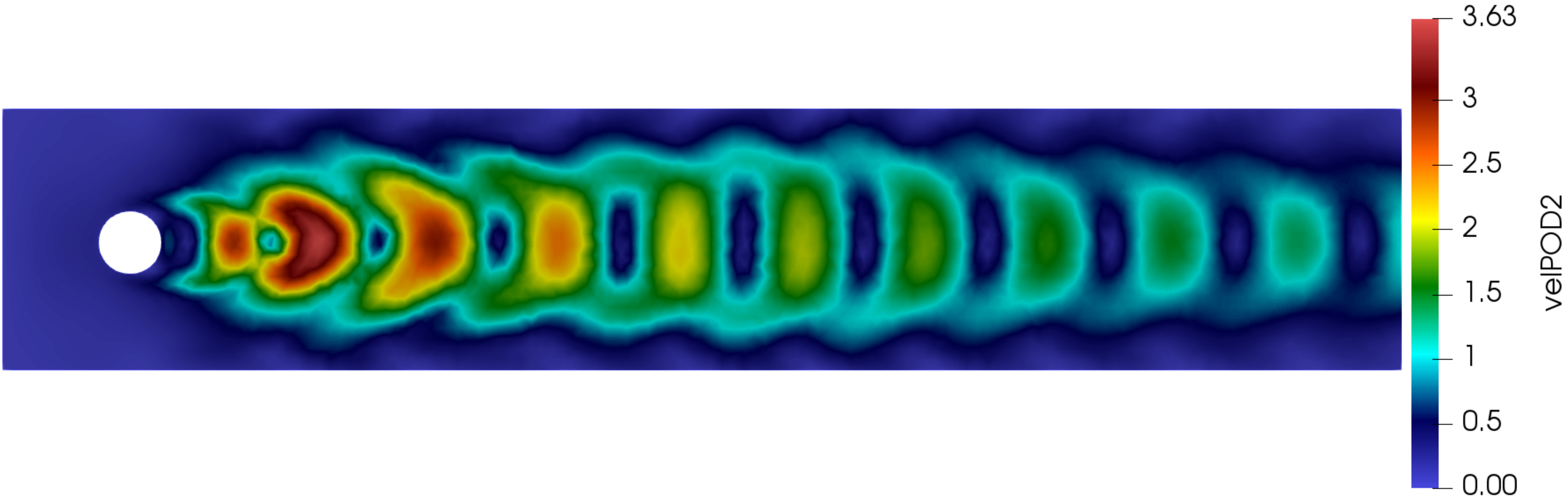}\hspace{-0.1cm}
\includegraphics[width=2.385in]{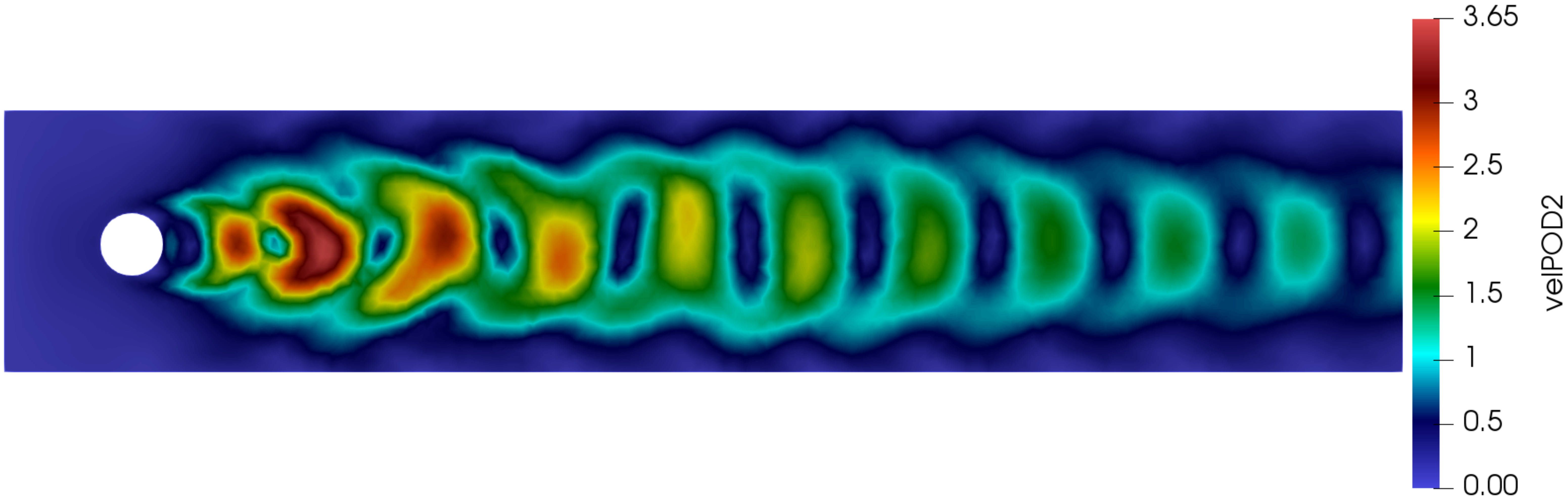}\\
\includegraphics[width=2.385in]{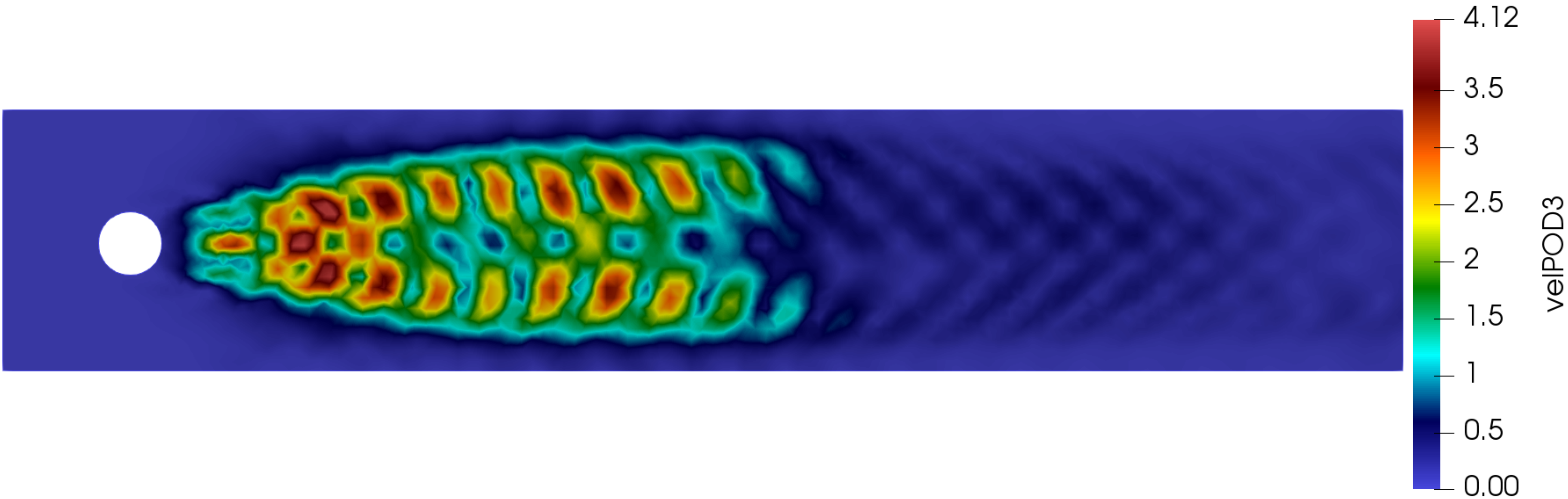}\hspace{-0.1cm}
\includegraphics[width=2.385in]{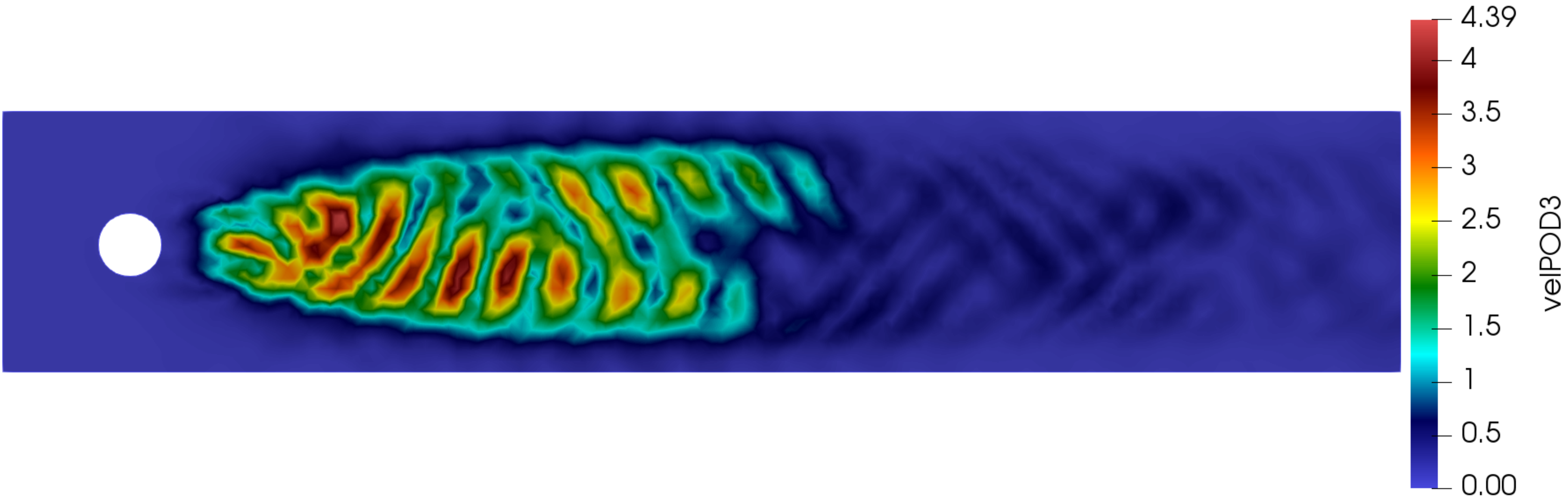}\\
\includegraphics[width=2.385in]{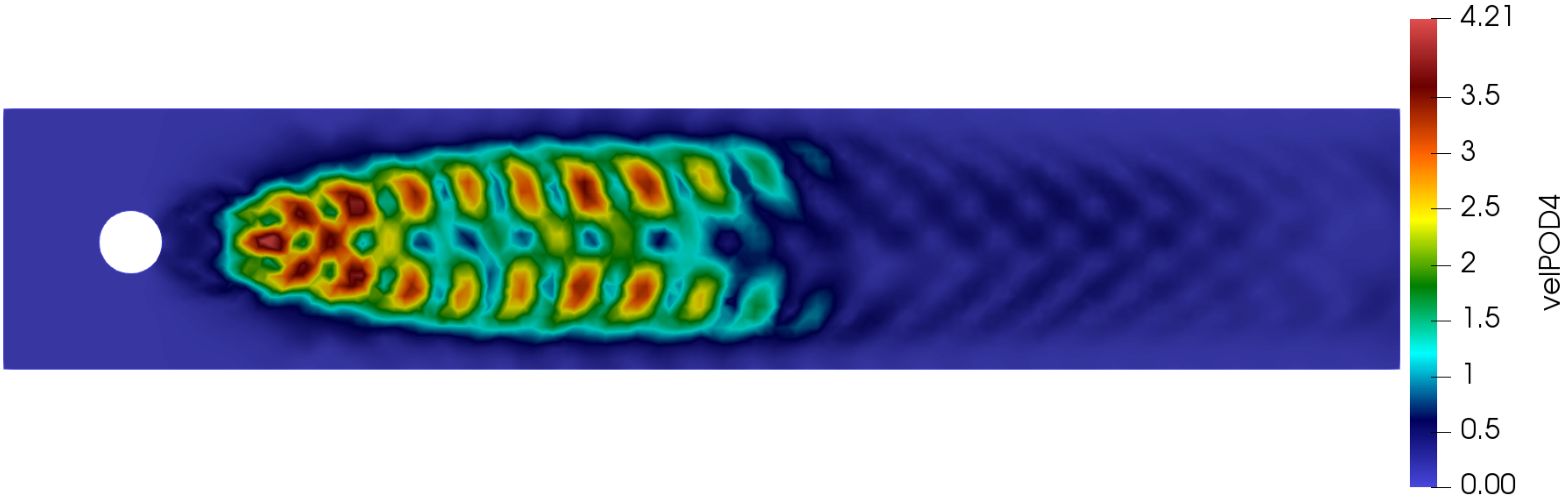}\hspace{-0.1cm}
\includegraphics[width=2.385in]{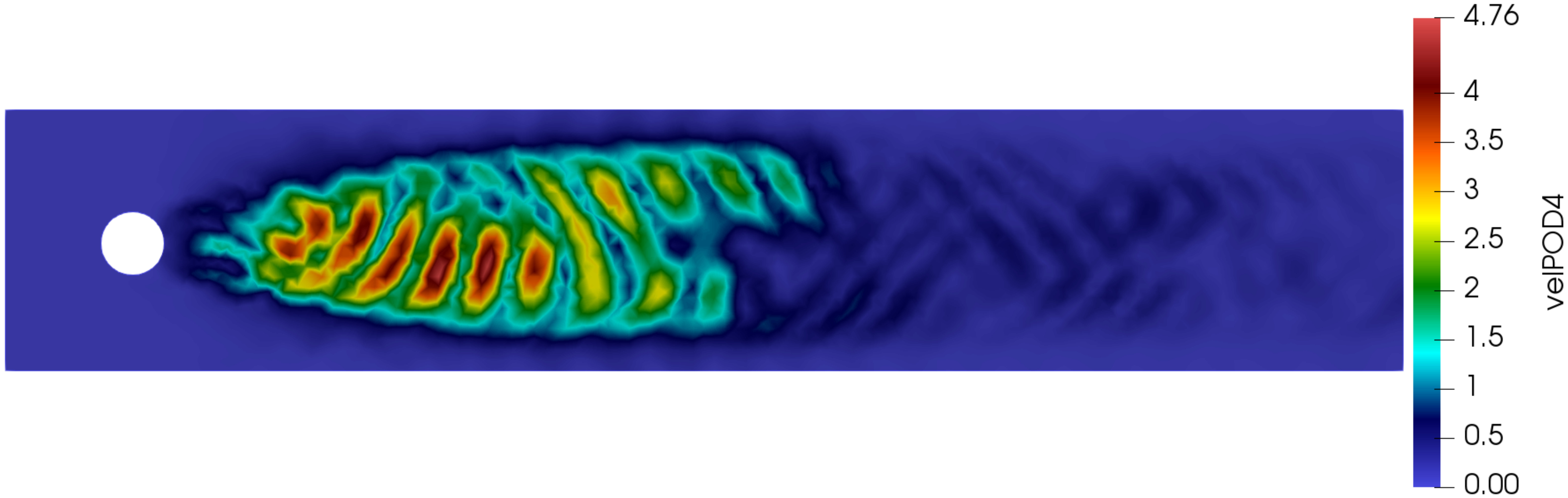}
\caption{Example \ref{sec:Re=100} (Case $Re=100$): First POD velocity modes (Euclidean norm) obtained with $166$ snapshots (full period basis, left) and $106$ snapshots (inaccurate basis corresponding to $64\%$ of one full period, right).}\label{fig:PODvelmodes}
\end{center}
\end{figure}

\begin{figure}[htb]
\begin{center}
\centerline{\includegraphics[width=5.in]{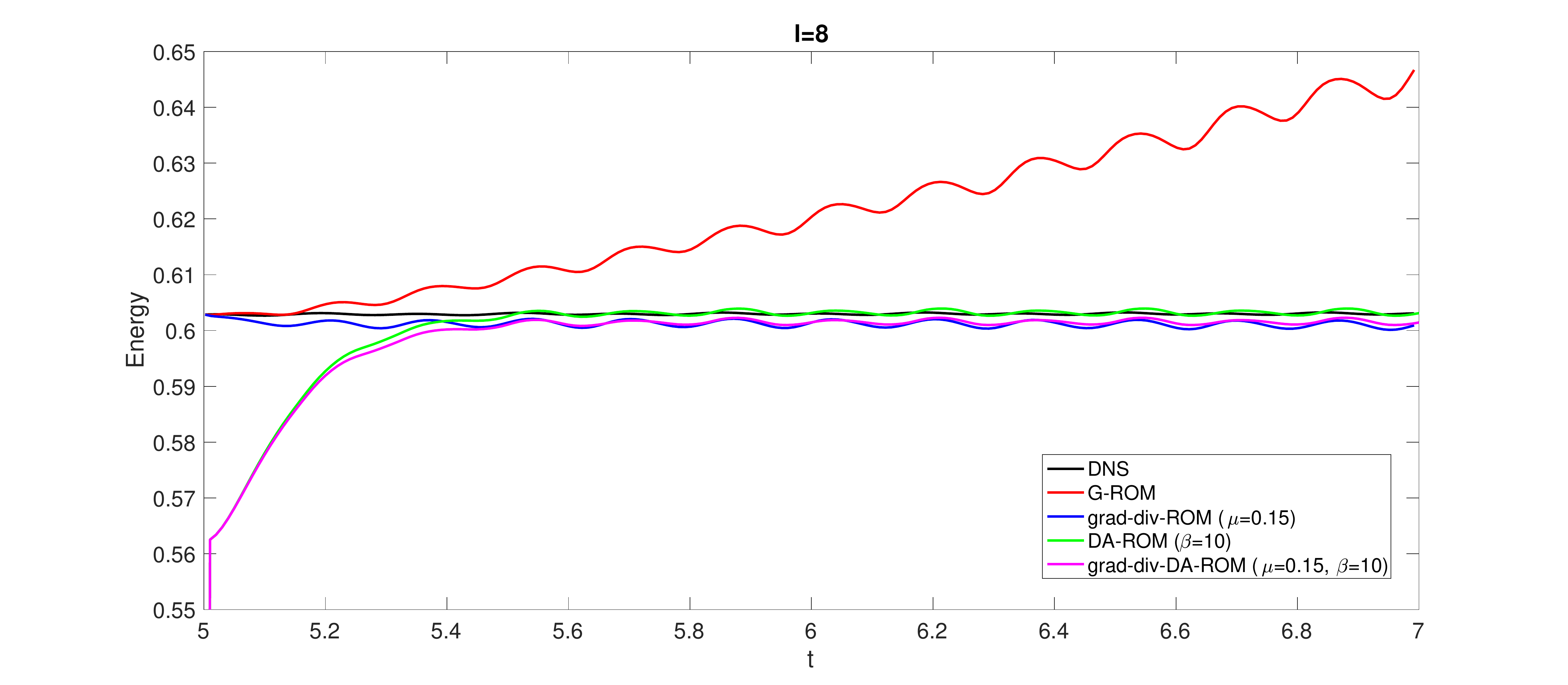}}
\centerline{\includegraphics[width=5.in]{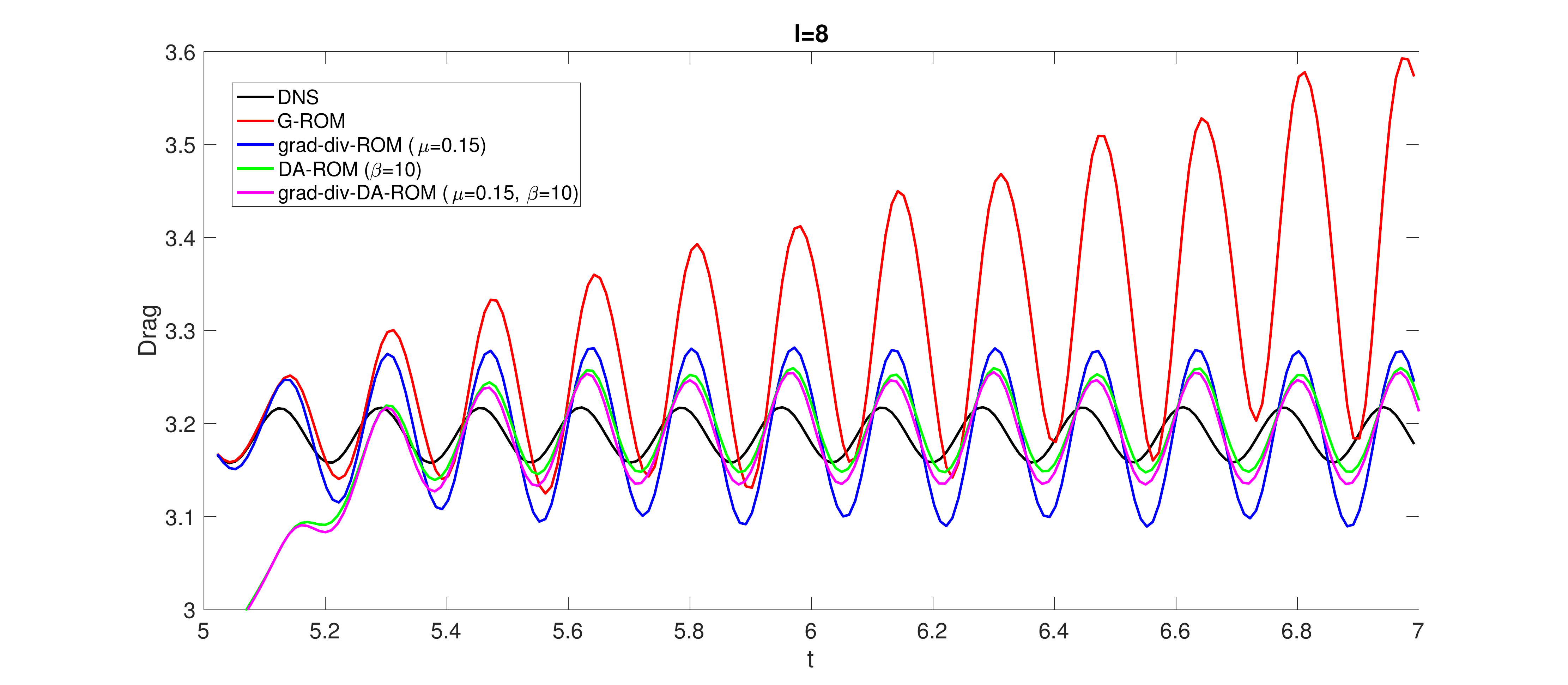}}
\centerline{\includegraphics[width=5.in]{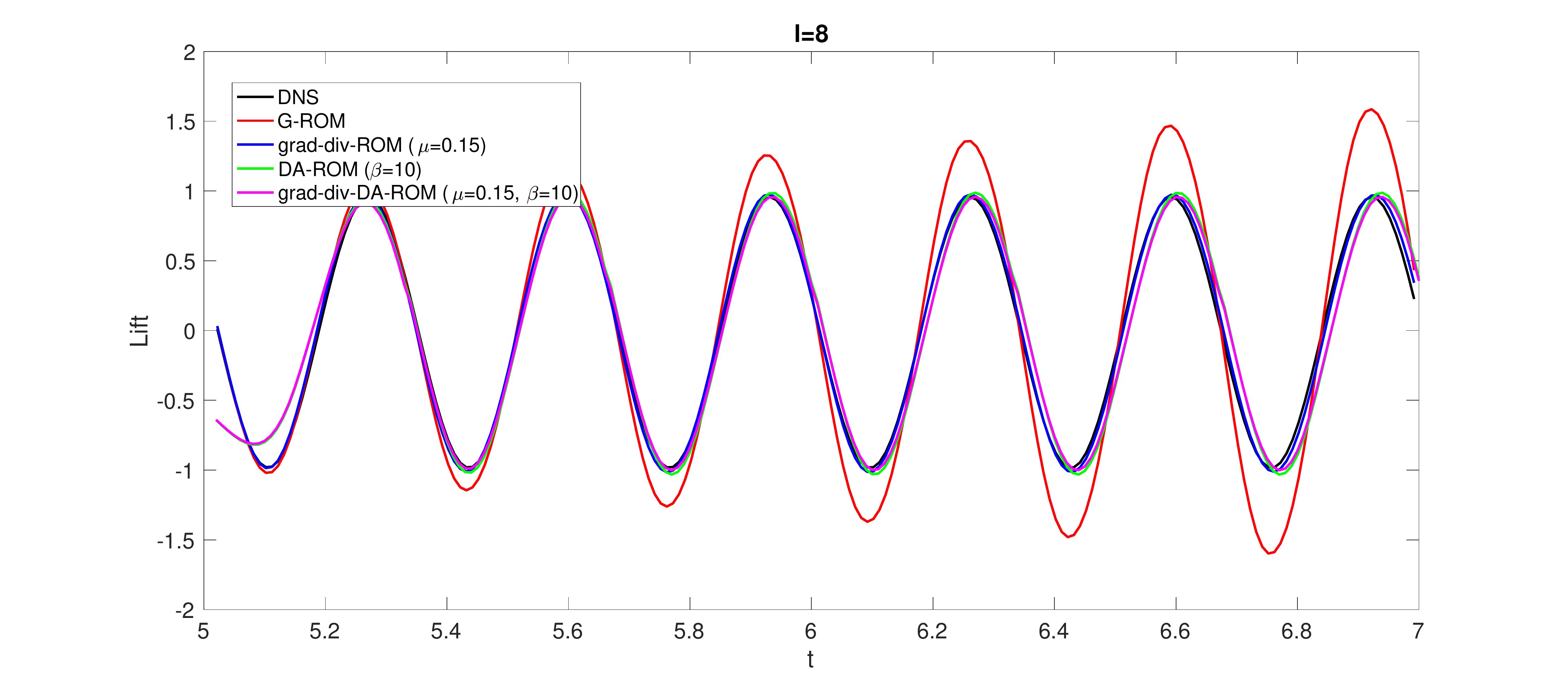}}
\caption{Example \ref{sec:Re=100} (Case $Re=100$): Temporal evolution of kinetic energy, drag coefficient and lift coefficient using $l=8$ modes ($106$ snapshots used, which comprise $64\%$ of one full period from $t=5\,\rm{s}$ to $t=5.212\,\rm{s}$).}\label{fig:QOIInac1}
\end{center}
\end{figure}

\begin{figure}[htb]
\begin{center}
\centerline{\includegraphics[width=5.in]{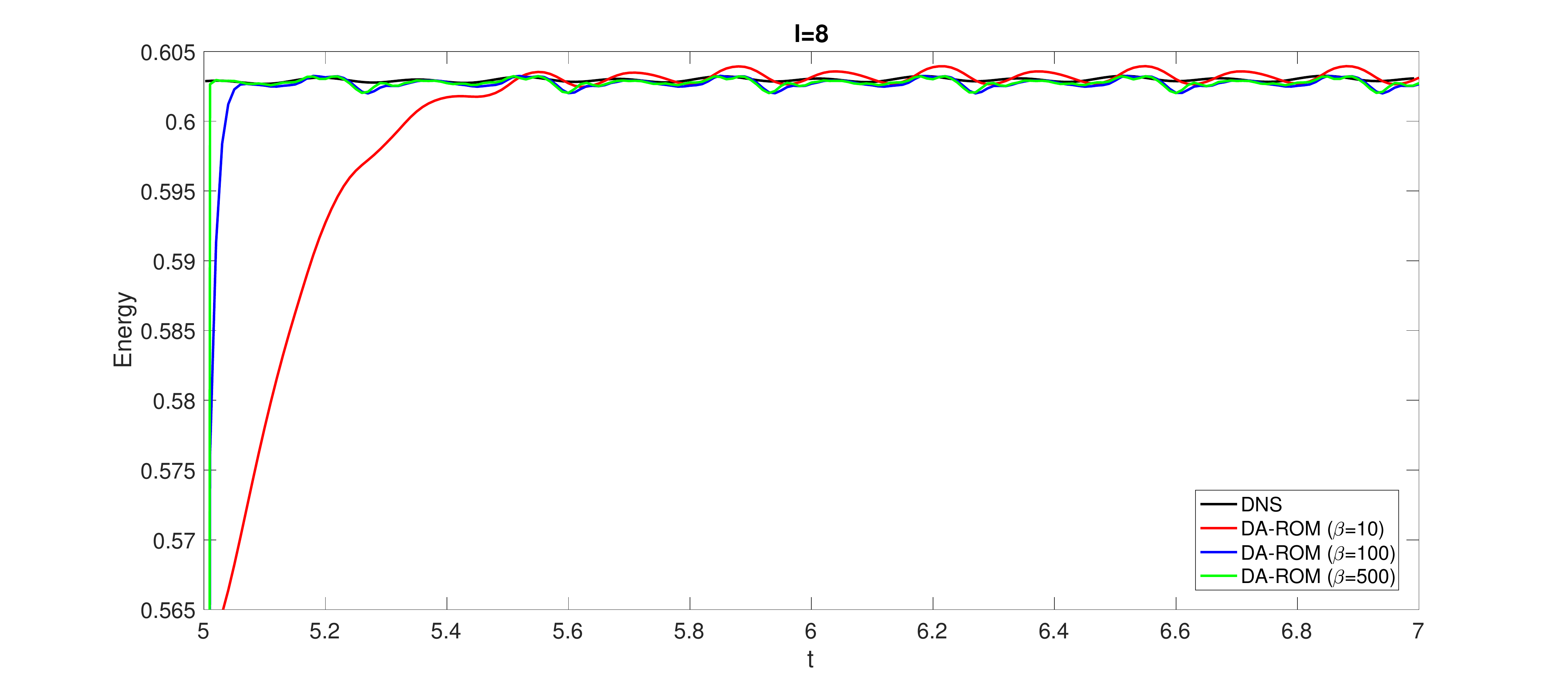}}
\centerline{\includegraphics[width=5.in]{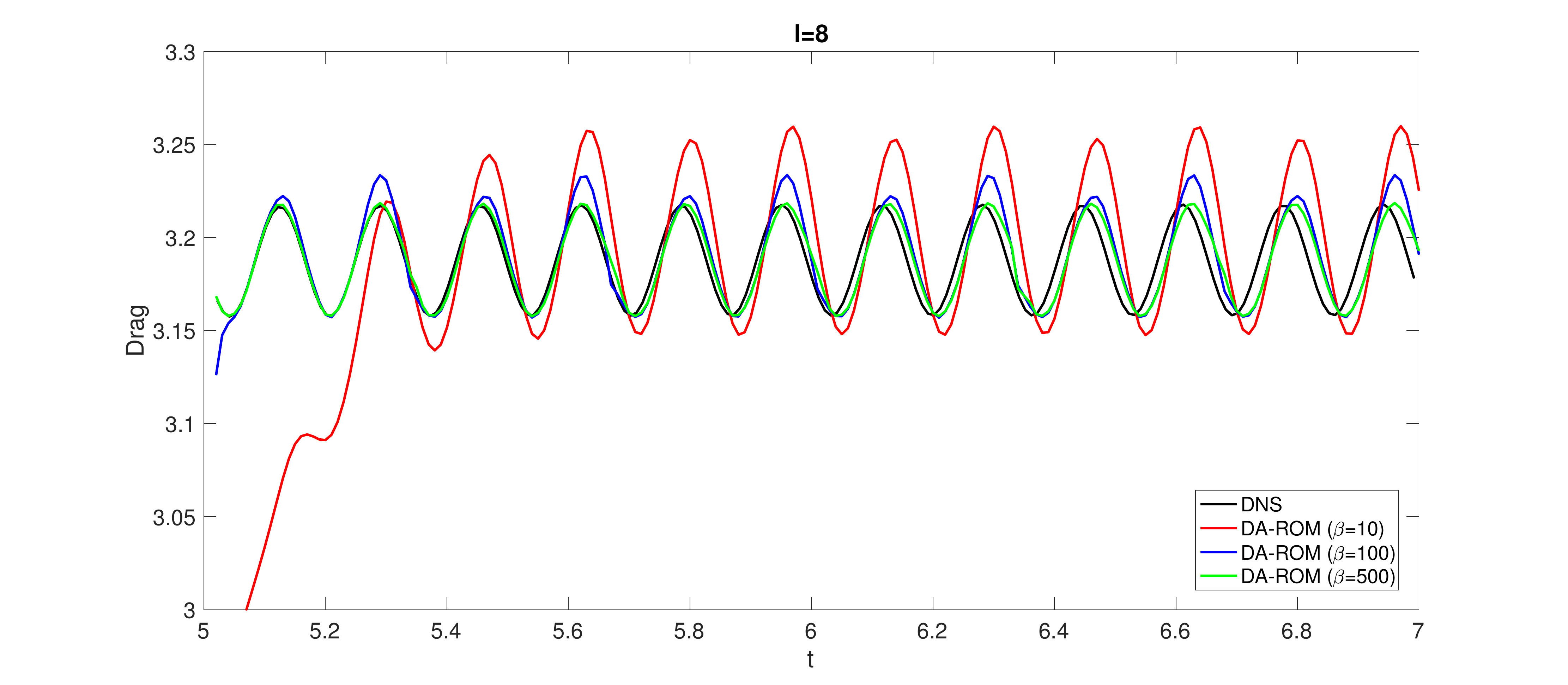}}
\centerline{\includegraphics[width=5.in]{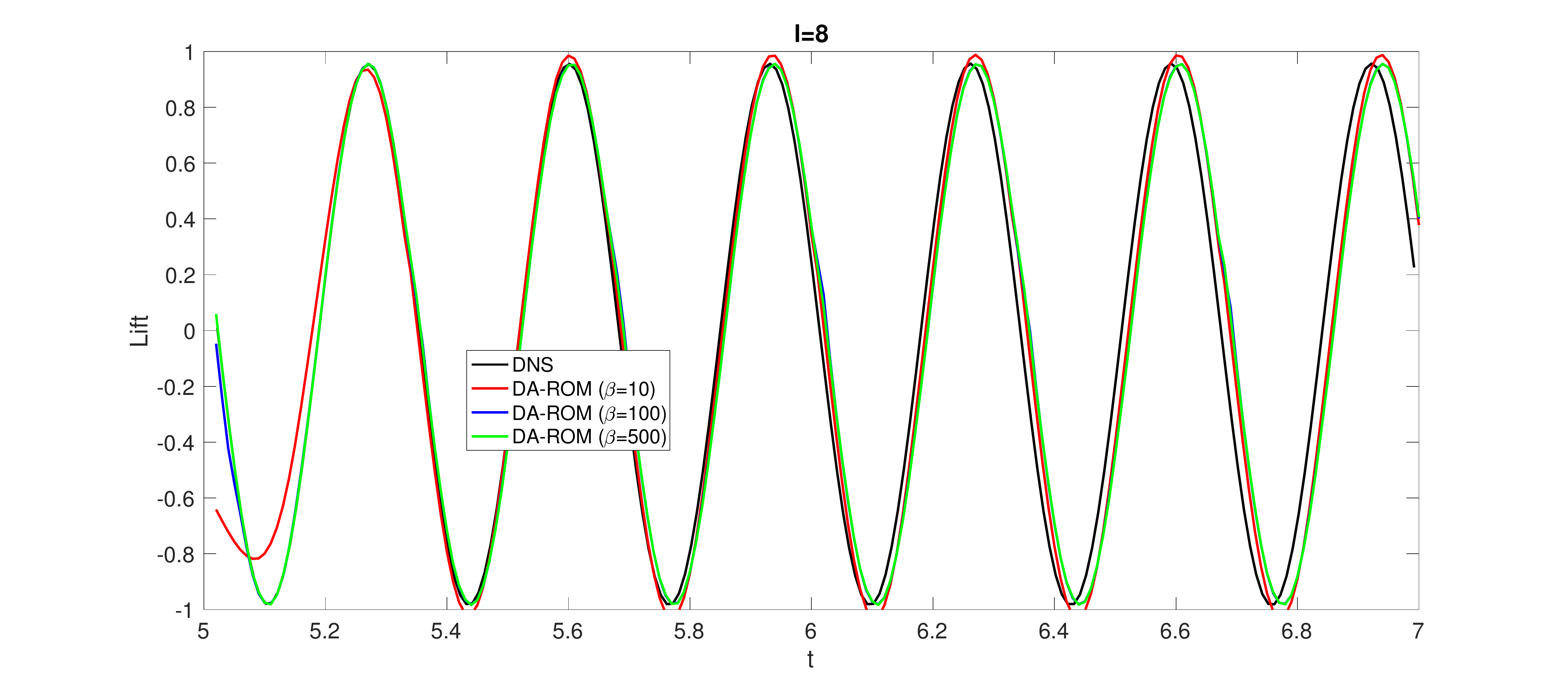}}
\caption{Example \ref{sec:Re=100} (Case $Re=100$): Temporal evolution of kinetic energy, drag coefficient and lift coefficient using $l=8$ modes for DA-ROM with $\beta=10,\,100,\,500$ ($106$ snapshots used, which comprise $64\%$ of one full period from $t=5\,\rm{s}$ to $t=5.212\,\rm{s}$).}\label{fig:QOIdaInac1}
\end{center}
\end{figure}

\begin{figure}[htb]
\begin{center}
\centerline{\includegraphics[width=5.in]{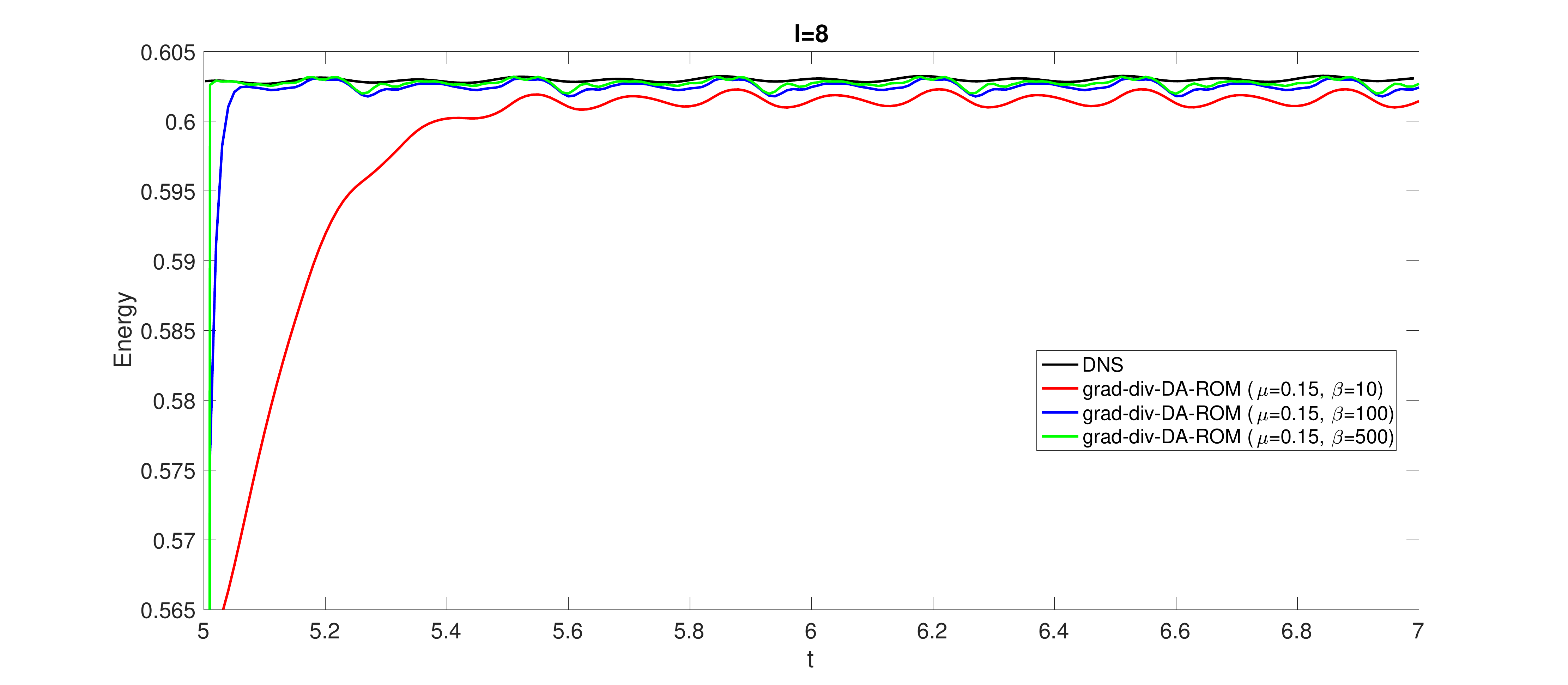}}
\centerline{\includegraphics[width=5.in]{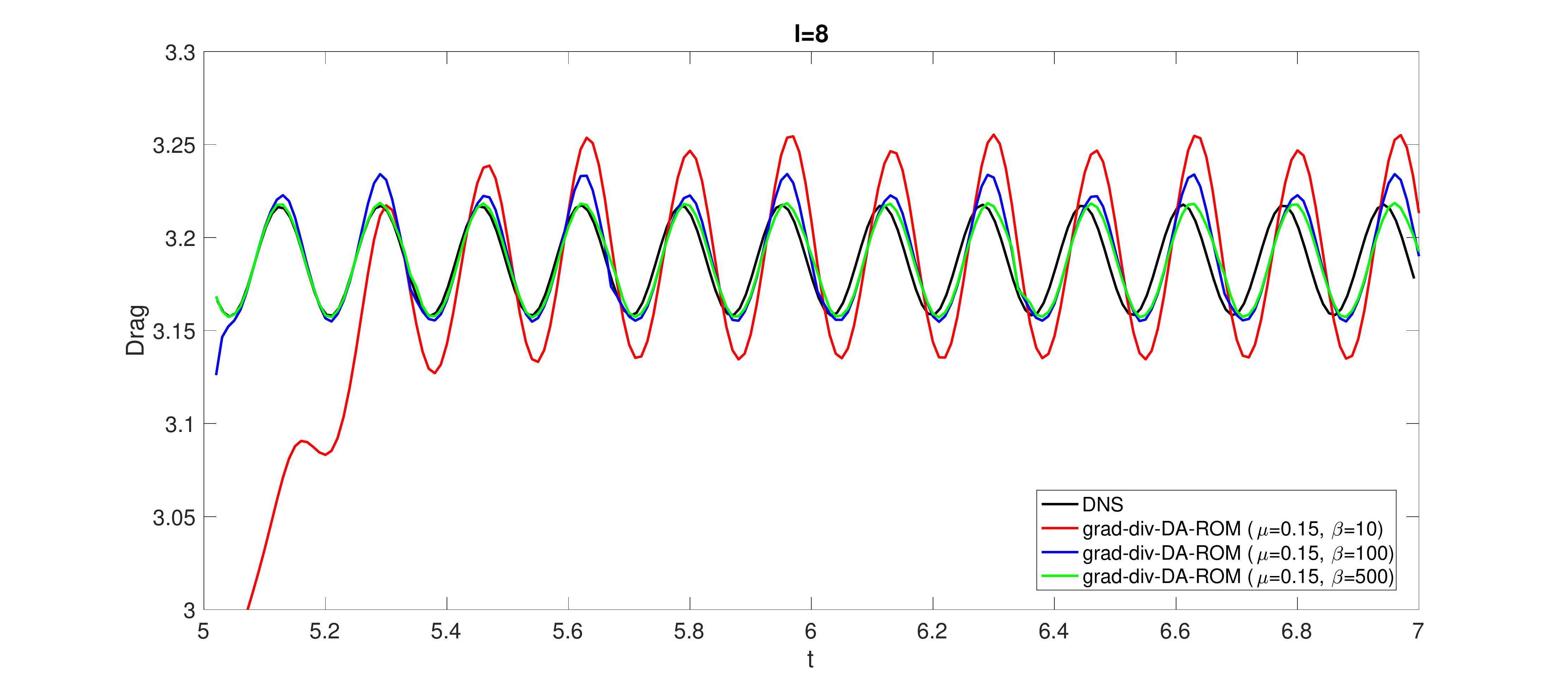}}
\centerline{\includegraphics[width=5.in]{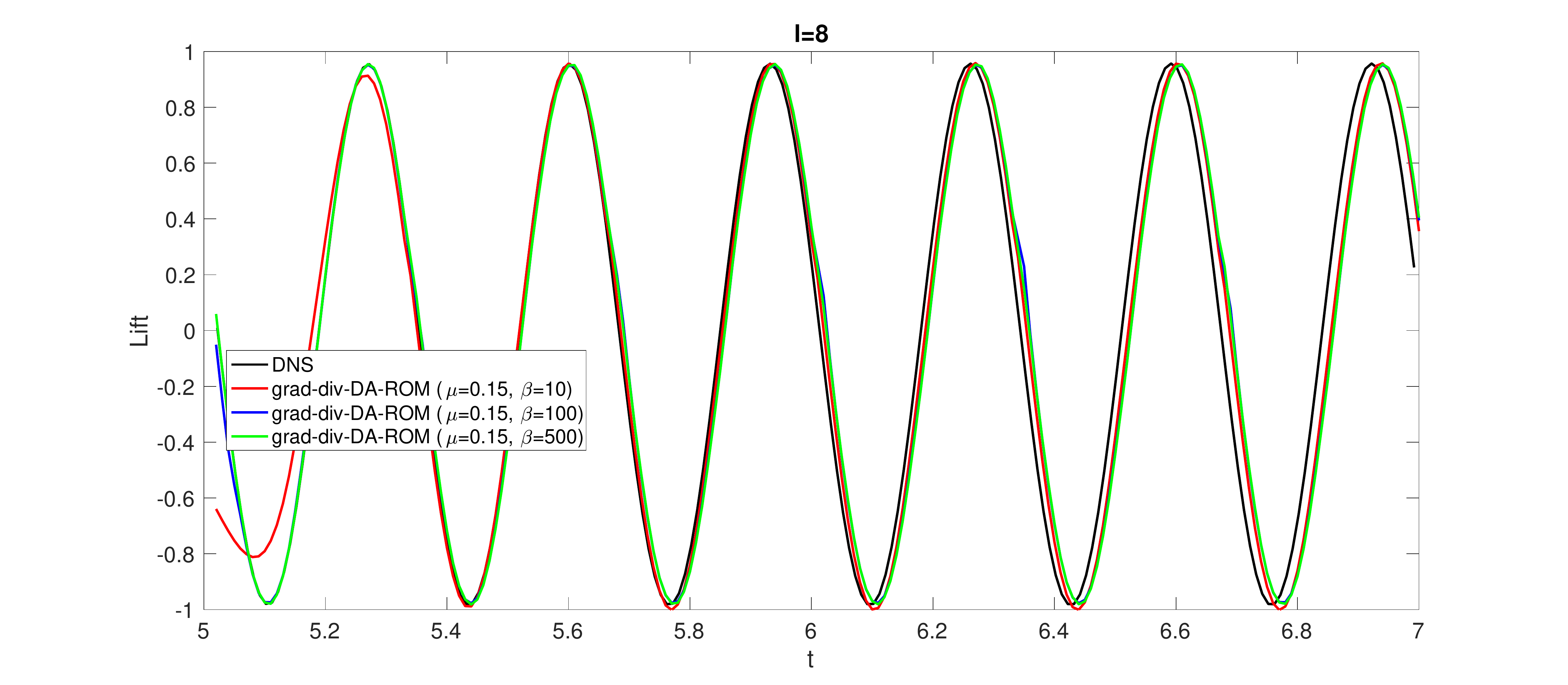}}
\caption{Example \ref{sec:Re=100} (Case $Re=100$): Temporal evolution of kinetic energy, drag coefficient and lift coefficient using $l=8$ modes for grad-div-DA-ROM with $\mu=0.15$ and $\beta=10,\,100,\,500$ ($106$ snapshots used, which comprise $64\%$ of one full period from $t=5\,\rm{s}$ to $t=5.212\,\rm{s}$).}\label{fig:QOIgdDAInac1}
\end{center}
\end{figure}

\begin{table}[htb]
$$\hspace{-0.1cm}
\begin{tabular}{|c|c|c|c|c|}
\hline
 & \multicolumn{4}{|c|}{$Re=100$ (Inaccurate snapshots)}\\
\hline
Errors & G-ROM & grad-div-ROM & DA-ROM & grad-div-DA-ROM\\
\hline
$E_{kin}^{max}$ & 4.34e-02 & 4.13e-04 & 2.30e-05 & 6.10e-05\\
\hline
$c_{D}^{max}$ & 3.75e-01 & 6.40e-02 & 1.14e-02 & 1.16e-02\\
\hline
$c_{L}^{max}$ & 6.29e-01 & 1.89e-02 & 1.16e-02 & 1.11e-02\\
\hline
$\ell^2({\bf L}^2)\, \uv$ norm & 1.76e-01 & 1.01e-01 & 2.99e-02 & 2.99e-02\\
\hline
\end{tabular}$$\caption{Example \ref{sec:Re=100} (Case $Re=100$): Errors levels with respect to DNS for G-ROM, grad-div-ROM ($\mu=0.15$), DA-ROM ($\beta=500$), and grad-div-DA-ROM ($\mu=0.15, \beta=500$) ($106$ snapshots used, which comprise $64\%$ of one full period from $t=5\,\rm{s}$ to $t=5.212\,\rm{s}$).}\label{tab:ErrLevCompInac1}
\end{table}


\subsection{Case $Re=1000$}\label{sec:Re=1000}

In this section, we discuss results for $Re=1000$. In this case, we have used a finer computational grid with respect to $Re=100$ to compute the snapshots (see Figure \ref{fig:MeshRef} on top, for which $h=1.46\times 10^{-2}\,\rm{m}$, resulting in $101\,820$ d.o.f. for velocities and $12\,885$ d.o.f. for pressure). This has been necessary to obtain stable DNS results. However, the coarse mesh for DA in ROM is the same as for the previous case (see Figure \ref{fig:MeshRef} on bottom). The full period length of the statistically steady state is now $0.22\,\rm{s}$, so that $110$ snapshots were collected, starting from $t=5\,\rm{s}$. Again, all tested ROM have been run in the stable response time interval $[5,7]\,\rm{s}$, corresponding now to nine periods for the lift coefficient. This time range is thus nine times wider with respect to the time window used for the generation of the POD modes, so that at the higher Reynolds number we are performing the longer time integration with respect to the time interval used to compute the snapshots.

Numerical results for energy, drag and lift predictions using $l=8$ modes are shown in Figures \ref{fig:QOI1000}, \ref{fig:QOIda1000}, \ref{fig:QOIgdDA1000}. In particular, Figure \ref{fig:QOI1000} shows a comparison within DNS, G-ROM, grad-div-ROM with $\mu=0.001$, DA-ROM with $\beta=10$, and grad-div-DA-ROM with $\mu=0.001$ and $\beta=10$. The value $\mu=0.001$ for the grad-div stabilization term has been fixed again minimizing the error with respect to the DNS energy. As already noticed in the previous case, from this figure we observe that, whereas the G-ROM solution is totally inaccurate, the application of the grad-div stabilization term helps to improve the G-ROM solution, although it shows larger error levels than the lower Reynolds number case $Re=100$ when compared to DNS results. A slight improvement is observed again for using DA with $\beta=10$, being results for DA-ROM and grad-div-DA-ROM almost identical. Looking at the temporal evolution of the kinetic energy (on top), we observe that also in this case the DA results almost stabilize around $t=5.4\,\rm{s}$ with $\beta=10$, even if the reached values under-estimate the DNS results.

Increasing the nudging parameter $\beta$ from $10$ to $500$ for DA reduced order methods (see Figures \ref{fig:QOIda1000}, \ref{fig:QOIgdDA1000}) already allows to almost approach DNS results, although we note a detachment in predicting $c_D, c_L$ as time increases. Almost identical results are obtained with both DA reduced order methods, for which the best predictions are given by the largest values $\beta=500$ of the nudging parameter, although we observe a similar accuracy already for $\beta=100$. Note again that for large values of the nudging parameter ($\beta=100, 500$), the DNS results are almost approached just after very few iterations (around $20$, i.e. $0.04\,\rm{s}$, for $\beta=100$ and $5$, i.e. $0.01\,\rm{s}$, for $\beta=500$). All these results are confirmed by Table \ref{tab:ErrLevComp1000}. Note that grad-div-ROM now just slightly reduces the error levels with respect to G-ROM for all quantities, while both DA reduced order methods still guarantee a reduction of two orders of magnitude for $E_{kin}^{max}, c_{D}^{max}$, and five times for $c_{L}^{max}$. In terms of $\ell^2({\bf L}^2)$ velocity norm, both DA reduced order methods reduces the G-ROM error level by a factor of $6.5$, while the grad-div-ROM is just slightly better accurate than G-ROM.

\begin{figure}[htb]
\begin{center}
\centerline{\includegraphics[width=5in]{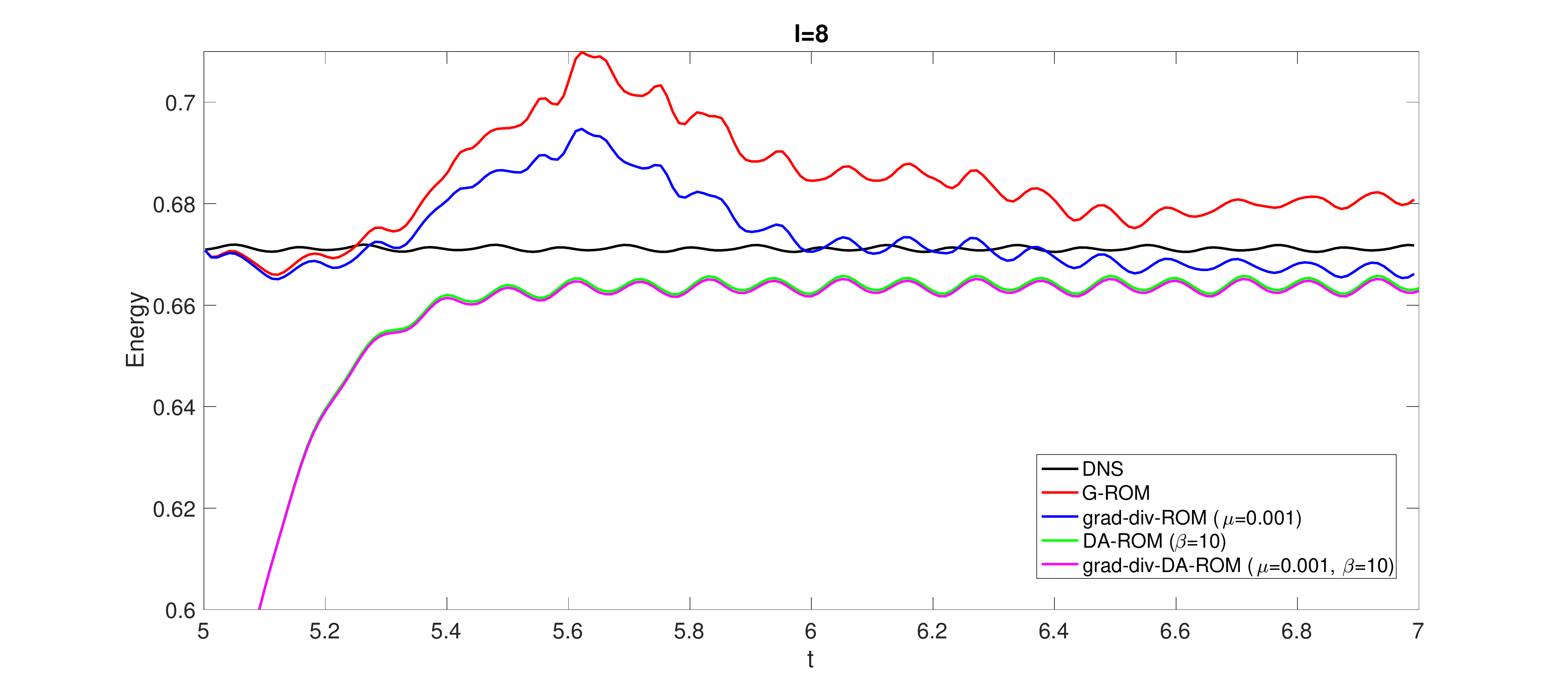}}
\centerline{\includegraphics[width=5in]{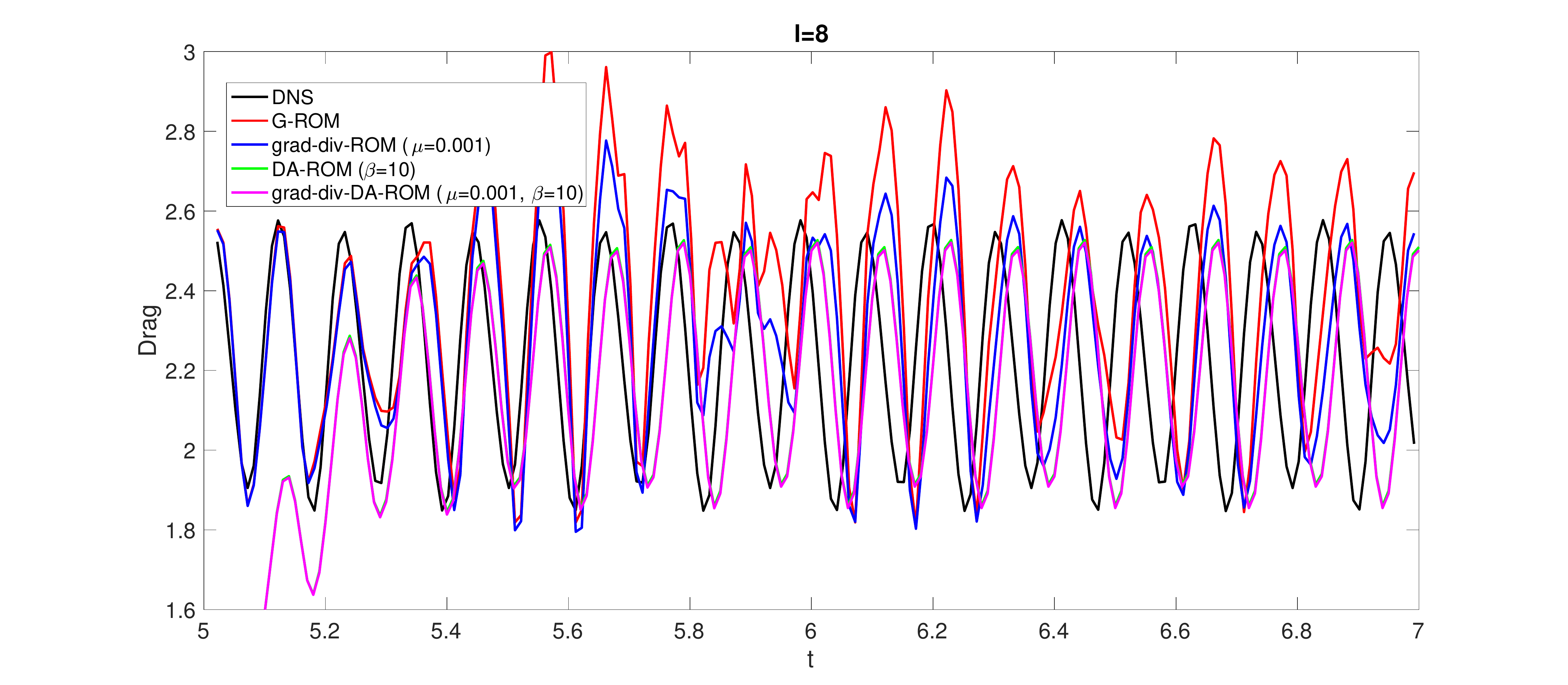}}
\centerline{\includegraphics[width=5in]{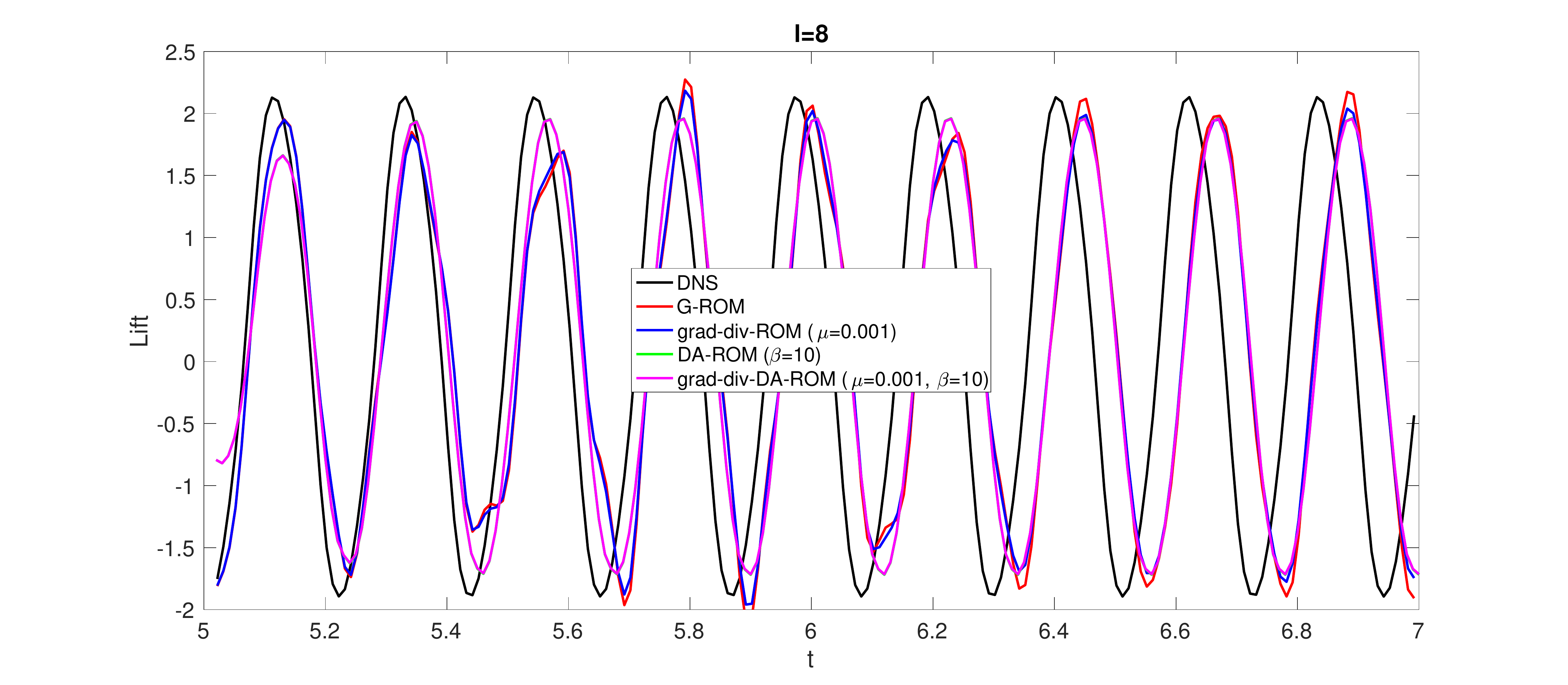}}
\caption{Example \ref{sec:Re=1000} (Case $Re=1000$): Temporal evolution of kinetic energy, drag coefficient and lift coefficient using $l=8$ modes ($110$ snapshots used, which comprise one full period from $t=5\,\rm{s}$ to $t=5.22\,\rm{s}$).}\label{fig:QOI1000}
\end{center}
\end{figure}

\begin{figure}[htb]
\begin{center}
\centerline{\includegraphics[width=5in]{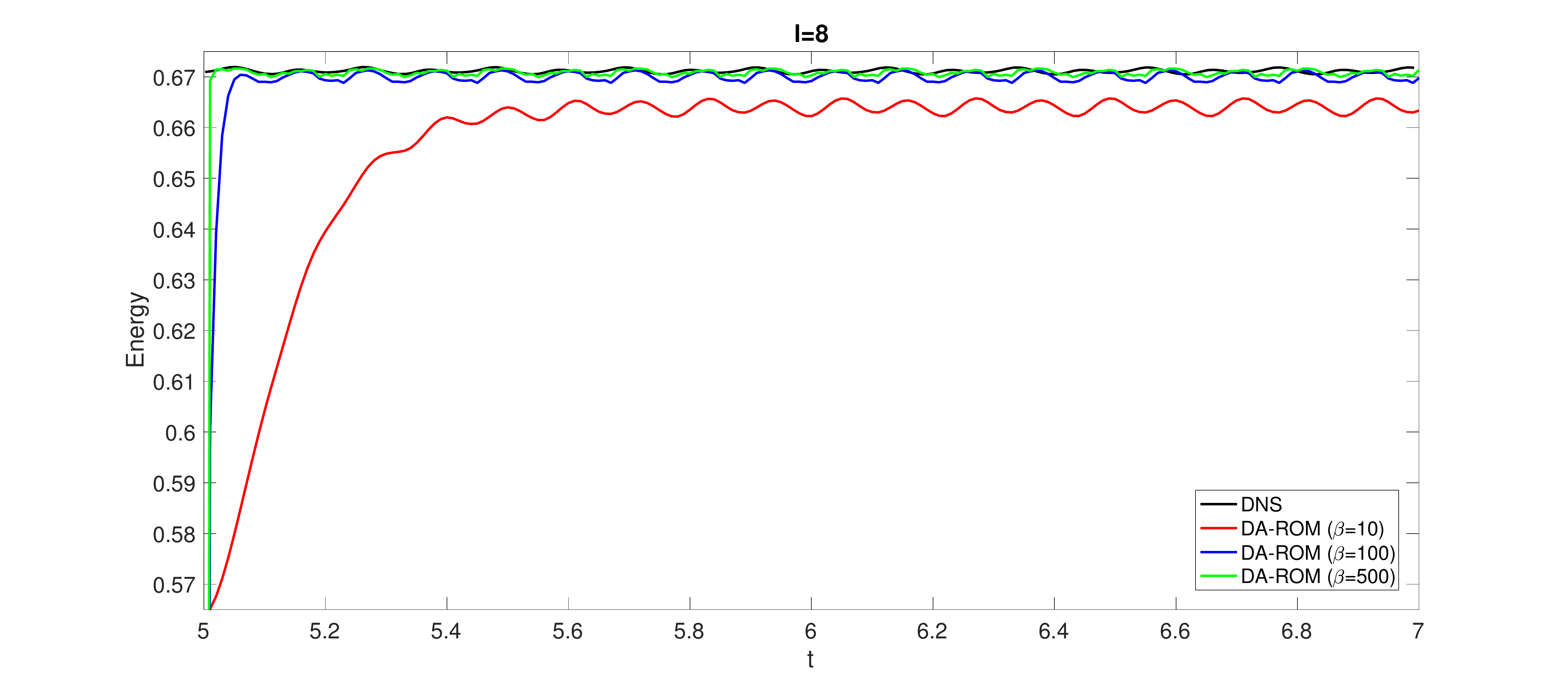}}
\centerline{\includegraphics[width=5in]{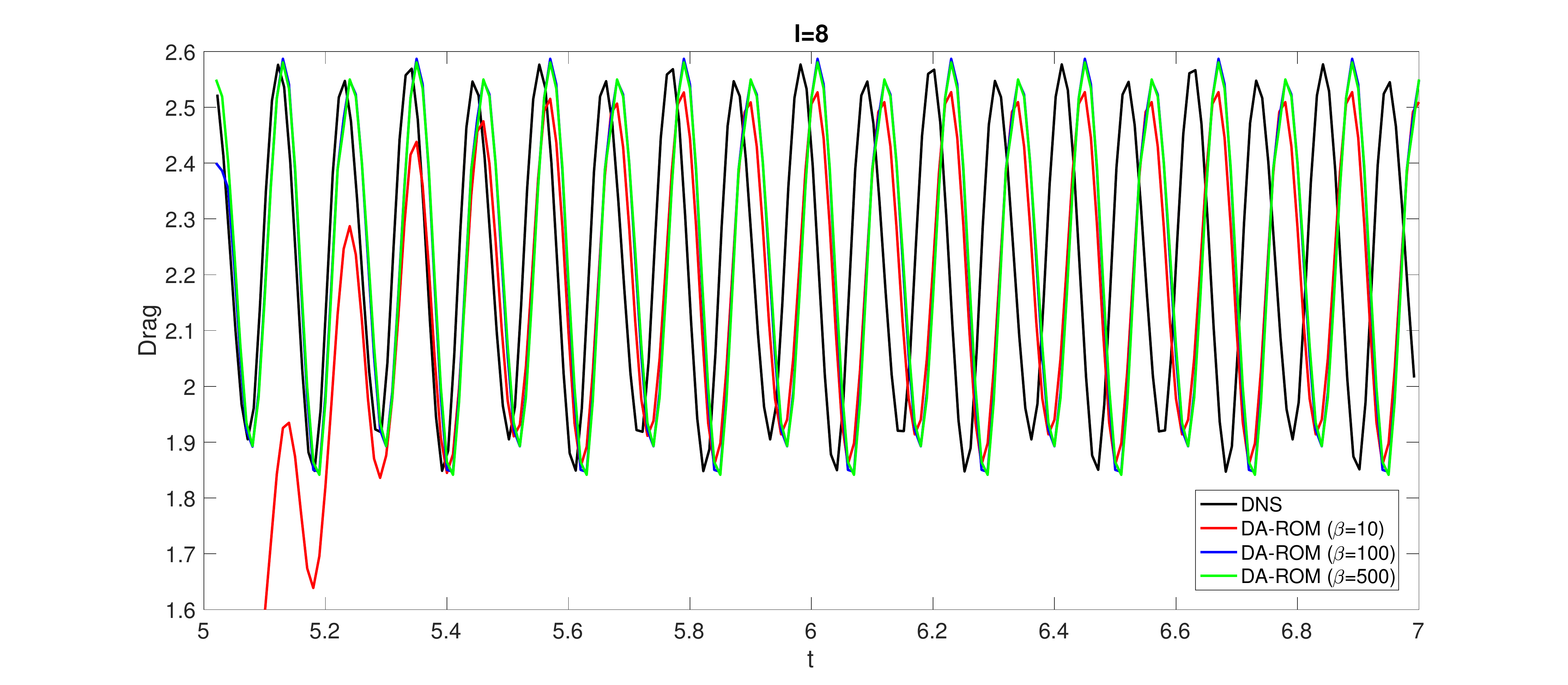}}
\centerline{\includegraphics[width=5in]{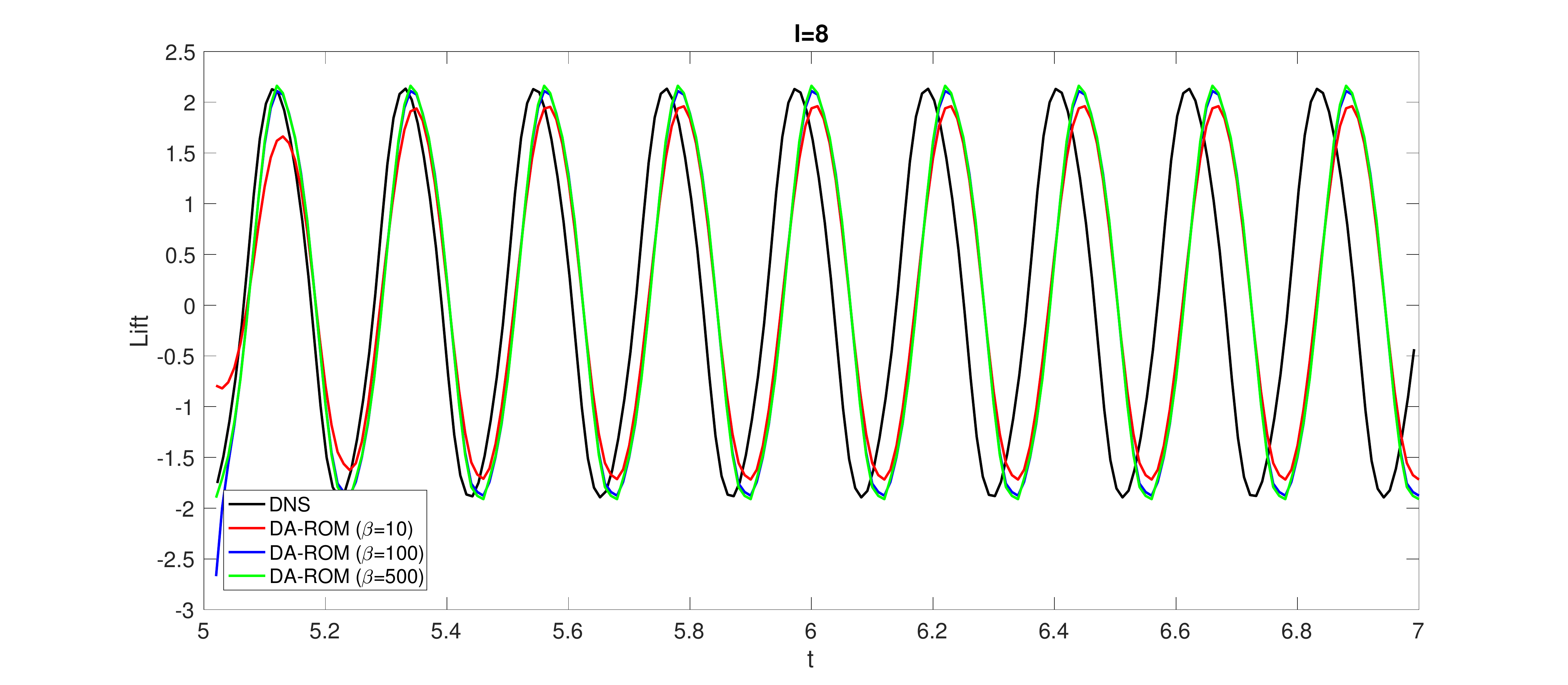}}
\caption{Example \ref{sec:Re=1000} (Case $Re=1000$): Temporal evolution of kinetic energy, drag coefficient and lift coefficient using $l=8$ modes for DA-ROM with $\beta=10,\,100,\,500$ ($110$ snapshots used, which comprise one full period from $t=5\,\rm{s}$ to $t=5.22\,\rm{s}$).}\label{fig:QOIda1000}
\end{center}
\end{figure}

\begin{figure}[htb]
\begin{center}
\centerline{\includegraphics[width=5in]{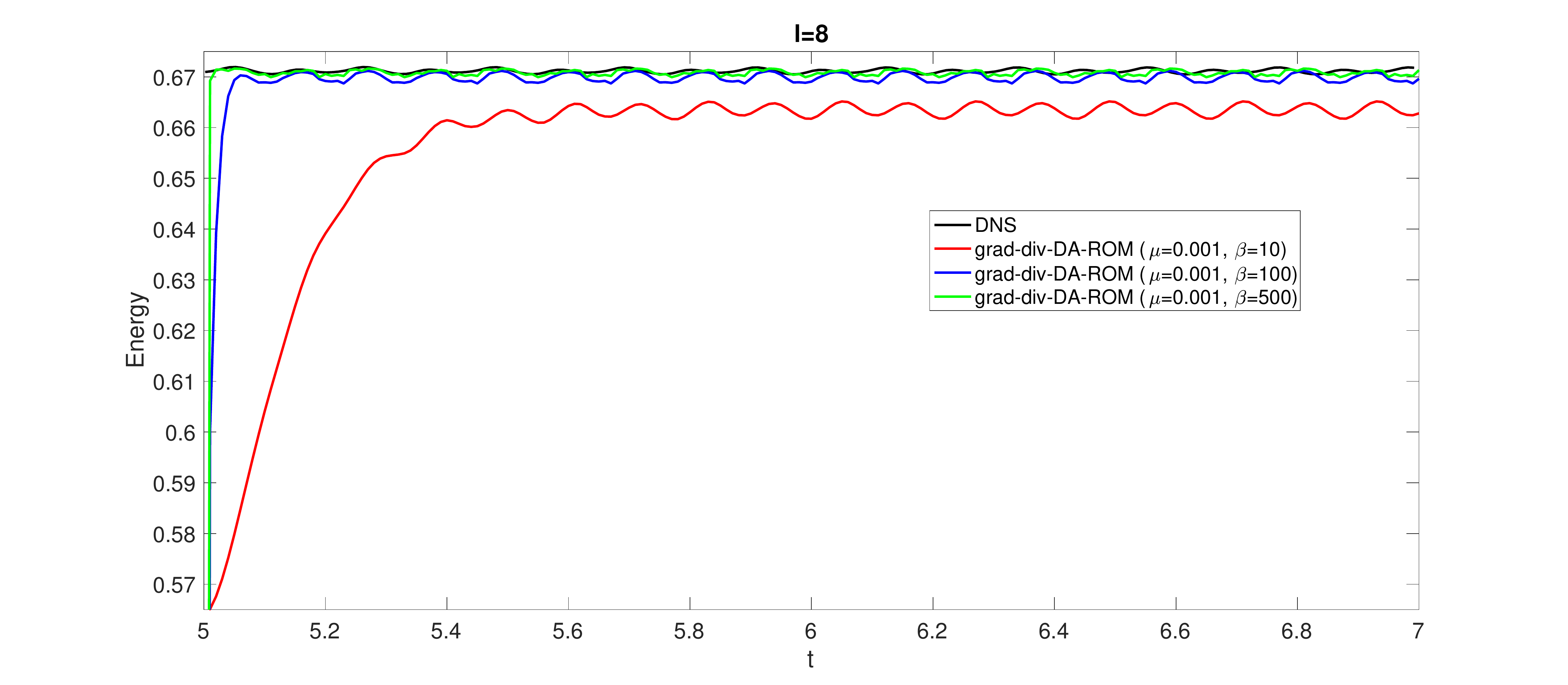}}
\centerline{\includegraphics[width=5in]{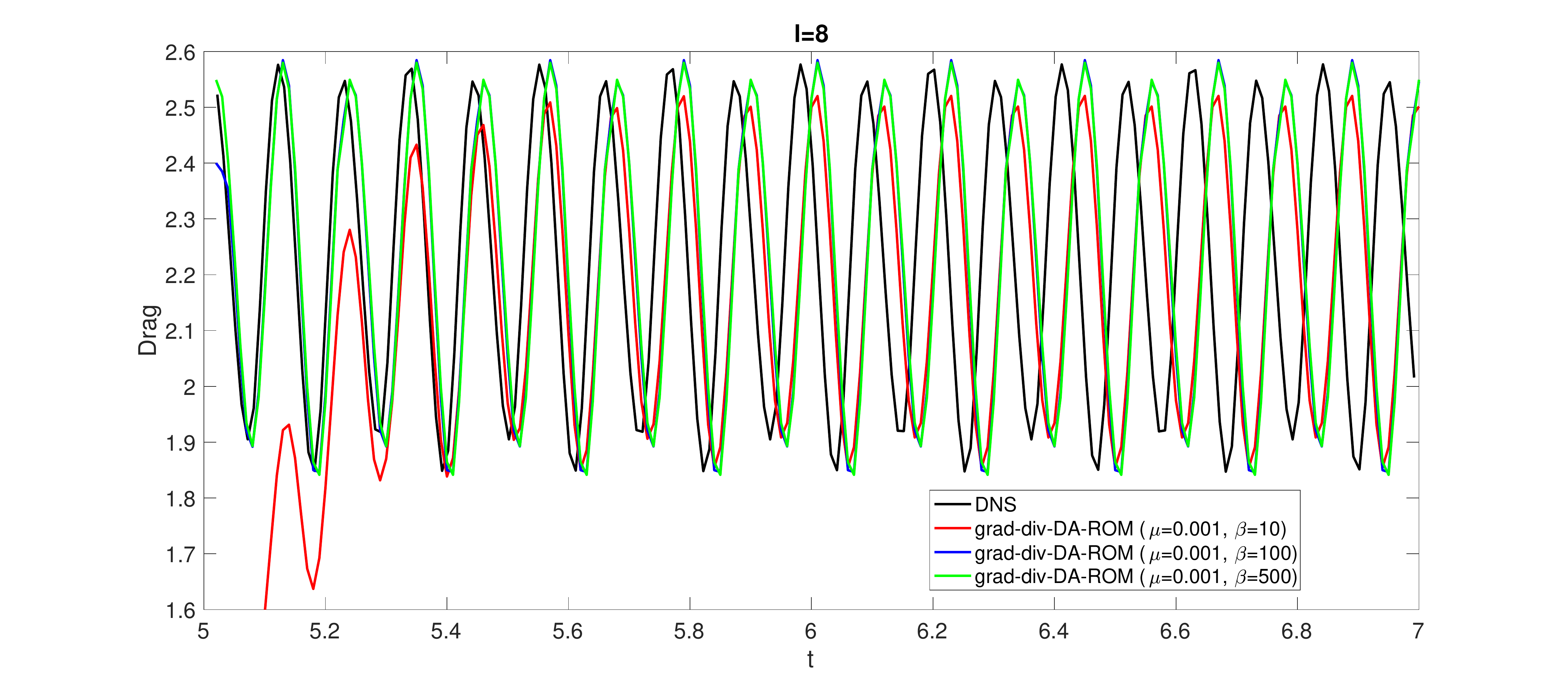}}
\centerline{\includegraphics[width=5in]{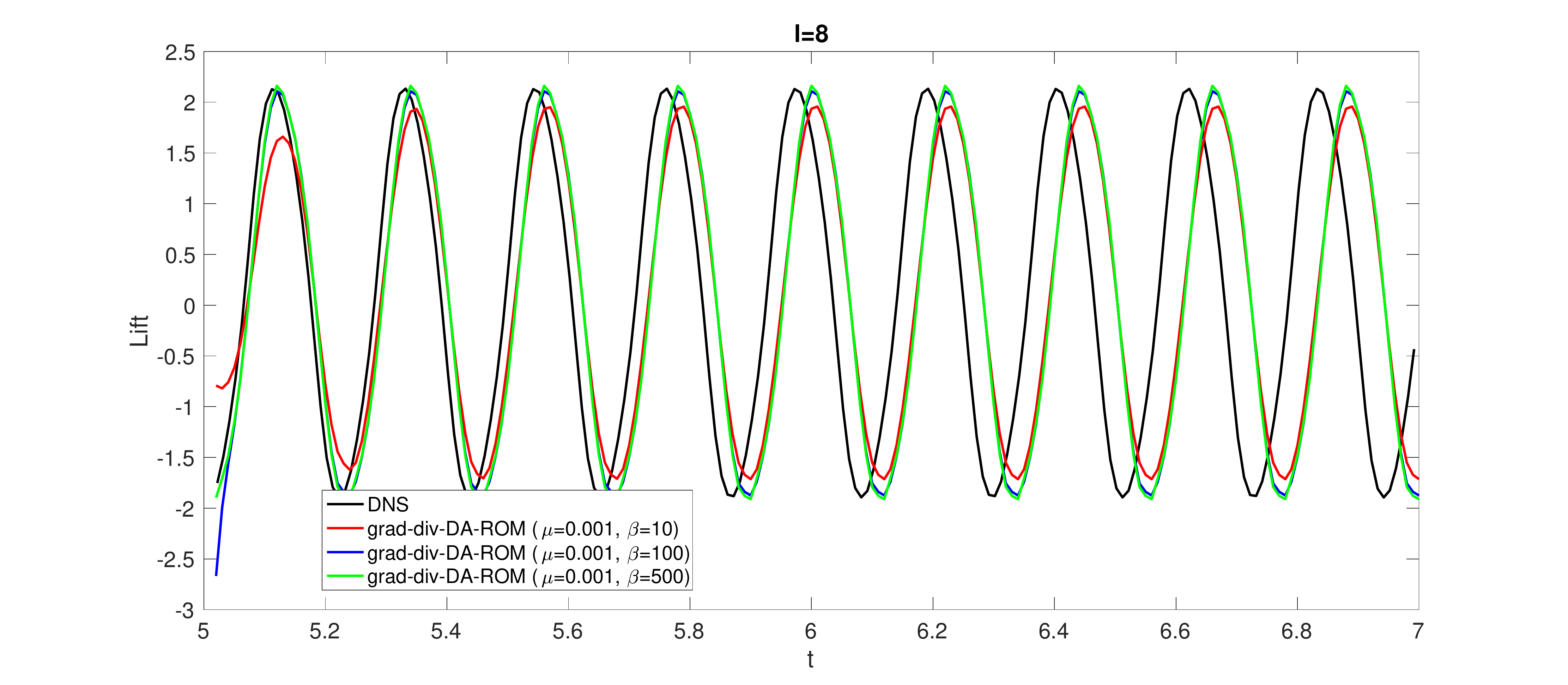}}
\caption{Example \ref{sec:Re=1000} (Case $Re=1000$): Temporal evolution of kinetic energy, drag coefficient and lift coefficient using $l=8$ modes for grad-div-DA-ROM with $\mu=0.001$ and $\beta=10,\,100,\,500$ ($110$ snapshots used, which comprise one full period from $t=5\,\rm{s}$ to $t=5.22\,\rm{s}$).}\label{fig:QOIgdDA1000}
\end{center}
\end{figure}

\begin{table}[htb]
$$\hspace{-0.1cm}
\begin{tabular}{|c|c|c|c|c|}
\hline
 & \multicolumn{4}{|c|}{$Re=1000$}\\
\hline
Errors & G-ROM & grad-div-ROM & DA-ROM & grad-div-DA-ROM\\
\hline
$E_{kin}^{max}$ & 3.79e-02 & 2.28e-02 & 2.85e-04 & 3.11e-04\\
\hline
$c_{D}^{max}$ & 4.21e-01 & 2.82e-01 & 3.46e-03 & 2.85e-03\\
\hline
$c_{L}^{max}$ & 1.40e-01 & 4.95e-02 & 2.83e-02 & 2.77e-02\\
\hline
$\ell^2({\bf L}^2)\, \uv$ norm & 3.38e-01 & 3.02e-01 & 5.19e-02 & 5.19e-02\\
\hline
\end{tabular}$$\caption{Example \ref{sec:Re=1000} (Case $Re=1000$): Errors levels with respect to DNS for G-ROM, grad-div-ROM ($\mu=0.001$), DA-ROM ($\beta=500$), and grad-div-DA-ROM ($\mu=0.001, \beta=500$) ($110$ snapshots used, which comprise one full period from $t=5\,\rm{s}$ to $t=5.22\,\rm{s}$).}\label{tab:ErrLevComp1000}
\end{table}

Also in this case we finally investigate the considered ROM performances in predicting quantities of interest when inaccurate snapshots ($64\%$ of one full period) are used in their construction. Thus, we generate inaccurate snapshots using $64\%$ of one full period of DNS data, which corresponds in this case to the first $70$ DNS time step solutions from $t=5\,\rm{s}$ to $t=5.14\,\rm{s}$. Figure \ref{fig:PODvelmodes1000} displays the Euclidean norm of the first POD velocity modes obtained with the full set of snapshots (left) and the inaccurate set of snapshots (right). Results for the considered ROM using $l=8$ modes in this case are shown in Figures \ref{fig:QOIInac11000}, \ref{fig:QOIdaInac11000}, \ref{fig:QOIgdDAInac11000}. Here, we observe that results for G-ROM and grad-div-ROM are rather inaccurate, being the grad-div-ROM slightly better, while results for both DA-ROM (with and without grad-div term) almost approaches DNS results as for the one full period case. All these considerations are also reflected by the error levels displayed in Table \ref{tab:ErrLevCompInac11000}. These results suggest that DA reduced order methods perform well also for low values of viscosity and display low sensitivity compared to increases in Reynolds number, even when low-resolution data are available to construct the reduced basis. This fact is extremely important in order to solve complex realistic flows at high Reynolds numbers and also provides a numerical support to the theoretical analysis performed, in which error bounds with constants independent on inverse powers of the viscosity parameter are derived.

\begin{figure}[htb]
\begin{center}
\includegraphics[width=2.385in]{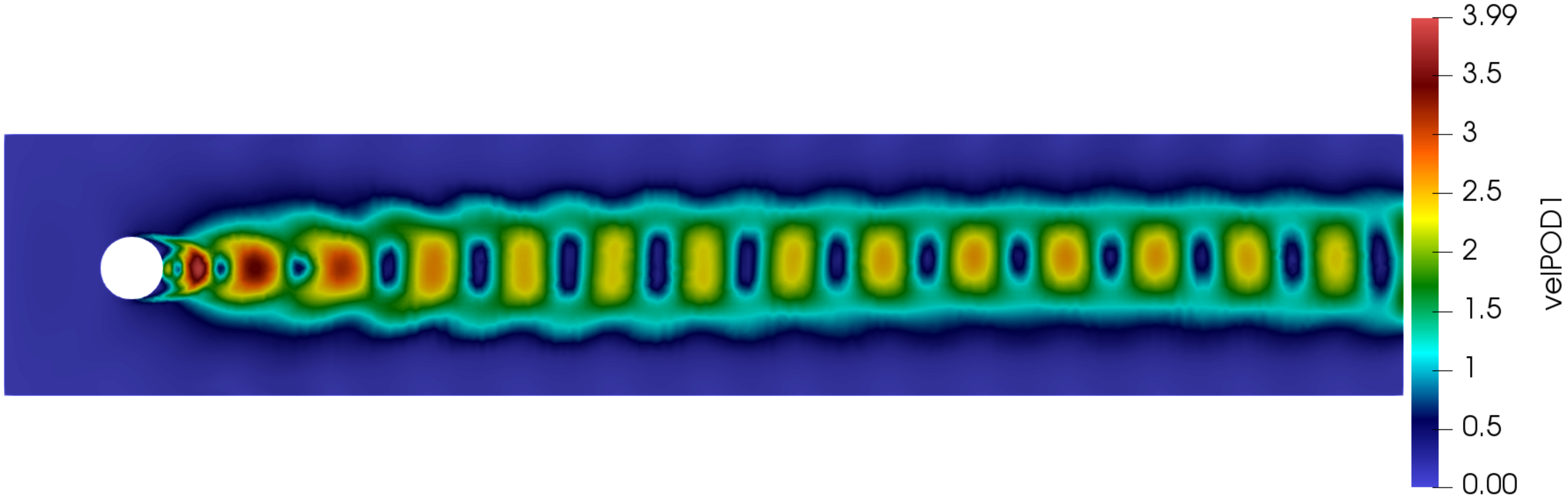}\hspace{-0.1cm}
\includegraphics[width=2.385in]{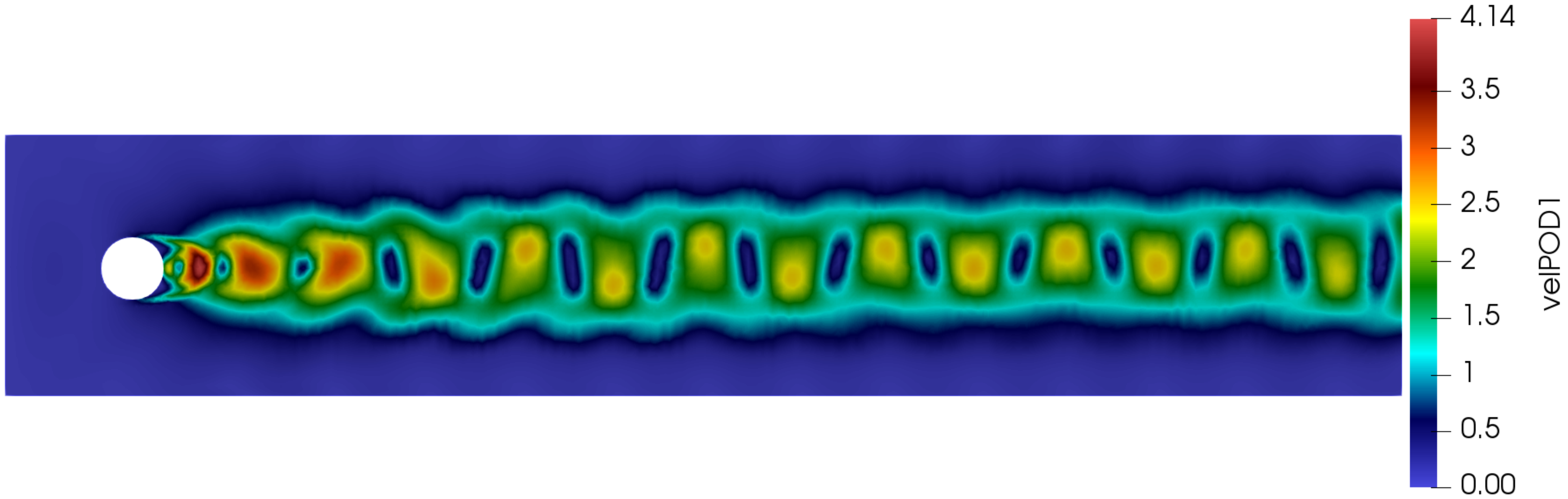}\\
\includegraphics[width=2.385in]{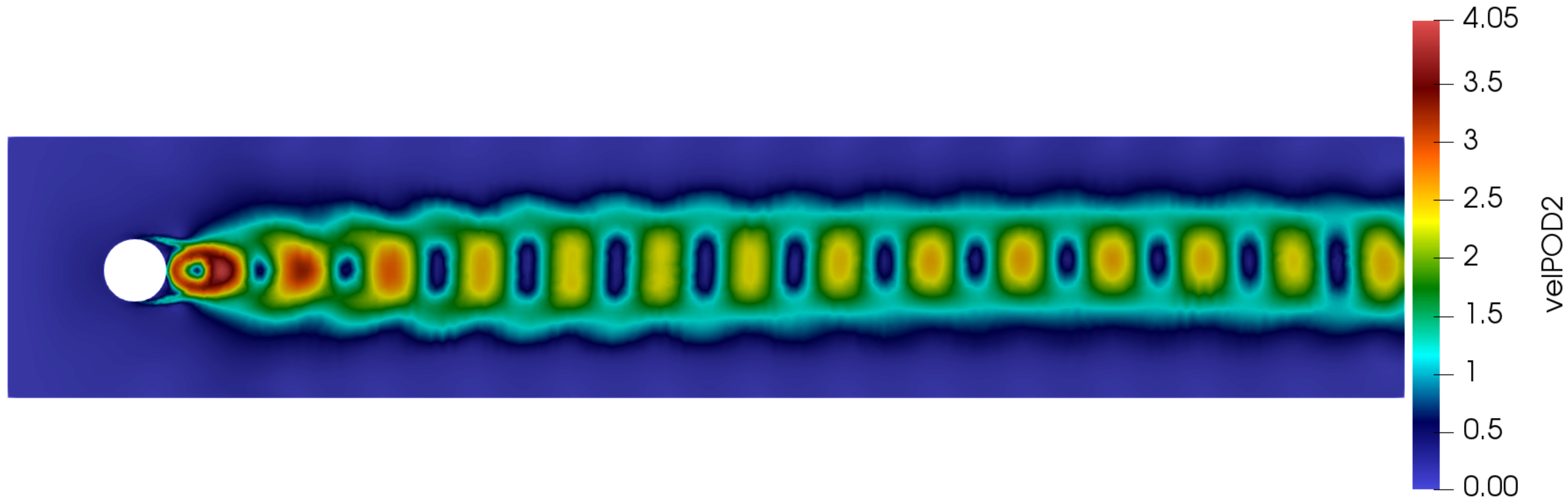}\hspace{-0.1cm}
\includegraphics[width=2.385in]{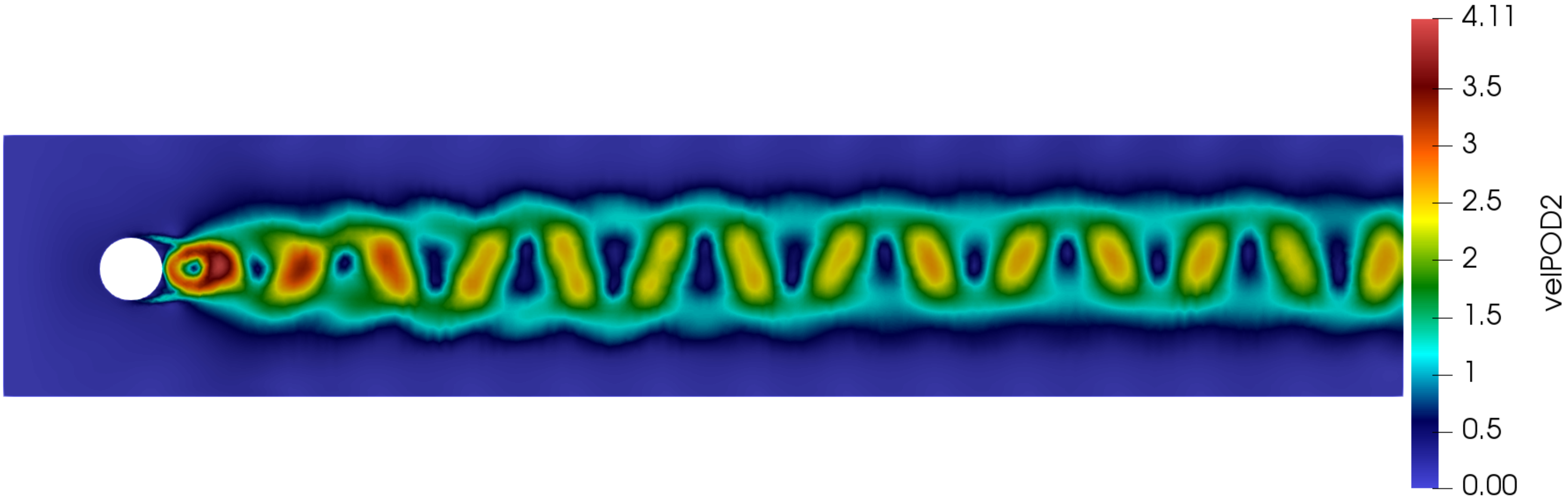}\\
\includegraphics[width=2.385in]{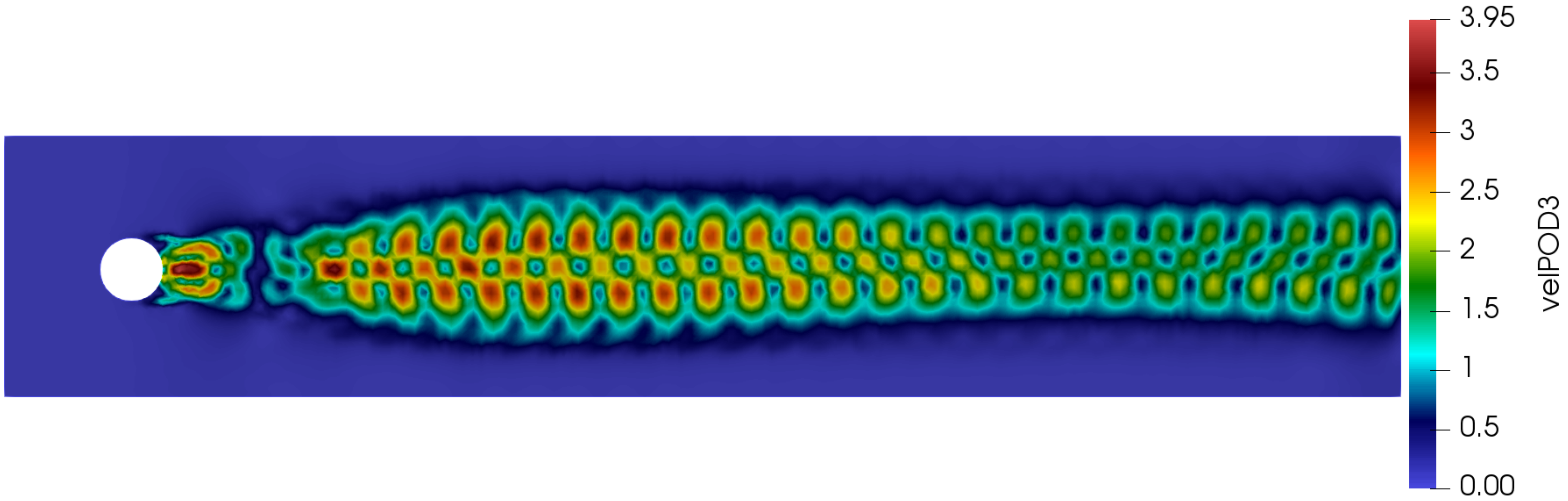}\hspace{-0.1cm}
\includegraphics[width=2.385in]{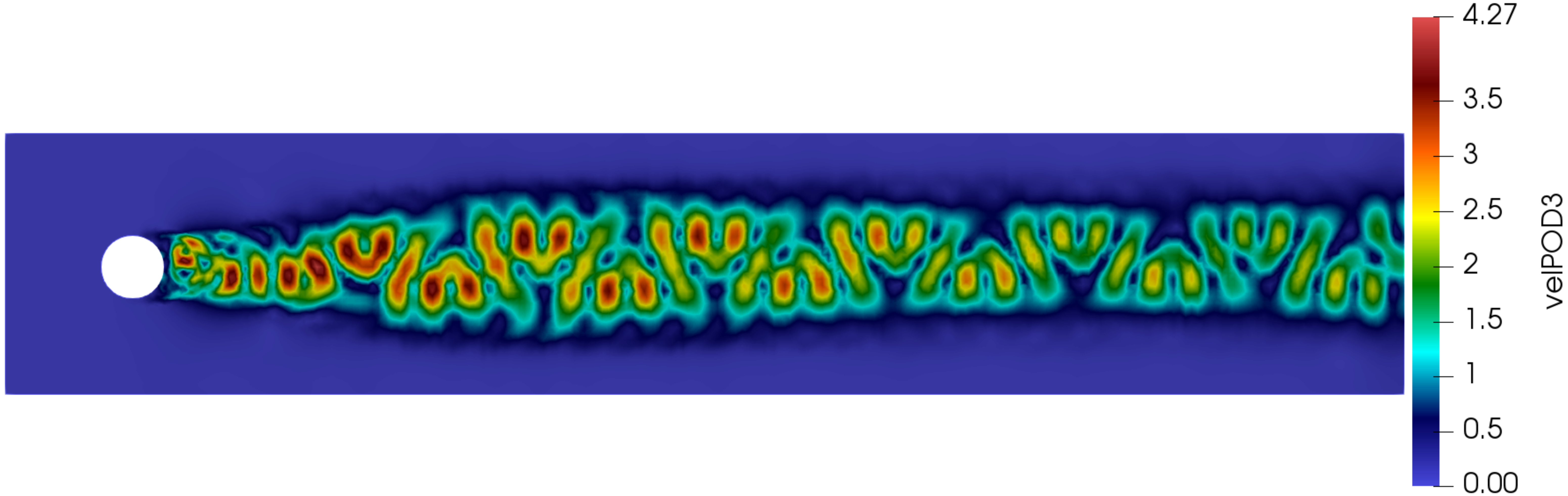}\\
\includegraphics[width=2.385in]{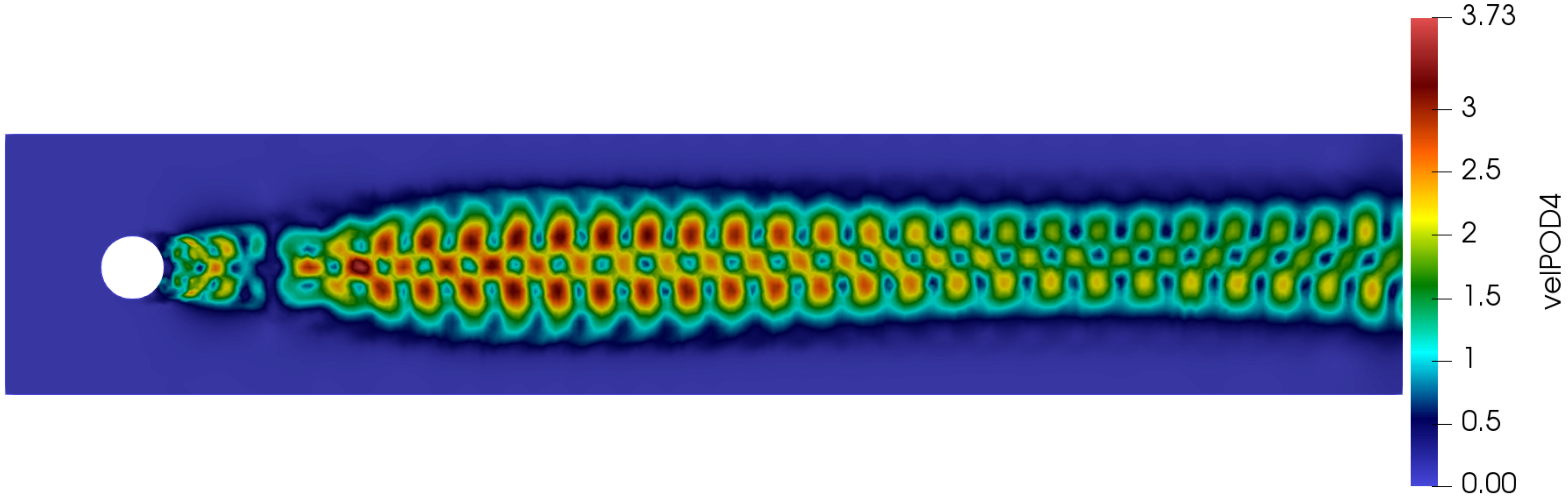}\hspace{-0.1cm}
\includegraphics[width=2.385in]{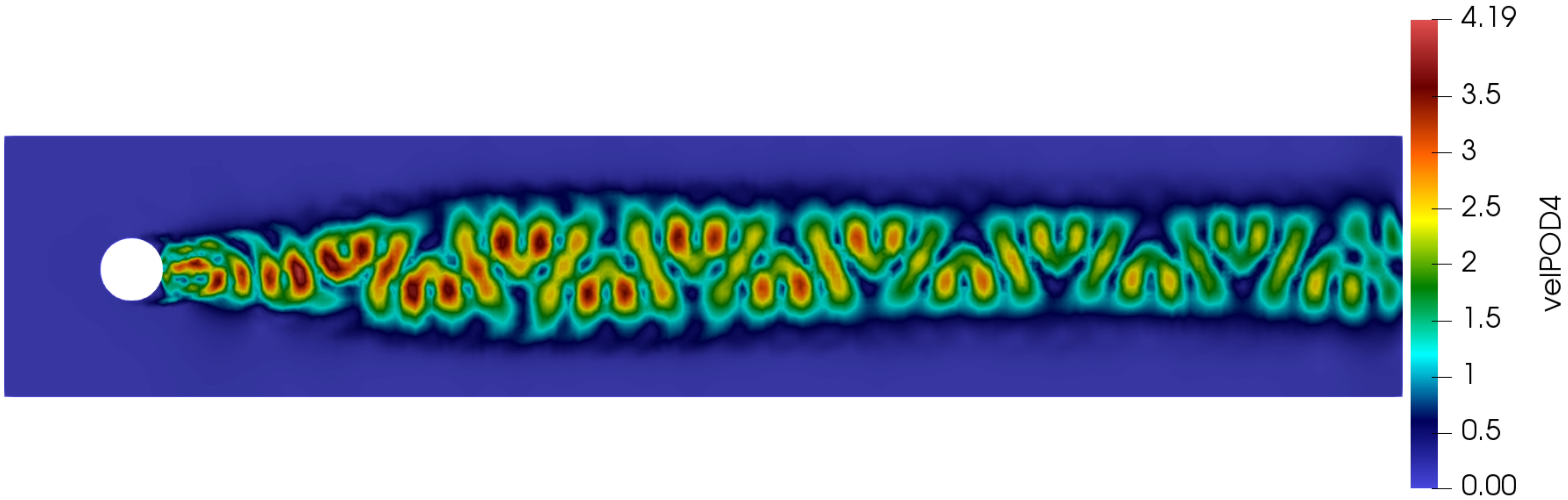}
\caption{Example \ref{sec:Re=1000} (Case $Re=1000$): First POD velocity modes (Euclidean norm) obtained with $110$ snapshots (full period basis, left) and $70$ snapshots (inaccurate basis corresponding to $64\%$ of one full period, right).}\label{fig:PODvelmodes1000}
\end{center}
\end{figure}

\begin{figure}[htb]
\begin{center}
\centerline{\includegraphics[width=5in]{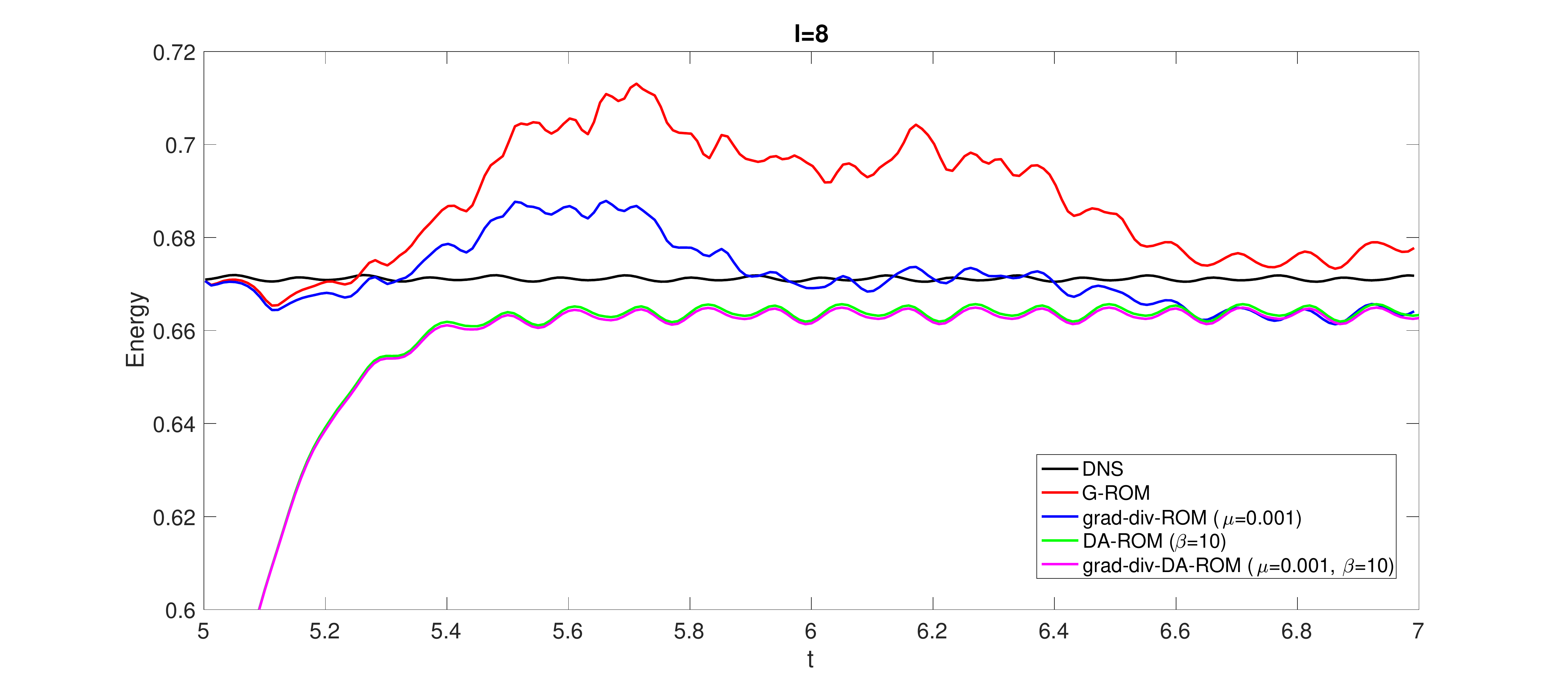}}
\centerline{\includegraphics[width=5in]{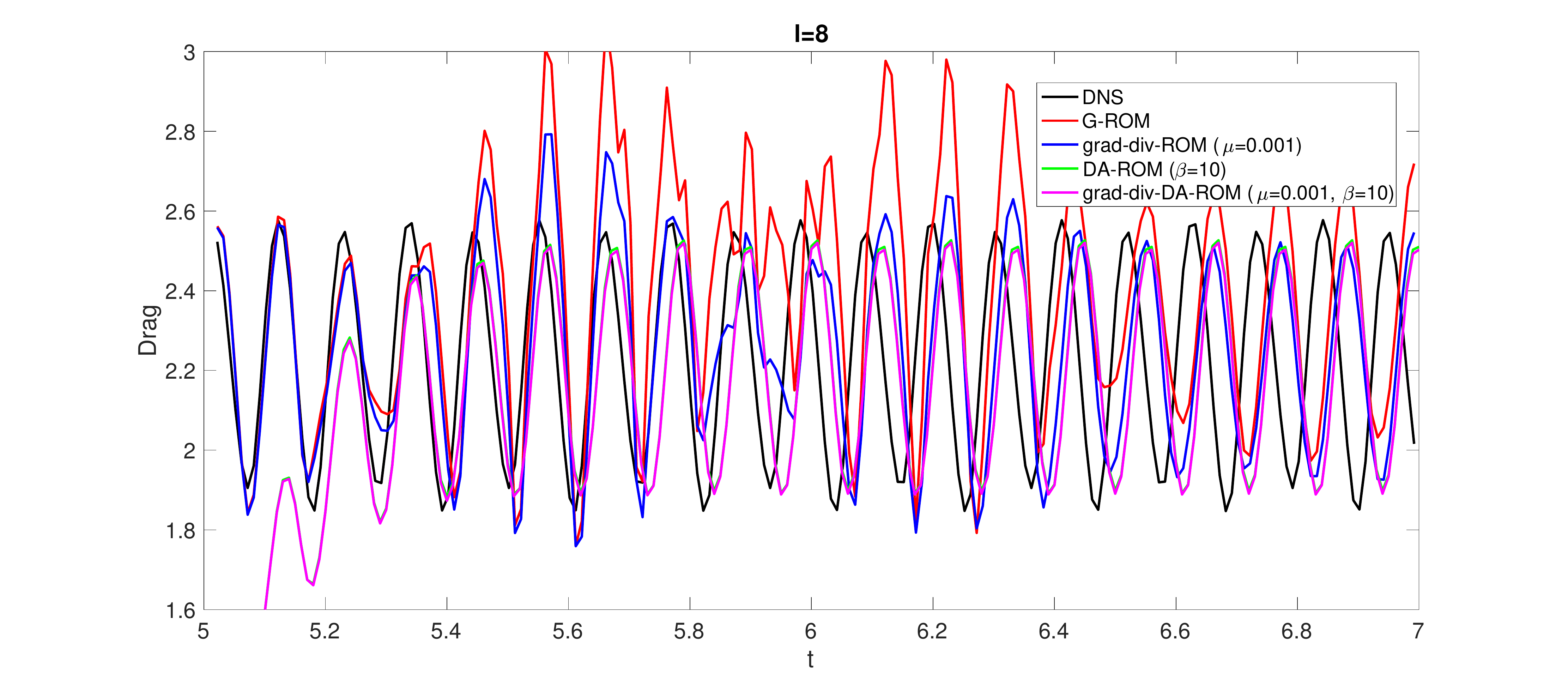}}
\centerline{\includegraphics[width=5in]{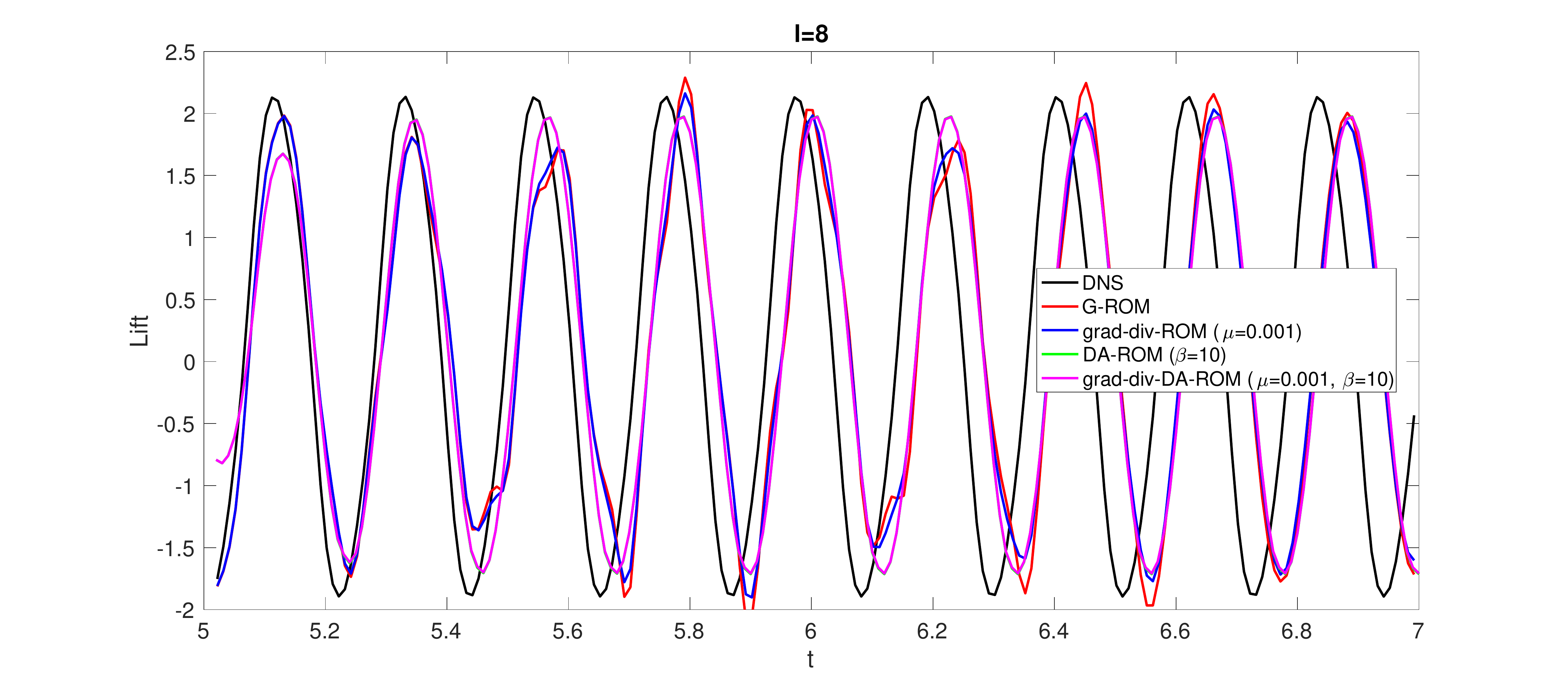}}
\caption{Example \ref{sec:Re=1000} (Case $Re=1000$): Temporal evolution of kinetic energy, drag coefficient and lift coefficient using $l=8$ modes ($70$ snapshots used, which comprise $64\%$ of one full period from $t=5\,\rm{s}$ to $t=5.14\,\rm{s}$).}\label{fig:QOIInac11000}
\end{center}
\end{figure}

\begin{figure}[htb]
\begin{center}
\centerline{\includegraphics[width=5in]{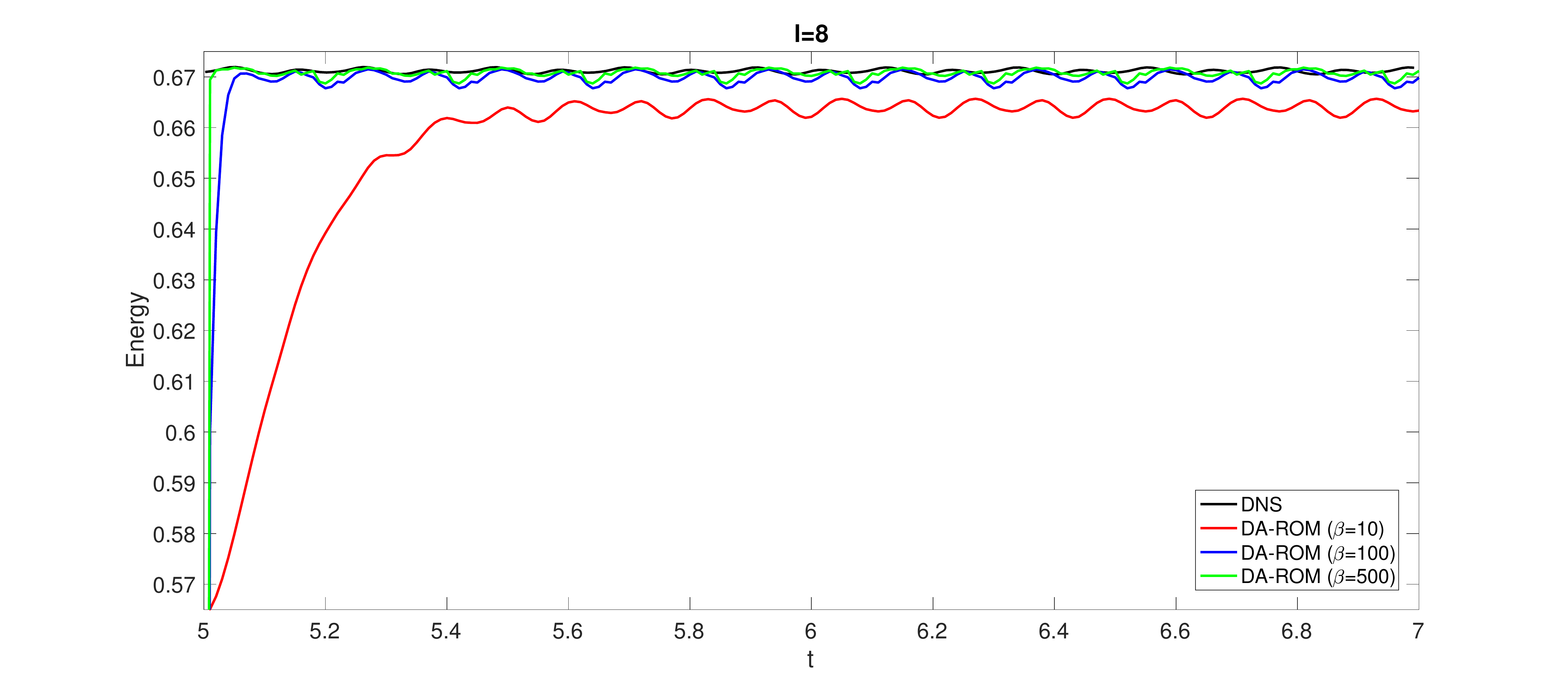}}
\centerline{\includegraphics[width=5in]{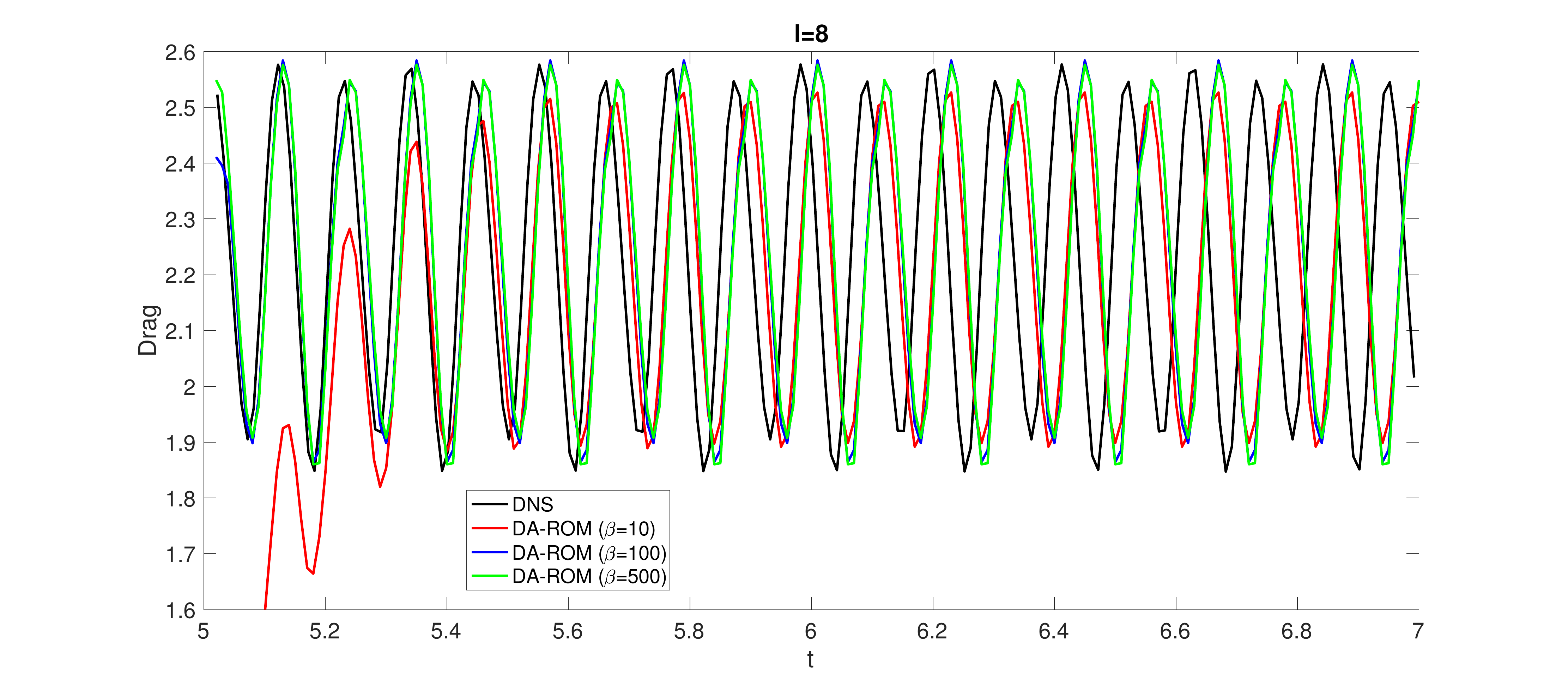}}
\centerline{\includegraphics[width=5in]{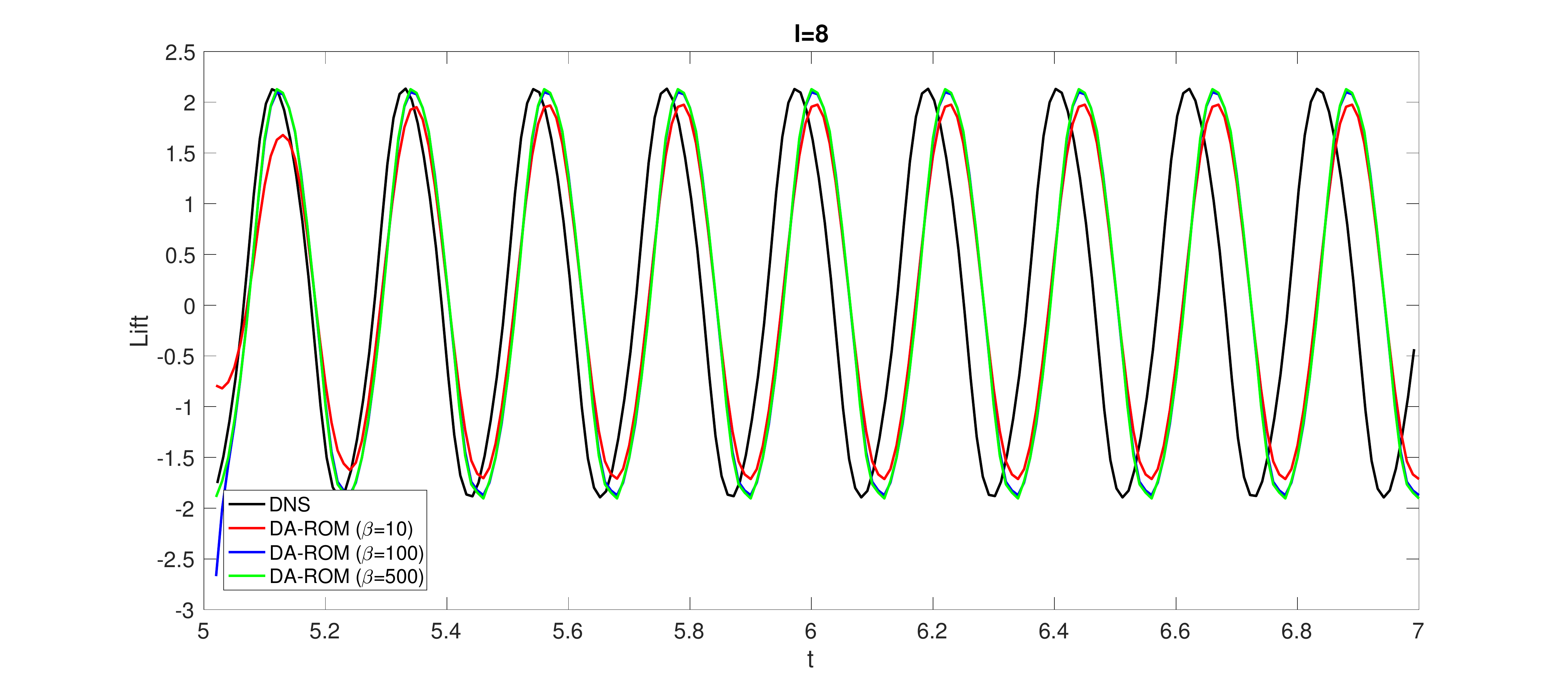}}
\caption{Example \ref{sec:Re=1000} (Case $Re=1000$): Temporal evolution of kinetic energy, drag coefficient and lift coefficient using $l=8$ modes for DA-ROM with $\beta=10,\,100,\,500$ ($70$ snapshots used, which comprise $64\%$ of one full period from $t=5\,\rm{s}$ to $t=5.14\,\rm{s}$).}\label{fig:QOIdaInac11000}
\end{center}
\end{figure}

\begin{figure}[htb]
\begin{center}
\centerline{\includegraphics[width=5in]{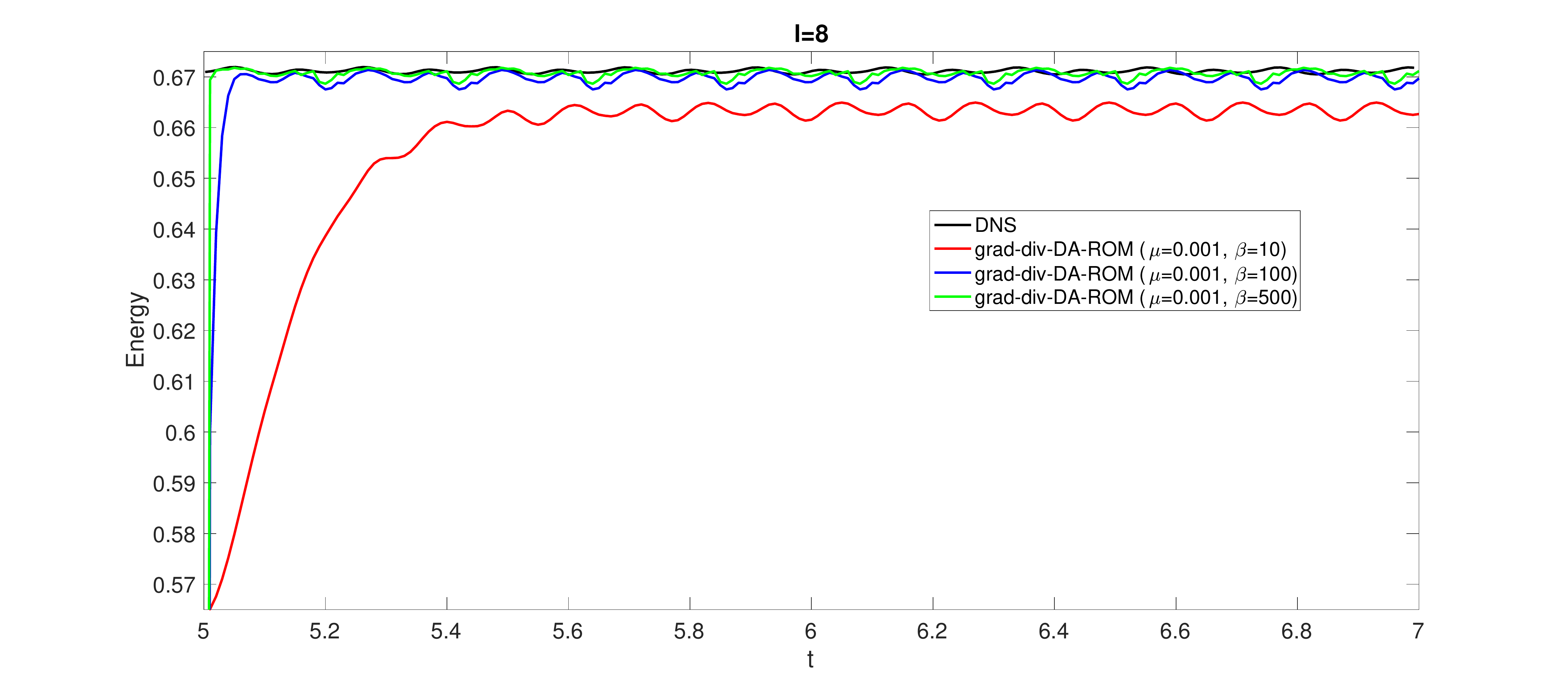}}
\centerline{\includegraphics[width=5in]{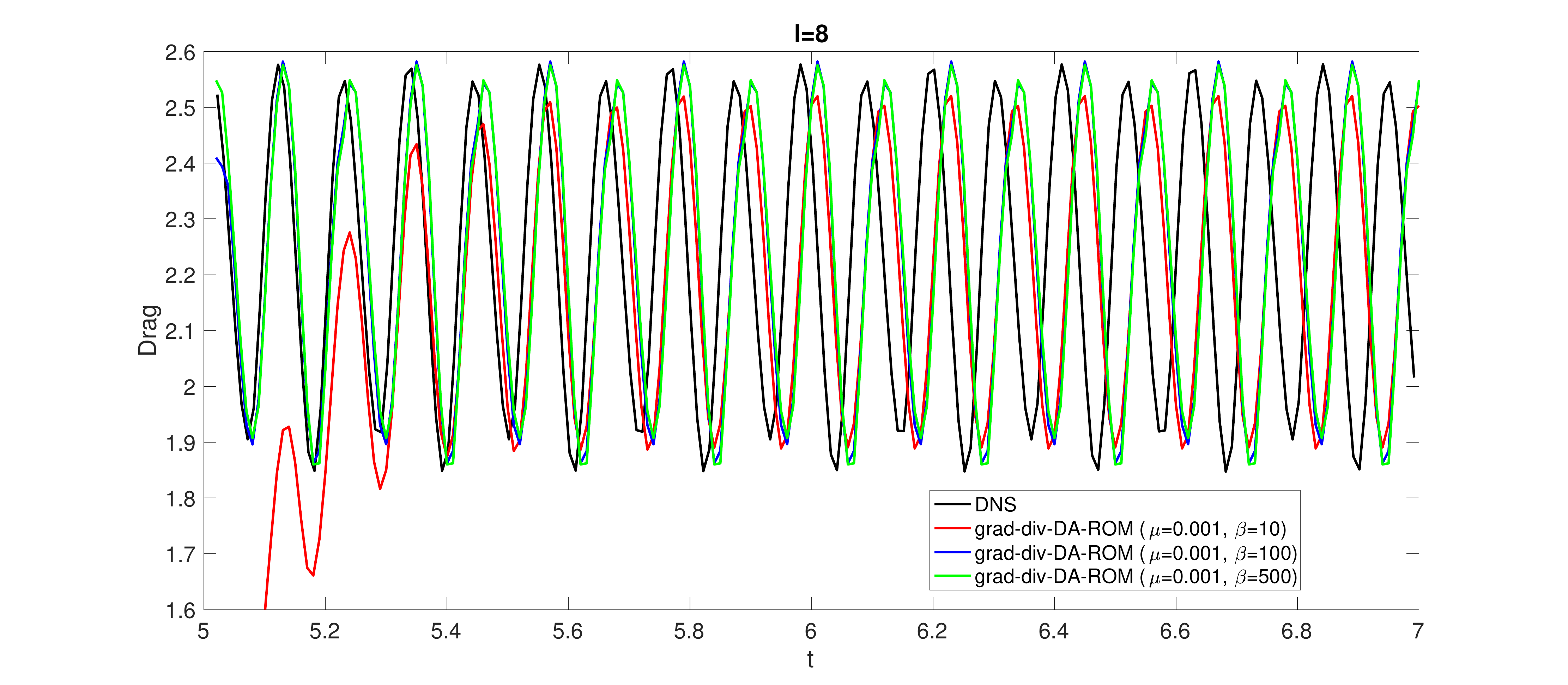}}
\centerline{\includegraphics[width=5in]{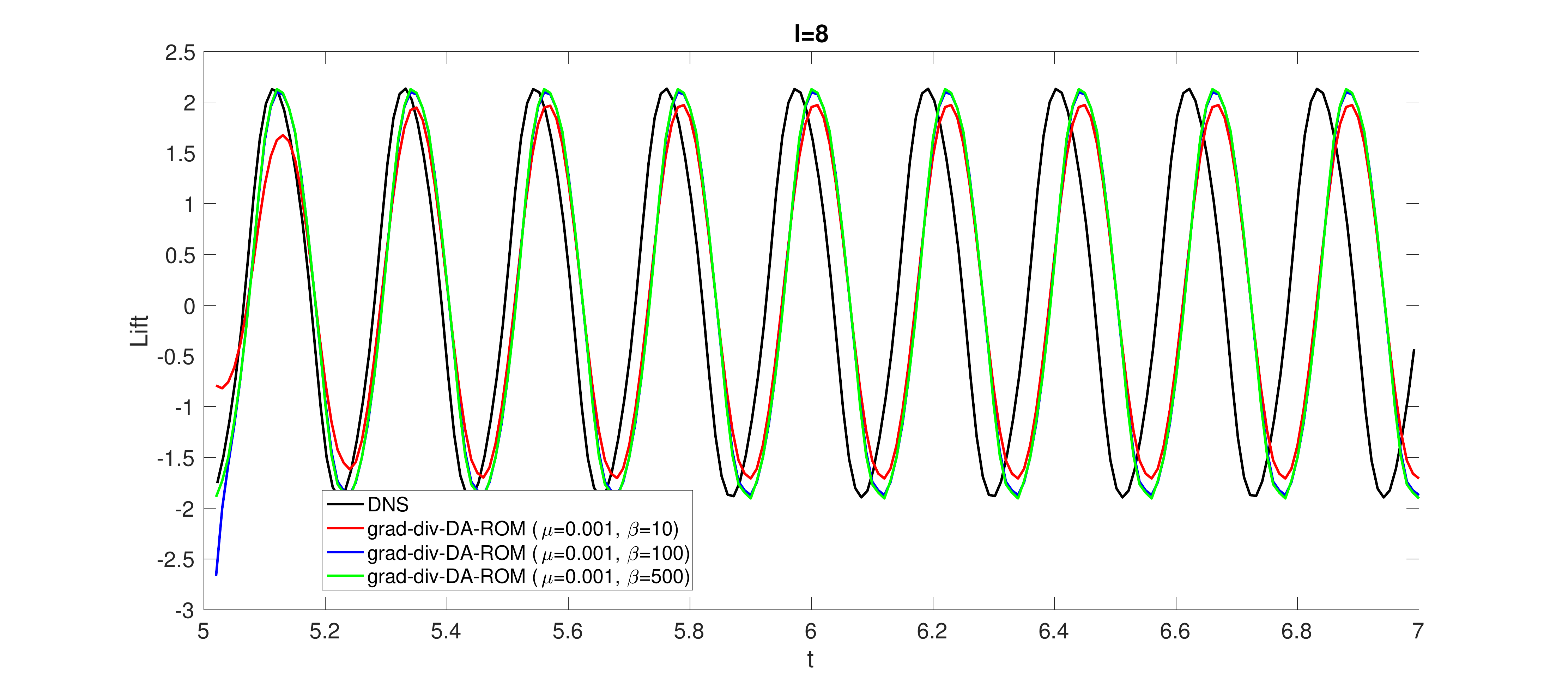}}
\caption{Example \ref{sec:Re=1000} (Case $Re=1000$): Temporal evolution of kinetic energy, drag coefficient and lift coefficient using $l=8$ modes for grad-div-DA-ROM with $\mu=0.001$ and $\beta=10,\,100,\,500$ ($70$ snapshots used, which comprise $64\%$ of one full period from $t=5\,\rm{s}$ to $t=5.14\,\rm{s}$).}\label{fig:QOIgdDAInac11000}
\end{center}
\end{figure}

\begin{table}[htb]
$$\hspace{-0.1cm}
\begin{tabular}{|c|c|c|c|c|}
\hline
 & \multicolumn{4}{|c|}{$Re=1000$ (Inaccurate snapshots)}\\
\hline
Errors & G-ROM & grad-div-ROM & DA-ROM & grad-div-DA-ROM\\
\hline
$E_{kin}^{max}$ & 4.11e-02 & 1.59e-02 & 1.20e-04 & 1.50e-04\\
\hline
$c_{D}^{max}$ & 4.93e-01 & 2.15e-01 & 9.00e-04 & 1.26e-03\\
\hline
$c_{L}^{max}$ & 1.55e-01 & 2.83e-02 & 5.29e-03 & 5.08e-03\\
\hline
$\ell^2({\bf L}^2)\, \uv$ norm & 3.62e-01 & 3.02e-01 & 6.79e-02 & 6.79e-02\\
\hline
\end{tabular}$$\caption{Example \ref{sec:Re=1000} (Case $Re=1000$): Errors levels with respect to DNS for G-ROM, grad-div-ROM ($\mu=0.001$), DA-ROM ($\beta=500$), and grad-div-DA-ROM ($\mu=0.001, \beta=500$) ($70$ snapshots used, which comprise $64\%$ of one full period from $t=5\,\rm{s}$ to $t=5.14\,\rm{s}$).}\label{tab:ErrLevCompInac11000}
\end{table}


\section{Conclusions}\label{sec:Concl}

In this paper, a new stabilized data assimilation reduced order method (grad-div-DA-ROM) for the numerical simulation of incompressible flows is proposed, analyzed and tested. The new grad-div-DA-ROM is a velocity nudging-based DA-ROM that incorporates a grad-div stabilization term.

The main contribution of the present paper is the numerical analysis of the fully discrete grad-div-DA-ROM applied to the unsteady incompressible NSE, where a rigorous error estimate is proved. This estimate takes into account the three sources of error: the spatial discretization error (due to the FE discretization), the temporal discretization error (due to the backward Euler method), and the POD truncation error.

With respect to a related approach \cite{zerfas_et_al} that, in a similar way, proposed, analyzed and tested a nudging-based DA-ROM (without grad-div) for incompressible flows, here we have obtained error bounds with constants independent on inverse powers of the viscosity parameter. Also, no upper bounds in the nudging parameter of the data assimilation method are required. Thus, in this respect, the present study can be seen as an improvement of the numerical analysis performed in \cite{zerfas_et_al}. 

Numerical experiments show that, for large values of the nudging parameter and a small number of POD modes, the new grad-div-DA-ROM converges to the true solution exponentially fast, and similarly to the DA-ROM in \cite{zerfas_et_al}, despite its simple implementation, it greatly improves the overall accuracy of the standard Galerkin POD-ROM (G-ROM) up to low viscosities over predictive time intervals. In the numerical experiments it can also be observed that the incorporation of the grad-div stabilization term in the ROM framework (grad-div-ROM, without DA) guarantees a significant improvement over G-ROM only for low Reynolds number. The numerical results suggest that DA reduced order methods display low sensitivity with respect to increase the Reynolds number, which is extremely important in order to solve complex realistic flows with low viscosities, and also provide a numerical support to the performed theoretical analysis, in which error bounds with constants independent on inverse powers of the viscosity parameter are derived.

\bibliographystyle{abbrv}
\bibliography{references,Biblio_STAB-POD-ROM_NSE}
\end{document}